\documentclass{amsart}
\usepackage{amsmath, amsfonts, amssymb,amsbsy,eucal,mathrsfs}
\usepackage[all]{xy}
\usepackage{pstricks}
\usepackage{pst-plot}
\usepackage{graphicx}
\usepackage[utf8]{inputenc}
\usepackage[T1]{fontenc}
\usepackage{tikz}

\newtheorem{thm}{Theorem}[subsection] 
\newtheorem{lemma}[thm]{Lemma}
\newtheorem{prop}[thm]{Proposition}
\newtheorem{fact}[thm]{Fact}
\newtheorem{cor}[thm]{Corollary}

\theoremstyle{definition}
\newtheorem{remark}[thm]{Remark}

\newtheorem{example}[thm]{Example}
\newtheorem{defn}[thm]{Definition}

\renewenvironment{proof}{{\flushleft \it Proof.}}{\hfill $\square$ \vspace{2mm}}

\DeclareMathOperator{\Gr}{Gr}

\DeclareMathOperator{\IG}{IG}

\DeclareMathOperator{\GL}{GL}
\DeclareMathOperator{\PGL}{PGL}
\DeclareMathOperator{\SL}{SL}
\DeclareMathOperator{\Sp}{Sp}

\DeclareMathOperator{\codim}{codim}

\newcommand{\Z}{{\mathbb Z}}
\newcommand{\Q}{{\mathbb Q}}

\newcommand{\cO}{{\mathcal O}}

\newcommand{\id}{\text{id}}

\newcommand{\pic}[2]{\includegraphics[scale=#1]{#2}}

\newcommand{\ignore}[1]{}

\def \ker {{{\rm Ker}}}
\def \kk {{{\mathsf{k}}}}
\def \odelta {{{\mathring{\Delta}}}}

\def\pic{{\rm Pic}}

\def\PL{{\rm PL}}
\def\LL{{\rm L}}

\def \co {{{\mathcal O}}}
\def\N{{\mathbb N}}
\def \oB {{{\overline{B}}}}

\def \divi {{{\rm Div}}}
\def \spec {{{\rm Spec}}}

\def \cox {{{{\mathcal C}_Y(X)}}}

\def \cV {{{\mathcal{V}}}}

\def \X {{{\Lambda}}}
\def \a {{\alpha}}

\def \deltat {{{\widetilde{\delta}}}}
\def \omt {{{\widetilde{\omega}}}}

\def \U {{{\mathcal U}}}

\def\aut{{\rm Aut}}

 \def\CC{\mathbb C} 
\newcommand{\G}{\mathbb{G}}

\def \Xt {{\widetilde{X}}}
\def \xt {{\widetilde{x}}}

\newcommand{\p}{\mathbb{P}}

\renewcommand{\a}{{\alpha}}

\newcommand{\rk}{{\rm rk}}

\def \div {{{\textrm{div}}}}

\newcommand{\scal}[1]{\langle #1 \rangle}

\newcommand{\Ct}{\widetilde{C}}
\newcommand{\Dt}{\widetilde{D}}

\def \cX {{{\mathcal{X}}}}

\def \D {{\mathcal D}}
\def \cC {{\mathcal C}}

\DeclareMathOperator{\chara}{{char}}
\newcommand{\g}{\mathfrak g}
\newcommand{\gp}{\mathfrak p}

\newcommand{\gh}{{{{\mathfrak h}}}}

\newcommand{\gt}{\mathfrak t}

\newcommand{\gr}{\mathfrak r}

\begin{document}

\title{Sanya lectures : geometry \\
  of spherical varieties}

\date{March 29, 2017}

\author{Nicolas Perrin}
\address{Laboratoire de Math\'ematiques de Versailles, UVSQ, CNRS, Universit\'e Paris-Saclay, 78035 Versailles, France.} 
\email{nicolas.perrin@uvsq.fr}


\begin{abstract}
These are expanded notes from lectures on the geometry of spherical varieties given in Sanya. We review some aspects of the geometry of spherical varieties. We first describe the structure of $B$-orbits. Using the local structure theorems, we describe the Picard group and the group of Weyl divisors and give some necessary conditions for smoothness. We later on consider $B$-stable curves and describe in details the structure of the Chow group of curves as well as the pairing between curves and divisors. Building on these results we give an explicit $B$-stable canonical divisor on any spherical variety.
\end{abstract}

\maketitle

\vskip -0.6 cm

\markboth{N.~PERRIN}
{\uppercase{Sanya lectures on the geometry of spherical varieties}}

\setcounter{tocdepth}{1}
\tableofcontents

\vskip -1.8 cm

\


\section*{Introduction}

We consider algebraic varieties over an algebraically closed field $\kk$ and denote by $G$ a linear algebraic group over $\kk$. From the point of view of equivariant birational geometry of $G$-varieties -- varieties endowed with an action of a reductive group $G$ -- spherical varieties correspond to one of the easiest possible situation. Indeed, a normal $G$-variety $X$ is spherical if and only if $\kk(X)^B = \kk$ where $B$ is a Borel subgroup of $G$ (Theorem \ref{theo-char}). This point of view is the starting point of the Luna-Vust Theory of embeedings described in Jacopo Gandini's lectures \cite{jacopo} (see also \cite{LV} and \cite{knop}). For this reason, spherical varieties have especially nice geometric properties that we review in these lectures.

There are many well known examples of spherical varieties. Let us cite the most famous ones: rational projective homogeneous spaces, toric varieties and symmetric spaces. One of the goal of a geometric study of spherical varieties is to extend as much as possible the classical geometric results known for the above three classes of spherical varieties to the general case. In particular, one would like to answer the following questions:
\begin{itemize}
\item What can be said on $B$-orbit closures compared to the very well understood situation of Schubert varieties in rational projective homogeneous spaces?
\item Can we describe, as in the toric case, the Picard group via convex geometry?
\item What can be said on the Chow groups especially on the duality between curves and divisors ?
\item Can we compute a canonical divisor ?
\end{itemize}
We will partially answer the above questions, reviewing results of many authors on spherical varieties.

\subsection*{Convention, notation and prerequisites}

In Section \ref{B-orbits}, we assume $\chara(\kk) \geq 0$ and specify when the assumption $\chara(\kk) = 0$ is needed. From Section \ref{sec-lst} on, we assume $\chara(\kk) = 0$.

All the groups we shall consider will be linear algebraic groups. We denote by $\Gamma$ such a group and use $G$ for reductive groups. We will denote by $R(\Gamma)$ and $R_u(\Gamma)$ the radical and the unipotent radical of a group $\Gamma$. The group of characters of $\Gamma$ is denoted by $\mathcal{X}(\Gamma)$. We denote by $\G_a$ and $\G_m$ the additive and multiplicative one dimensional group respectively. If $\Gamma$ acts on a variety $X$ and $x \in X$, we write $\Gamma_x$ for the stabiliser of $x$ in $\Gamma$.

We will denote by $T$ a maximal torus of $G$ and $B$ a Borel subgroup containing $T$. We denote by $R$ the root system of $(G,T)$, by $R^+$, respectively $R^-$, the sets of positive, respectively negative roots. For $P$ a parabolic subgroup of $G$ containing $B$, we denote by $P^-$ the opposite subgroup with respect to $T$. 

We denote by $W$ the Weyl group of $G$ and by $W_P$ the Weyl group of a parabolic subgroup $P \supset B$. We denote by $U$ the maximal unipotent subgroup of $B$. We refer to \cite{borelb} for results on algebraic groups. We denote by ${\rm Lie}(G)$ the Lie algebra of $G$, we shall also use the gothic letter $\g$. We write $\g_\a$ for the root space associated to the root $\a$ and $U_\a$ for the one-dimensional unipotent subgroup of $G$ with ${\rm Lie}(U_\a) = \g_\a$. 

We will also assume some familiarities with basics on algebraic geometry and use \cite{hartshorne-book} as reference. A variety is an irreducible and reduced scheme of finite type over $\kk$.

We tried to keep these notes as independent as possible from the theory of embeddings considered in \cite{jacopo}. This was of course only possible up to a certain point. In particular, in Section \ref{B-orbits} and Section \ref{sec-lst} we do not use the colored cones and colored fans but assume at a few places some familiarities with toric varieties (see for example \cite{fulton-toric}). In Section \ref{sec-pic} however, the description of the Picard group relies on the colored fans and we refer to \cite[Section 7]{jacopo} for definitions. In the last two sections, we tried to avoid as much as possible the use of Luna-Vust theory of embeddings but had to use it as several places. We then gave references and tried to use \cite{jacopo} when this was possible.

\subsection*{Acknowledgements} I first thank Michel Brion and Baohua Fu for the kind invitation to give lectures on the geometry of spherical varieties in Sanya. I thank Jacopo Gandini for many helpful discussions on our lectures during the two weeks of the conference. I also thank all the participants especially Johannes Hofscheier and Dmitry Timashev for the many questions and discussions during and after the talks. This led to many improvements and expansions of the first version of these notes. Finally I thank the referee for his comments and corrections.

\section{$B$-orbits}
\label{B-orbits}

In this section, we define the complexity of a $G$-variety. Spherical varieties are $G$-varieties of vanishing complexity. We then prove that they have finitely many $B$-orbits and study the inclusion graph of $B$-orbits.

\subsection{Reminders on $\Gamma$-varieties}

In this subsection, we gather results and definitions that we shall use in these notes.

\begin{defn}
  Let $\Gamma$ be a linear algebraic group.

  1. A $\Gamma$-variety is a variety endowed with an (algebraic) action of $\Gamma$.

  2. A $\Gamma$-variety is homogeneous if is has a unique $\Gamma$-orbit.

  3. A $\Gamma$-variety is quasi-homogeneous if is has a dense $\Gamma$-orbit.
\end{defn}

\begin{example}
  1. The grassmannian variety $\Gr(p,n)$ of linear subspaces of dimension $p$ in $\kk^n$ is a homogeneous $\Gamma$-variety with $\Gamma = \GL_n(\kk)$.

2. The projective space $\p^n$ is a homogeneous $\GL_{n+1}(\kk)$-variety but also a quasi-homogeneous $(\G_m)^n$-variety where $(\G_m)^n$ is the subtorus of diagonal matrices with determinant $1$ in $\GL_{n+1}(\kk)$.
\end{example}

Recall the following results of Sumihiro \cite{sumihiro,sumihiro2} enabling in many cases to assume that a $\Gamma$-variety is quasi-projective. Let me mention the very recent preprint by Brion \cite{brion-sumihiro} where the second result below is generalised to $\Gamma$-varieties defined over any field and with $\Gamma$ algebraic by not necessarily linear.

\begin{thm}[Equivariant Chow-Lemma]
  \label{thm-chow-lemma}
  Let $\Gamma$ be a connected linear algebraic group and $X$ a $\Gamma$-variety. There exists a quasi-projective $\Gamma$-variety $\Xt$ and a birational $\Gamma$-equivariant projective surjective morphism $f : \Xt \to X$.
\end{thm}

\begin{thm}
\label{sumihiro} Let $\Gamma$ be a connected linear algebraic group and $X$ a $\Gamma$-variety. Let $Y$ be a $\Gamma$-orbit.

1. There exists a quasi-projective $\Gamma$-invariant open subset of
$X$ containing $Y$.

2. If $X$ is quasi-projective, there exists a finite dimensional $\Gamma$-module $V$ together with an equivariant embedding $X\to\p(V)$.
\end{thm}

In the last result, we cannot replace \emph{quasi-projective} by \emph{affine}. Indeed, in general there is no $\Gamma$-stable affine covering: take for $X$ any projective homogeneous variety, e.g. $X = \p^n$ and $\Gamma = \GL_{n+1}(\kk)$. This is however possible if we replace $\Gamma$-stability by $B$-stability with $B$ a Borel subgroup of $\Gamma$.

\begin{prop}
  \label{prop-lift}
Assume $\textrm{char}(\kk) = 0$. Let $X$ be a $G$-variety and $Y$ any $G$-stable closed
subvariety. Then there exists an open $B$-stable affine open subset
$X_0$ of $X$ such that
\begin{itemize}
\item[1.] $X_0 \cap Y \neq \emptyset$
\item[2.] the restriction map $\kk[X_0]^{(B)} \to \kk[X_0 \cap
  Y]^{(B)}$ is surjective.
\end{itemize}
\end{prop}

\begin{proof}
  Using Sumihiro's Theorem, we may assume that $X$ is equivariantly embedded in  $\p(V)$ where $V$ is a finite dimensional $G$-module. Let $\bar X$ and $\bar Y$ be the closures of $X$ and $Y$ in $\p(V)$, let $\partial X = \bar X \setminus X$ and let $\widehat{X}$, $\widehat{\partial X}$ and $\widehat{Y}$ the cones in $V$ over $X$, $\partial X$ and $Y$.
  
  Note that $\widehat{Y} \not\subset \widehat{\partial X}$ and choose a homogeneous $B$-eigenfunction $f' \in \kk[\widehat{Y}]^{(B)}$ vanishing on $\widehat{\partial X}$ but not on $\widehat{Y}$. Now, since $\chara(\kk) = 0$, representation theory tells us that the surjective map $\kk[\widehat{X}] \to \kk[\widehat{Y}]$ induces a surjective map $\kk[\widehat{X}]^{(B)} \to \kk[\widehat{Y}]^{(B)}$.
  
  Choose $f$ such that $f|_{\widehat{Y}} = f'$ and set $X_0 = D_X(f) = \{x \in X \ | \ f(x) \neq 0 \}$. Then $X_0 \subset X$ is affine $B$-stable and meets $Y$.

  Now let $\phi \in \kk[X_0 \cap Y]^{(B)}$ homogeneous. There exists $m \geq 0$ such that $\phi {f'}^m \in \kk[\widehat{Y}]^{(B)}$. By the above surjectivity, we get $\psi \in \kk[\widehat{X}]^{(B)}$ with $\psi|_{\widehat{Y}} = \phi{f'}^m$. Then $(\psi f^{-m}) \in \kk[X_0]^{(B)}$ with $(\psi f^{-m})|_{\widehat{Y}} = \phi$.
\end{proof}

\begin{remark}
  The assumption on the characteristic is not really restrictive. Indeed, in positive characteristic, we have a similar result but we need to use $p$-th powers to lift functions. More precisely, the second condition has to be replaced by
  \begin{itemize}
  \item[2'.] $\forall f \in \kk[X_0 \cap Y]^{(B)}$, $\exists N \in \N$ and $\exists f' \in \kk[X_0]^{(B)}$ with $f'\vert_{X_0 \cap Y} = f^{p^N}$.
  \end{itemize}
  The above proof works using the following (see \cite[Theorems 1.3 and 2.1]{Gro1}):

  \begin{thm}
    Let $X$ be an affine $G$-variety, $Y \subset X$ a closed $G$-stable subset and $f \in \kk[Y]^{(B)}$. Then there exists $N \in \N$ and $f' \in \kk[X]^{(B)}$ with $f'\vert_{Y} = f^{p^N}$.
  \end{thm}
\end{remark}

Let us finally recall a result of Rosenlicht \cite{rosentlicht}.

\begin{thm}
  Let $\Gamma$ be alinear algebraic group and $X$ a $\Gamma$-variety. There exists a non empty open $\Gamma$-stable subset $X_0$ and a morphism $\pi : X_0 \to X'_0$ such that
  \begin{itemize}
  \item[1.] The fibers of $\pi$ are the $\Gamma$-orbits in $X_0$.
  \item[2.] The morphism $\pi$ induces an isomorphism $\kk(X_0') \simeq \kk(X_0)^\Gamma = \kk(X)^\Gamma$.
  \end{itemize}
\end{thm}

\subsection{Complexity and rank of $G$-varieties}

Let $\Gamma$ be an algebraic group, $G$ is a reductive algebraic group and $T \subset B \subset G$ is a Borel subgroup containing a maximal torus $T$.

\begin{defn}
  Let $X$ be a $B$-variety, the weight lattice of $X$ is the subgroup $\Lambda(X) \subset \cX(B)$ of weights of $B$ occuring in $\kk(X)$. This is a free abelian group and its rank $\rk(X)$ is the rank of $X$.
\end{defn}

  \begin{prop}
    Let $Y \subset X$ an inclusion of $B$-varieties. Then $\rk(Y) \leq \rk(X)$.
  \end{prop}

  \begin{proof}
    Let $X_0$ be as in Proposition \ref{prop-lift}. Any weight of $\Lambda(X)$ is the difference of weights of $\kk[X_0]$. The same is true for $Y$. The results follows from the surjection $\kk[X_0] \to \kk[X_0 \cap Y]$.
  \end{proof}

  \begin{defn}
The complexity $c(X)$ of a $\Gamma$-variety $X$ is the minimal codimension
of a $\Gamma$-orbit: $c_\Gamma(X) = \min\{\codim(Y) \ | \ Y \subset X \textrm{ is
  a $\Gamma$-orbit} \}$. If the group $\Gamma$ is clear from the context, we write $c_\Gamma(X) = c(X)$.
\end{defn}

\begin{prop}
  Let $X$ be a $\Gamma$-variety, then $c_\Gamma(X) = \operatorname{Trdeg}(\kk(X)^\Gamma)$.
\end{prop}

\begin{proof}
  Let $\pi : X_0 \to X_0'$ be as in Rosenlicht's
  Theorem. Since the dimension of $\Gamma$-orbits is lower semi-continuous
  the maximal dimension of an orbit is the dimension of the fibers of
  $\pi$ thus $c_\Gamma(X) = \dim X_0' = \textrm{Trdeg}(\kk(X'_0)) = \textrm{Trdeg}(\kk(X_0)^\Gamma) = \textrm{Trdeg}(\kk(X)^\Gamma)$.
\end{proof}

\begin{prop}
Let $X$ be a $G$-variety, then $c_U(X) = c_B(X) + \rk(X)$. 
\end{prop}

\begin{proof}
Replacing $X$ wih an open $G$-stable subset and using Rosenlicht's Theorem, we have morphisms $X \to X' \to X''$ with fiber the $U$-orbits and the $T$-orbits so the fibers of the composition are the $B$-orbits. We have $\dim(X') = c_U(X)$ and $\dim(X'') = c_B(X)$. Now $T$ acts on $X'$ and the dimension of its orbits is the rank of its weight lattice in $\kk(X')$ which is also $\rk(X)$. So the fibers of the second map $X' \to X''$ are of dimension $\rk(X)$ and we get $c_U(X) = c_B(X) + \rk(X)$.
\end{proof}

\begin{example}
\label{exam-Mn}
Quasi-homogeneous $G$-varieties do not always admit finitely many
$G$-orbits. For example, consider $X = \p(M_2(\kk))$ the projective
space over the space of square matrices of size $2$. The group
$G = \SL_2(\kk)$ acts by left multiplication and has a dense orbit: the
locus where the determinant is not vanishing therefore $c_G(X) = 0$.
However, for $v \in \kk^2$ with $v \neq 0$, the variety $Y_v = \{ [M] \in X \ | \ v \in \ker M \}$ is stable under the $G$-action so that we must have infinitely many $G$-orbits. 

Let $B$ be the Borel subgroup of upper triangular matrices. The
$B$-complexity is $c_B(X) = 1$. Indeed, the open subset of rank $2$
matrices is covered by a 1-dimensional family of $2$-dimensional $B$-orbits:
$$X_{[v]} = \{ [M] \in X \ | \ [Mv] \textrm{ is $B$-stable} \},$$
with $v \in \kk^2$ and $[v] \in \p^1$ its class. We also have the equalities $c_U(X) = 2$ and $\rk(X) = 1$.
\end{example}

As the above example suggests, in order to get a nice behaviour with
respect to the number of orbits, one needs to consider the
$B$-complexity for $B$ a Borel subgroup of $G$.

\begin{defn}
A spherical variety is a normal $G$-variety with $c_B(X) = 0$.
\end{defn}

\begin{remark}
1. A $G$-variety is spherical iff it is normal with a dense $B$-orbit.

2. For a spherical variety, the rank is the minimal codimension of a $U$-orbit. 
\end{remark}

\begin{thm}
A spherical $G$-variety $X$ has finitely many $B$-orbits.
\end{thm}

This theorem is a direct consequence of the following proposition.

\begin{prop}
Let $X$ be a $G$-variety and $Y$ be a closed $B$-stable subvariety,
then $c_B(Y) \leq c_B(X)$.
\end{prop}

\begin{proof}
Write $c$ for $c_B$. Let $Y \subset X$ be  closed and $B$-stable. We prove $c(Y) \leq c(X)$.

We start with a $G$-orbit $Y$. Let $X_0$ be an open affine
$B$-stable subvariety of $X$ with $X_0 \cap Y \neq \emptyset$ and a
such that the restriction map
$$\kk[X_0]^{(B)} \to \kk[X_0 \cap Y]^{(B)}$$
is surjective. Such an open subset exists by Proposition
\ref{prop-lift} (in positive characteristic, replace in what follows $u'$ and $v'$ by $p$-th powers). Now let $f\in \kk(Y)^B$, we can write $f=u/v$ with
$u,v\in \kk[X_0\cap Y]^{(B)}$ having the same weight. There exist
$u',v'\in \mathsf{k}[X_0]^{(B)}$ such that $u'\vert_Y=u$ and
$v'\vert_Y=v$. We get $(u'/v')\vert_Y=f$. It follows that the
transcendence degree of $\mathsf{k}(X)^B$ is bigger than or equal to the
transcendence degree of $\mathsf{k}(Y)^B$ so $c(X) \geq c(Y)$.

\smallskip

Now let $Y$ be any closed $B$-stable subset. We prove that $c(Y) \leq
c(G Y)$ and conclude by the previous argument. Recall
that $G$ is generated by the minimal parabolic subgroups strictly
containing $B$. We therefore only need to prove that $c(Y) \leq c(P
Y)$ for any minimal parabolic subgroup $P$.

Consider the contracted product $P \times^B Y$ defined as the quotient
of $P \times Y$ by the action of $B$ defined by $b \cdot (p,y) =
(pb,b^{-1}y)$. The projection on the first factor induces a morphism $P
\times^B Y \to P/B$ which is $P$-equivariant, locally trivial for the
Zariski tolopogy (it is trivial on the open subset $(U^- \cap P)B/B$
where $U^-$ is the unipotent radical of $B^-$) with fiber isomorphic to
$Y$. In particular $\dim P \times^B Y = \dim Y + 1$. 

The map $\pi : P\times^BY\to PY, [p,y] \mapsto py$ is surjective. It
is also proper since it can be viewed as the restriction of the
projection $P \times^B X \to X$. In particular $PY$ is closed and $\dim
PY \leq \dim Y + 1$. Assume that $PY \neq Y$ (otherwise we trivially
get $c(Y) \leq c(PY)$). Then the map $\pi$ is generically finite so
$c(P \times^B Y) = c(PY)$.

Now let $p \in P \setminus B$, the orbit $BpB/B$ is dense in $P/B$ and
if we set $B_p = B \cap pBp^{-1}$, we have an isomorphism $BpB/B \simeq
B/B_p$. Consider the contracted product $B \times^{B_p} Y$. We have an
embedding $B \times^{B_P} Y \to P \times^B Y$ defined by $[b,y] \mapsto [bp,y]$. Its image is $p^{-1}(BpB/B)$ therefore $B$-invariant and open. In particular $c(P \times^B Y) \geq c(B \times^{B_p}
Y)$. But any $B$-orbit in $B \times^{B_p} Y$ is of the form $B
\times^{B_p} Z$ for $Z$ a $B_p$-orbit in $Y$. In particular $c(B
\times^{B_p} Y) = c_{B_p}(Y)$ where we write $c_{B_p}(Y)$ for the minimal
codimension of a $B_p$ orbit in $Y$. In particular $c_{B_p}(Y) \geq
c(Y)$ so that we get 
$$c(PY) \geq c(P \times^B Y) \geq c(B \times^{B_p} Y) = c_{B_p}(Y) \geq
c(Y).$$
This completes the proof.
\end{proof}

\begin{example}
\label{exam-Mn2}
The assumption that $X$ is a $G$-variety is important. Recall Example
\ref{exam-Mn} where $X = \p(M_2(\kk))$ is the projective space over
the vector space of $2 \times 2$ matrices. For $v \in \kk^2$ with $v \neq 0$, let $Z =
\overline{X_{[v]}}$ be the closure of the $B$-orbit 
$$X_{[v]} = \{ [M] \in X \ | \ [Mv] \textrm{ is $B$-stable} \}.$$
Then $Z$ contains the subvariety
$$Y = \left\{ \left[ \left( \begin{array}{cc} 
a & b \\
0 & 0 \\
  \end{array} \right) \right] \in X \ \big| \ a,b \in \kk \right\}$$
whose points are fixed by $B$. In particular, we have $c_B(Z) = 0$ and
$c_B(Y) = 1$ so $c_B(Y) > c_B(Z)$ even if $Y \subset Z$. 
\end{example}

As a consequence we obtain equivalent definitions of
spherical varieties.

\begin{thm}
\label{theo-char}
Let $X$ be a normal $G$-variety. The following are equivalent
\begin{itemize}
\item[1.] $X$ is spherical.
\item[2.] $X$ has finitely many $B$-orbits.
\item[3.] $\mathsf{k}(X)^B=\mathsf{k}$.
\end{itemize}
\end{thm}

\begin{proof}
  (1. $\Rightarrow$ 2.) Follows from the previous theorem. 

(2. $\Rightarrow$ 3.) Any function $f \in \kk(X)^B$ is constant on $B$-orbits. Since $X$ has finitely many $B$-orbits, there must be a dense orbit thus $f$ is constant on $X$.

(3. $\Rightarrow$ 1.) We know that $c_B(X) = \textrm{Trdeg}(\kk(X)^B) =
  \textrm{Trdeg}(\kk) = 0$.
\end{proof}

We also mention the following equivalent definition of spherical
varieties which is more intrisic since it does not use Borel subgroups
but we need instead to consider all birational models. We refer to
\cite[Theorem 2.1.2]{survey} for a proof and to \cite[Definition-Theorem 25.1]{timashev} or \cite[Section 2]{jacopo} for more equivalent definitions of spherical varieties. 

\begin{thm}
A normal $G$-variety is spherical if and only if all its
$G$-equivariant birational models have finitely many $G$-orbits.
\end{thm}


\subsection{Structure of $B$-orbits}

The set of $B$-orbits in a spherical variety beeing finite, it is a useful combinatorial invariant. 

\begin{defn}
The set of $B$-orbits in a $G$-variety $X$ will be denoted by $B(X)$. The set of $B$-orbit closures will be denoted by $\oB(X)$. 
\end{defn}

\begin{remark}
There is a bijection between $B(X)$ and $\oB(X)$ given by $Y \mapsto \overline{Y}$, the reverse inclusion is given by taking the dense $B$-orbit.
\end{remark}

\begin{defn}
The sets $B(X)$ and $\oB(X)$ are posets for the Bruhat order $\leq$ defined by $Y \leq Z$ if $Y \subset \overline{Z}$.
\end{defn}

\begin{example}
For rational homogeneous spaces, for example $G/B$, the above Bruhat order coincides with the classical Bruhat order on the Weyl group $W$ of $G$: the covering relations for this order are of the form $u \leq v$ if there exists a reflection $t$ with $v = tu$ and $\ell(v) > \ell(u)$.
\end{example}

One can define another order on $B(X)$ and $\oB(X)$: the weak Bruhat order.

\begin{defn}
  The weak Bruhat order is defined by its covering relations. For a pair $(P,Y)$ with $P$ a minimal parabolic subgroup containing $B$ and $Y \in \oB(X)$ such that $Y \subsetneq PY$, we say that $P$ raises $Y$ and write $Y\prec PY$. These are the covering relations of the weak Bruhat order.
\end{defn}

\begin{example}
  Again this generalises the weak Bruhat order on rational homogeneous spaces: the covering relations for this order are of the form $u \leq v$ if there exists a simple reflection $s$ with $v = su$ and $\ell(v) > \ell(u)$.
\end{example}

For a given covering relation $Y\prec PY$, consider the morphism $\pi:P\times^BY\to PY$. It is a proper morphism. Denote by $\cO$ and $\cO'$ the dense $B$-orbits in $Y$ and $Y'=PY$. In the next proposition, the symbol $\cO''$ stands for another $B$-orbit.

\begin{prop}
\label{prop-UNT}
If $X$ is spherical, then one of the following occurs:
\begin{itemize}
\item Type {\rm (U)}: $P\cO=\cO\cup \cO'$ and $\pi$ is birational.
\item Type {\rm (N)}: $P\cO=\cO\cup \cO'$ and $\pi$ has degree 2.
\item Type {\rm (T)}: $P\cO=\cO\cup \cO'\cup \cO''$ 
with $\dim \cO''=\dim \cO$ and $\pi$ is birational.
\end{itemize}
In type {\rm (U)}, we have $\rk(\cO') = \rk(\cO)$, while in type {\rm (N)} and {\rm (T)}, we have $\rk(\cO') = \rk(\cO)+1 = \rk(\cO'')+1$.
\end{prop}

\begin{proof}
  If $\cO$ is the dense $B$-orbit in $Y$ and for $x$ in $\cO$, write $P_x$ for the stabiliser of $x$ in $P$. The $B$-orbits in $P\cO$ are in bijection with the $P_x$-orbits in $P/B$ and the correspondence is given as follows:
  $$P_x [g] \leftrightarrow B g^{-1} x.$$
  Now let $H$ be the image of $P_x$ in $\aut(P/B) \simeq \PGL_2(\kk)$. Since $c(P\cO) = 0$ we have finitely many such orbits thus $H$ has a dense orbit in $P/B$ and $\dim H \geq 1$. Consider the following subgroups of $\PGL_2(\kk)$: $T_0$ is the maximal torus of diagonal matrices, $U_0$ is the maximal unipotent subgroup of upper triangular matrices with $1$ on the diagonal, $B_0 = T_0U_0$ and $N_0$ is the normaliser of $T_0$. It is an easy exercice to prove that any positive dimensional subgroup of $\PGL_2(\kk)$ is conjugated to either $\PGL_2(\kk)$ itself, $T_0$, $N_0$ or $ZU_0$ with $Z$ a subgroup of $T_0$.

  Note that the first case is not possible: since $P$ raises $Y$, $P\cO$ contains at least two $B$-orbits $\cO$ and $\cO'$ while $H \simeq \PGL_2(\kk)$ has a unique orbit on $P/B\simeq \p^1$. We are left with the last three cases which we call type (T), type (N) and type (U) respectively.

  In type (U), the group $U_0$ has two orbits on $P/B$ and $ZU_0 \subset B$, therefore we have only two orbits $\cO$ and $\cO'$. In type (N) as well, $H \simeq N_0$ has only two orbits therefore there are two orbits while in type (T) there are three $T_0$ orbits therefore three $B$-orbits in $P\cO$.

  Now we compute the degree of the map $\pi$ (note that $\pi$ has to be generically finite). For this it is enough to compute $P_y$ for a general $y \in \cO'$ \emph{i.e.} for $y = p_0x$ with $p_0 \in P$ general. One easily checks that the fiber $\pi^{-1}(y)$ over $y$ is given by $\{[p_0g,x] \in P \times^B Y \ | \ g \in P_x \}$ and this set is isomorphic to $p_0 \cdot (P_x/B_x)$. Furthermore the stabilisers are as follows: 
$$(P_x)_{[g]} = P_{[g]} \cap P_x = g(P_{[1]} \cap P_{g^{-1}x})g^{-1} =
  g(B \cap P_{g^{-1}x})g^{-1} = gB_{g^{-1}x}g^{-1}.$$ 
This implies that the closed $B$-orbit $B \cdot x$ in $P\cO$ is in
correspondence with the closed $P_x$-orbit $P_x/(P_x)_{[1]} = P_x/B_x$
in $P/B$. In particular, $P_x/B_x$ is a closed $P_x$-orbit in
$P/B$. In type (U) there is a unique such orbit having one element, in
type (T) there are two such orbits both with one element while in type
(N) there is a unique such orbit with two elements. 

We are left with the rank computation. We need to determine the part of $T$ contained in the stabiliser. The stabiliser of $x$ is $B \cap P_x$ while the stabiliser of a general element is $B \cap g^{-1}P_x g$. This stabiliser is conjugated to the stabiliser of the action of $P_x$ on $[g] \in P/B$ which equals $gBg^{-1} \cap P_x$. We therefore need to compute in each case the part of $T$ contained in the stabiliser. We get that in type (U), the stabilisers of the closed and open orbits contain the same part of $T$, while in type (N) and (T), the intersection of the stabiliser of the open orbit with $T$ is of dimension one less than the intersection of the stabiliser of the closed orbits with $T$. This concludes the
proof. 
\end{proof}
 
\begin{defn}
The graph $\Gamma(X)$ has $B(X)$ as set of vertices and has a
$P$-edge of type (U), (T) or (N) for each covering relation
$Y\preccurlyeq PY$ of this type.
\end{defn}

\begin{example}
For projective rational homogeneous spaces, the graph $\Gamma(X)$ is
connected, has only edges of type (U) and has only one maximal and one
minimal element. This graph is called the Bruhat graph or the Hasse
diagram.

More precisely, if $P \subset G$ is a parabolic subgroup. Denote $W,W_P$ the Weyl groups of $G$ and $P$ and let $W^P \subset W$ be the set of minimal length representatives of $W/W_P$. Then the Bruhat decomposition induces a bijection between $B(G/P)$ and $W^P$, there is an edge between $u$ and $v$ if and only if $v = su$ and $\ell(v) > \ell(u)$ with $s$ a simple reflection. All edges are of type (U).
\end{example}

\begin{example}
  \label{exam-sl2}
  Let $G = \SL_2(\kk)$. We describe here most of the spherical $\SL_2(\kk)$-varieties. It is easy to check that a spherical subgroup $H$ is conjugated to a subgroup of the following form: $\SL_2(\kk)$, $B_0$ a Borel subgroup of $G$, $T_0$ a maximal torus, $N_0 = N_G(T_0)$ or $ZU_0$ with $U_0$ a maximal unipotent subgroup and $Z$ a subgroup of $T_0$. We refer to \cite[Examples 3.4, 6.10, 8.19 and 8.20]{jacopo} for more details.
  \begin{itemize}
  \item[1.] If $H = B_0$, then $G/H \simeq \p^1$ is complete and this is the unique embedding. There is a unique $B_0$-fixed point raised with type (U) to $G/H$.
  \item[2.] If $H = T_0$, then $G/H$ has a unique non trivial embedding given by $\p^1 \times \p^1$. There are two $B_0$-stable not $G$-stable divisors $\p^1 \times \{0\}$ and $\{0\} \times \p^1$. Both are raised with type (T).
  \item[3.] If $H = N_0$, then $G/H$ has a unique non trivial embedding given by $\p^2$. There is a unique $B_0$-stable not $G$-stable divisor given by a line which is raised with type (N).
  \item[4.] If $H = ZU_0$ and $Z$ is finite of order $n$, then there is a unique toroidal complete embedding: the ruled rational surface $p : \p_{\p^1}(\cO_{\p^1} \oplus \cO_{\p^1}(n)) \to \p^1$. There are three $B_0$-stable not $G$-stable divisors: two sections of the map $p$ and a fiber of $p$. All of them are raised with type (U).    
  \end{itemize}
\end{example}

\begin{example}
  \label{ex-ind}
  Let $G = \GL_3(\kk)$ and $H$ be the subgroup of matrices of one of the following forms:
  $$\left(\begin{array}{ccc}
  \star & 0 & * \\
  0 & \star & * \\
  0 & 0 & \star \\
  \end{array}\right) \textrm{ or }
  \left(\begin{array}{ccc}
  0 & \star & * \\
  \star & 0 & * \\
  0 & 0 & \star \\
  \end{array}\right)$$
  where $*$ represents an arbitrary element in $\kk$ and $\star$ represents an arbitrary element in $\kk^\times$. Let $X = G/H$, it is a $G$-spherical variety of dimension $4$ and one can check that $\Gamma(X)$ has the following form:
  
  \centerline{
\begin{tikzpicture}
  \node (Y5) at (10,9)  {$X$};
  \node (Y4) at (11,8)  {$Y_4$};
  \node (Y3) at (9,8)  {$Y_3$};
  \node (Y1) at (9,7) {$Y_1$};
  \node (Y2) at (11,7)  {$Y_2$};
  \node (Y) at (10,6)  {$Y$};
  \draw[double] (Y) -- node {\tiny{2}} (Y2);
  \draw[double] (Y5) -- node {\tiny{1}} (Y3);
  \draw (Y1) -- node {\tiny{2}} (Y3);
    \draw (Y2) -- node {\tiny{1}} (Y4);
  \draw (Y4) -- node {\tiny{2}} (Y5);
  \draw (Y) -- node {\tiny{1}} (Y1);
  \node (0) at (2,8.2) {};
  \node (1) at (3,7.5)  {};
  \node (2) at (4,7.5) {};
  \node[anchor=west] (3) at (4,7.5)  {\small{$P_1$-edges of type (U)}};
\node (4) at (3,7)  {};
  \node (5) at (4,7) {};
  \node[anchor=west] (6) at (4,7)  {\small{$P_2$-edges of type (U)}};
  \node (7) at (3,6.5)  {};
  \node (8) at (4,6.5) {};
  \node[anchor=west] (9) at (4,6.5)  {\small{$P_2$-edges of type (N)}};
    \draw[double] (7) -- node {\tiny{2}} (8);
    \draw (4) -- node {\tiny{2}} (5);
    \draw (1) -- node {\tiny{1}} (2);
    \node[anchor=west] (10) at (4,8)  {\small{$P_1$-edges of type (N)}};
  \node (11) at (3,8)  {};
  \node (12) at (4,8) {};
    \draw[double] (11) -- node {\tiny{1}} (12);
  \end{tikzpicture}}

\centerline{Graph $\Gamma(X)$}

\vskip 0.2 cm

The above example is obtained via a more general process called parabolic induction (see Subsection \ref{subsection-par-ind}).
\end{example}

\begin{cor}
\label{cor-TUN}
Let $Y \in \oB(X)$ and $P_{\a_1},\cdots,P_{\a_\ell}$ be a sequence of minimal parabolic subgroups. For $i \in [0,\ell]$, set $Y_i = P_{\a_i}\cdots P_{\a_1}Y$. Assume that $P_{\a_i}$ raises $Y_{i-1}$ to $Y_i$ for all $i$. Let $\ell_N$ and $\ell_T$ be the number of edges of type {\rm (N)} and {\rm (T)}.

1. Then the product $w = s_{\a_\ell} \cdots s_{\a_1}$ is reduced (of length $\ell$).

2. We have $\rk(Y_\ell) - \rk(Y) = \ell_N + \ell_T$.

3. The map $\overline{BwB} \times^B Y \to Y_\ell$ is surjective generalically finite of degree $2^{\ell_N}$.
\end{cor}

\begin{proof}
  1. Note that $Bs_i Y_{i-1}$ is dense in $Y_i$, for all $i$. In particular $B s_{\a_\ell} B \cdots B s_{\a_1} Y$ is dense in $Y_\ell$. Since $\dim Y_\ell = \dim Y + \ell$ we must have $\dim(B s_{\a_\ell} B \cdots B s_{\a_1}B) = \ell$ proving the result.

  2. This directly follows from the fact that the only edges raising the rank (by one) are those of type (T) and (N).

  3. On the one hand, since the expression $w = s_{\a_\ell} \cdots s_{\a_1}$ is reduced, the morphism $P_{\a_\ell} \times^B \cdots \times^B P_{\a_1} \to B w B$ is birational, thus the map $\overline{BwB} \times^B Y \to X$ has image $Y_\ell$. On the other hand, the raising maps $P_{\a_i} \times^B (P_{\a_{i-1}} \cdots P_{\a_1}Y) = P_{\a_i} \times^B Y_{i-1} \to Y_i = P_{\a_{i}} \cdots P_{\a_1}Y$ are generically finite of degree $2$ if and only if the corresponding edge is of type (N) (and otherwise birational). Now the composition of these maps is $P_{\a_\ell} \times^B \cdots \times^B P_{\a_1} \times^B Y \to Y_\ell$ and has degree $2^{\ell_N}$. But this map generically factors through $\overline{BwB} \times^B Y \to Y_\ell$ proving the result.  
\end{proof}

\subsection{Normality criteria}
Some geometric properties of the orbit closures can be detected on $\Gamma(X)$.

\begin{prop}
\label{coro-non-normal} Assume that $P_1$ and $P_2$ raise $Y$ to
$Y_1$ and $Y_2$ with type {\rm (U)} or {\rm (T)} and {\rm (N)} respectively and that
$P_2$ raises $Y_1$ to $Y_3$ with type {\rm (U)} or {\rm (T)}, then $Y_3$ is not
normal along $Y_2$.

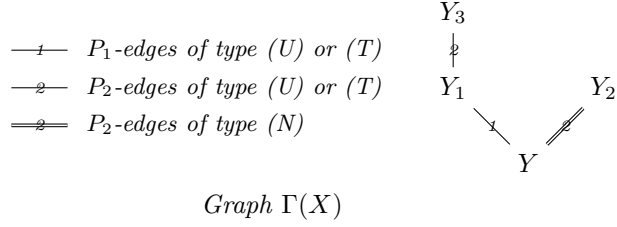
\centerline{\begin{tikzpicture}
   \node (Y3) at (9,8)  {$Y_3$};
  \node (Y1) at (9,7) {$Y_1$};
  \node (Y2) at (11,7)  {$Y_2$};
  \node (Y) at (10,6)  {$Y$};
    \draw[double] (Y) -- node {\tiny{2}} (Y2);
    \draw (Y1) -- node {\tiny{2}} (Y3);
    \draw (Y) -- node {\tiny{1}} (Y1);
  \node (0) at (2,8.2) {};
  \node (1) at (3,7.5)  {};
  \node (2) at (4,7.5) {};
  \node[anchor=west] (3) at (4,7.5)  {\small{$P_1$-edges of type (U) or (T)}};
\node (4) at (3,7)  {};
  \node (5) at (4,7) {};
  \node[anchor=west] (6) at (4,7)  {\small{$P_2$-edges of type (U) or (T)}};
  \node (7) at (3,6.5)  {};
  \node (8) at (4,6.5) {};
  \node[anchor=west] (9) at (4,6.5)  {\small{$P_2$-edges of type (N)}};
    \draw[double] (7) -- node {\tiny{2}} (8);
    \draw (4) -- node {\tiny{2}} (5);
    \draw (1) -- node {\tiny{1}} (2);
\end{tikzpicture}}

\centerline{Graph $\Gamma(X)$}
\end{prop}

\begin{proof}
The morphism $P_2\times^BY_1\to Y_3$ is birational while its
restriction $P_2\times^BY\to Y_2$ has non connected fibres. Zariski
Main Theorem gives the conclusion.
\end{proof}

\begin{example}
  \label{exam-non-normal}
  Recall Example \ref{ex-ind}. By the above result the orbit closure $Y_3$ is not normal and singular along $Y_2$.
\end{example}

As Proposition \ref{coro-non-normal} shows, in some cases, the
existence of edges of type (N) is the graph $\Gamma(X)$ prevents some
orbit closures from being normal. If there is no edge of type (N),
Brion in \cite{brion-multi-free} proved the following general result.

\begin{defn}
A $B$-orbit $\cO \in B(X)$ is called multiplicity-free if there is no
edge of type (N) in the full subgraph of $\Gamma(X)$ with vertices
the elements $\cO'\succcurlyeq \cO$. In this case, the orbit closure $Y = \overline{\cO}$ is also called multiplicity free.
\end{defn}

\begin{thm}
\label{thm-normal} Let $Y\in \oB(X)$ be multiplicity-free and assume
that for $Y'\succcurlyeq Y$, the variety $Y'$ contains a $G$-orbit
if and only if $Y'=GY$. If $GY$ is normal, Cohen-Macaulay or has a
rational resolution, then so does $Y$.
\end{thm}

\begin{remark}
If $\chara{\kk} = 0$, the variety $GY$ is always normal,
Cohen-Macaulay and has a rational resolution (see \cite[Corollary 2.3.4]{survey}).
\end{remark}

\begin{example}
Let $X = G/P$ with $P$ a parabolic subgroup of $G$. The $B$-orbits in $X$ are the Schubert varieties $BwP/P$ for $w \in W/W_P$ and one easily checks that the graph $\Gamma(X)$ has only edges of type (U). In particular, since $G/P$ is smooth and since Schubert varieties contain no $G$-orbit, we get that Schubert varieties are normal with rational singularities.
\end{example}


\subsection{Parabolic induction}
\label{subsection-par-ind}

In this subsection we explain a general process to deduce the graph $\Gamma(X)$ from the graph of a smaller spherical variety $Z$.

Let $P$ be a parabolic subgroup of $G$ and $L$ be its Levi subgroup containing $T$. Let $Z$ be a spherical $L$-variety and define $X = G \times^P Z$ where $P$ acts on $Z$ via its projection onto $L= P/R_u(P)$. Denote by $W,W_P$ the Weyl groups of $G$ and $P$ and by $W^P \subset W$ the set of minimal length representatives of $W/W^P$. Recall that the graph $\Gamma(G/P)$ has $W^P$ as vertices and only edges of type (U).

\begin{thm}
  Let $X$ be as above.

  1. The variety $X$ is $G$-spherical.

  2. The map $\oB(Z) \times W^P \to \oB(X)$ defined by $(Y,w) \mapsto \overline{Bw \cdot Y}$ is bijective.

  3. Let $\a \in S$ and $\beta = w^{-1}(\a)$. Then $P_\a$ raises $(w,Y)$ if and only if the following alternative occurs:
    \begin{itemize}
  \item[(a)] $\beta$ is positive but not a simple root of $P$.
  \item[(b)] $\beta$ or $-\beta$ is a simple root of $P$ and $(P_\beta \cap L)$ raises $Y$ to $Y'$.
      \end{itemize}
  In case {\rm (a)}, we have an edge between $(w,Y)$ and $(s_\a w,Y)$ of type {\rm (U)}. In case {\rm (b)}, we have an edge between $(w,Y)$ and $(w,Y')$ of the same type as the one between $Y$ and $Y'$.
\end{thm}

\begin{proof}
  Let $B \subset P$ be a Borel subgroup. We have $B = R_u(P)(B \cap L)$ and $B \cap L$ is a Borel subgroup of $L$.
  Recall the characterisation: $w \in W^P \Leftrightarrow w(B \cap L) \subset Bw$.
  
  Let $w \in W^P$ and $z \in Z$. Let $[w,z] \in X$ be the class of $(w,z)$.
    By the above consideration, we have $B \cdot [w,z] = \{ [bw,b'\cdot z] \ | \ b \in B \textrm{ and } b' \in B \cap L\} = B \cdot [w,z']$ for all $z' \in (B \cap L) \cdot z$. 
    In particular this $B$-orbit only depends on the $(B \cap L)$-orbit of $z$. By the Bruhat decomposition $G = \coprod_{w \in W^P} BwP$, any $B$-orbit in $X$ is of the form $B \cdot [w,z]$. This proves statements 1. and 2. 

  3. First note that if $P_\a$ raises $B \cdot [w,z]$ to $B \cdot [w',z']$, then $BwP/P$ and $Bw'P/P$ are contained in $P_\a w P/P$ and the latter is dense in $P_\a w P/P$. In particular $w' \in \{ w , s_\a w \}$ and $w' \geq w$.
  
  Assume that $BwP/P$ is not dense in $P_\a wP/P$. this is equivalent to $w < w' = s_\a w \in W^P$. This is then equivalent to $w^{-1}(\a)$ being positive but not being a simple root of $P$. Applying the above we get $P_\a \cdot [w,z] = \{ [pw,b'\cdot z] \ | \ p \in P_\a \textrm{ and } b' \in B \cap L\}$. In particular the fibers of the morphisms $P_\a \cdot [w,z] \to P_\a wP/P$ and $B \cdot [w,z] \to BwP/P$
  are isomorphic to $(B \cap L) \cdot z$. We thus have a cartesian diagram
  $$\xymatrix{P_\a \times^B B \cdot [w,z] \ar[r] \ar[d] & P_\a \cdot [w,z] \ar[d] \\
    P_\a \times^B BwP/P \ar[r] & P_\a w P/P.}$$
  In particular in case (a), the edge has the type of an edge in $G/P$ and is therefore of type (U).

  Assume now that $BwP/P$ is dense in $P_\a wP/P$, thus $w' = w$. Note that if $s_\a w \in W^P$, then we are in the previous situation (with $w$ and $s_\a w$ exchanged) and $P_\a$ is not raising $B \cdot [w,z]$. We may therefore assume $s_\a w \not \in W^P$ which is equivalent to $w^{-1}(\a)$ or $-w^{-1}(\a)$ being a simple root of $P$. Set $\beta = w^{-1}(\a)$ and assume $\beta > 0$, the other situation is symmetric. Then $BwB$ is contained in the closure of $Bs_\a B w B$. Let $L_\beta$ be the minimal parabolic of $L$ associated to $\beta$, and set $Y = \overline{(B \cap L) \cdot z}$. Identifying $Z$ with $P \times^P Z$, we have
  $$P_\a \cdot [w,z] = \overline{B s_\a B w Y} = \overline{B s_\a w Y}  = \overline{B w s_\beta Y} = \overline{B w B s_\beta Y} = \overline{B w L_\beta Y}.$$
  In particular, $P_\a$ raises $(w,Y)$ if and only if $L_\beta$ raises $Y$ say to $Y'$. Now we are left to prove that both edges are in of the same type.

  We give two proofs. First a \emph{direct proof}: one easily checks the isomorphisms $BwL_\beta \times^{(B\cap L)} Y \simeq P_\a w (B \cap L) \times^{(B \cap L)} Y \simeq P_\a \times^B BwY$, the last one given by $[pw,y] \leftrightarrow [p,wy]$. Therefore the map $\pi : BwL_\beta \times^{(B \cap L)} P_\a \times^B BwY \to BwL_\beta$ is obtained by multiplication with $Bw$ from $L_\beta \times^{(B \cap L)} Y \to L_\beta Y$ proving the result.

  Otherwise, remark that there are two chains of raising from $Y$ to $BwY'$. The first chain raises $Y$ to $BwY$ with only edges of type (U) by the above and a last edge raising to $BwY'$ of the type we want to find. The second chain starts with the raising from $Y$ to $Y'$ of a given type (U), (N) or (T) and a chain of edges raising from $Y'$ to $BwY'$ all of type (U) again. If the edge between $Y$ and $Y'$ is of type (U) the rank of $BwY'$ must be the same as the rank of $Y$ and the edge between $BwY$ and $BwY'$ must also be of type (U). If not, then the edge is of type (T) or (N). But the degree of the map $Bs_\a wB \times^B Y \to BwY'$ is $2^a$ where $a$ is the number of edges of type (N) on any path from $Y$ to $BwY'$ (Corollary \ref{cor-TUN}). Therefore, the two unknown edges are either both of type (N) or both of type (T).
  \end{proof}

\begin{example}
  \label{ex-ind-2}
Recall the notation of Example \ref{ex-ind}. Consider the group $G = \GL_3(\kk)$ and the subgroup $H$ of matrices of one of the following forms:
  $$\left(\begin{array}{ccc}
  \star & 0 & * \\
  0 & \star & * \\
  0 & 0 & \star \\
  \end{array}\right) \textrm{ or }
  \left(\begin{array}{ccc}
  0 & \star & * \\
  \star & 0 & * \\
  0 & 0 & \star \\
  \end{array}\right)$$
    and let $P$ be the parabolic subgroup of matrices of the form
    $$\left(\begin{array}{ccc}
  a & b & * \\
  c & d & * \\
  0 & 0 & \star \\
  \end{array}\right),$$
  here $* \in \kk$, $\star \in \kk^\times$ and $ad - bc \in \kk^\times$. The quotient $P/H$ is isomorphic to $Z = \SL_2(\kk)/N$ where $N$ is the normaliser of the maximal torus in $\SL_2(\kk)$. The variety $X = G/H$ is obtained by parabolic induction from $Z$: we have $X  \simeq G \times^P P/H \simeq G \times^P Z$. Its graph $\Gamma(X)$ can therefore be obtained from the graphs of $G/P \simeq \p^2$ and $Z = \SL_2(\kk)/N$. These later graphs are as follows:

    \vskip 0.1 cm
  
  \centerline{
    \begin{tikzpicture}
  \node (A) at (-1,8) {$\bullet$};
  \node (B) at (-2,8) {$\bullet$};
  \node (AB) at (-1.5,7.4) {Graph $\Gamma(Z)$};
  \draw[double] (A) -- node {\tiny{1}} (B);
  \node (C) at (1,8) {$\bullet$};
  \node (D) at (2,8) {$\bullet$};
  \node (E) at (3,8) {$\bullet$};
    \node (AB) at (2,7.4) {Graph $\Gamma(G/P)$};
  \draw (C) -- node {\tiny{2}} (D);
  \draw (D) -- node {\tiny{1}} (E);
  \node[anchor=west] (3) at (6,7.5)  {\small{edges of type (U)}};
  \node (1) at (5,7.5)  {};
  \node (2) at (6,7.5) {};
  \draw (1) -- node {} (2);
  \node[anchor=west] (10) at (6,8)  {\small{edges of type (N)}};
  \node (11) at (5,8)  {};
  \node (12) at (6,8) {};
  \draw[double] (11) -- node {} (12);
  \end{tikzpicture}}

  \vskip 0.1 cm

It is now easy to deduce the graph $\Gamma(X)$ as claimed in Example \ref{ex-ind}:
  
  \centerline{
\begin{tikzpicture}
  \node (Y5) at (10,9)  {$X$};
  \node (Y4) at (11,8)  {$Y_4$};
  \node (Y3) at (9,8)  {$Y_3$};
  \node (Y1) at (9,7) {$Y_1$};
  \node (Y2) at (11,7)  {$Y_2$};
  \node (Y) at (10,6)  {$Y$};
  \draw[double] (Y) -- node {\tiny{1}} (Y2);
  \draw[double] (Y5) -- node {\tiny{2}} (Y3);
  \draw (Y1) -- node {\tiny{1}} (Y3);
    \draw (Y2) -- node {\tiny{2}} (Y4);
  \draw (Y4) -- node {\tiny{1}} (Y5);
  \draw (Y) -- node {\tiny{2}} (Y1);
  \node (0) at (2,8.2) {};
  \node (1) at (3,7.5)  {};
  \node (2) at (4,7.5) {};
  \node[anchor=west] (3) at (4,7.5)  {\small{$P_1$-edges of type (U)}};
\node (4) at (3,7)  {};
  \node (5) at (4,7) {};
  \node[anchor=west] (6) at (4,7)  {\small{$P_2$-edges of type (U)}};
  \node (7) at (3,6.5)  {};
  \node (8) at (4,6.5) {};
  \node[anchor=west] (9) at (4,6.5)  {\small{$P_2$-edges of type (N)}};
    \draw[double] (7) -- node {\tiny{2}} (8);
    \draw (4) -- node {\tiny{2}} (5);
    \draw (1) -- node {\tiny{1}} (2);
    \node[anchor=west] (10) at (4,8)  {\small{$P_1$-edges of type (N)}};
  \node (11) at (3,8)  {};
  \node (12) at (4,8) {};
    \draw[double] (11) -- node {\tiny{1}} (12);
  \end{tikzpicture}}

\centerline{Graph $\Gamma(X)$}

\end{example}

\subsection{Complements}

There are many more results on the structure of $B$-orbits. In particular, Richardson and Springer give in \cite{RS1,RS2} a combinatorial description of the graph $\Gamma(X)$ for $X$ a symmetric. There are very few other descriptions of these graphs, see for example \cite{ressayre-minimal} and \cite{GP}. See also \cite{piotr} and \cite{bender} for results on the normality of $B$-orbit closures.

We describe here more structure on the set $B(X)$, namely, the Weyl group $W$ acts on $B(X)$. For $\a$ a simple root and $\cO \in B(X)$, let $s_\a$ be the corresponding simple reflection and $P= P_\a$ the corresponding minimal parabolic subgroup. Define an action of $s_\a$ as follows:
\begin{itemize}
\item If $P \cdot \cO = \cO$, then $s_\a \star \cO = \cO$;
\item If $P$ raises $\cO$ with type (U), then $P \cdot \cO = \cO \cup \cO'$. We set $s_\a \cdot \cO = \cO'$ and $s_\a \star \cO' = \cO$;
\item If $P$ raises $\cO$ with type (T), then $P \cdot \cO = \cO \cup \cO' \cup \cO''$ with $\cO'$ dense in $P \cdot \cO$. We set $s_\a \cdot \cO = \cO''$, $s_\a \star \cO'' = \cO$ and $s_\a \star \cO' = \cO'$;
\item If $P$ raises $\cO$ with type (N), then $P \cdot \cO = \cO \cup \cO'$. We set $s_\a \star \cO = \cO$ and $s_\a \star \cO' = \cO'$.
\end{itemize}
Knop \cite{K2} proves that this extends to an action of the full Weyl group $W$ of $G$.

\begin{thm}
The above action of $s_\a$ extends to an action of $W$ on $B(X)$.
\end{thm}

We finish with the following result.

\begin{prop}
Spherical varieties are rational.
\end{prop}

\begin{proof}
Let $G/H$ be a spherical homogeneous space and let $B$ be a Borel
subgroup such that $B/B\cap H$ is dense in $G/H$. It is enough to
prove that a quotient $B/K$ is 
rational for any connected solvable group $B$ and any subgroup
scheme $K$. We proceed by induction on $(\dim B,\dim B/K)$ with the
lexicographical order. Pick $Z$ a connected one dimensional normal
subgroup of $B$. Such a group always exists. Indeed, if the unipotent
radical of $B$ is trivial then $B$ is a torus and pick for $Z$ any one
dimensional subgroup. Otherwise, pick for $Z$ a one dimensional
subgroup of the center of the unipotent radical $U$ of $B$ (this
center is non trivial and such a subgroup exists for example by
\cite[Theorem 10.6(2)]{borelb}). If $Z\subset K$ then quotienting by
$Z$ we conclude by induction on $\dim B$. Otherwise, we have a
fibration $B/K\to B/KZ$ obtained after quotienting by the action of
$Z$ via left multiplication (recall that $Z$ is normal and that $KZ$
is the subgroup scheme generated by $K$ and $Z$). The fiber of this fibration
is $Z/(Z\cap K)$ which is a connected group of dimension 1. Since any
connected solvable group is special (see \cite[Proposition 14]{serre})
the fibration $B/K\to B/KZ$ is locally trivial and since $B/KZ$ is rational
(by induction on $\dim B/K$) the result follows.
\end{proof}

\section{Local structure theorems}
\label{sec-lst}

From now on, we assume $\chara(\kk) = 0$.
We start with a structure theorem for $G$-modules and deduce a structure theorem for $G$-varieties. Later on we apply these results to spherical varieties.

\subsection{Local structure for $G$-varieties}
\label{section-struct}

Let $V$ be a $G$-module and let $Y$ be a closed orbit of $\p(V)$. The stabiliser of any point in $Y$ is a parabolic subgroup of $G$ because $Y$ is projective. Furthermore, there exists an element $y\in Y$ such that $By$ is open and dense in $Y$. Recall the following facts.

\begin{fact}
  Let $v\in V$ such that $[v]=y$.
  
1. There exists a $B$-eigenvector $\eta\in (V^\vee)^{(B)}$ such that
  $\scal{\eta,v}=1$. 

2. $G_y$ and $G_\eta$ are opposite parabolic subgroups and we have $B \cdot y= G_\eta \cdot y$.
\end{fact}

Set $P = G_\eta$ and $L = P \cap G_y$.
The group $L$ is reductive and is a maximal reductive subgroup of both $P$ and $G_y$. Denote by $R_u(P)$ the unipotent radical of $P$, we have $P=LR_u(P)$. 
The subset $\p(V)_\eta$ where $\eta$ does not vanish is open $P$-stable  and contains $P \cdot y$.

\begin{prop}
\label{theo-proj}
  There exists a closed $L$-subvariety $S$ of $\p(V)_\eta$ containing $y$ such that the morphism
$$R_u(P)\times S\to\p(V)_\eta$$
defined by $(p,x)\mapsto px$ is a $P$-equivariant isomorphism.
\end{prop}

\begin{proof}
  We first reduce to the case of simple modules.
  Denote by $\scal{G \cdot v}$ and $\scal{G \cdot \eta}$ the $G$-submodules of $V$ and $V^\vee$ spanned by $v$ and $\eta$. Note that $\scal{G \cdot v}$ is simple while $\scal{G \cdot \eta}$ is isomorphic to its dual.
    The orthogonal $\scal{G \cdot \eta}^\perp$ is therefore of codimension $\dim \scal{G \cdot v}$ in $V$ and in direct sum with $\scal{G \cdot v}$. We thus get a decomposition
$V=\scal{G \cdot v}\oplus\scal{G \cdot \eta}^\perp.$
  The projection $p$ from $\scal{G \cdot \eta}^\perp$ onto $\p\scal{G \cdot v}$ defines a rational $G$-equivariant morphism 
$p:\p(V)_\eta\to\p\scal{G \cdot v}.$
This morphism restricts to the identity on $Y$ since $Y \subset \p\scal{G \cdot v}$. If the statement is true for $\scal{G \cdot v}$, then there exists $S$ as above and we get the Cartesian diagram
$$\xymatrix{R_u(P)\times p^{-1}(S)\ar[d]_p\ar[r] & \p(V)_\eta\ar[d]\\
R_u(P)\times S\ar[r] & \p\scal{G \cdot v}.\\}$$
Since the bottom horizontal arrow is an isomorphism, the same is true for the top horizontal arrow and the result follows.
 
We are left to prove the result for $V$ simple. Let $T_v=T_v(G \cdot v)$ be the tangent space of $G \cdot v$ at $v$. We consider $T_v$ as a vector subspace of $V$. Since $v$ is a $G_y$-eigenvector, the group $G_y$ acts on $T_v$. Since the weight of $v$ as an eigenvector is non trivial (otherwise $V$ would be trivial), the space $T_v$ contains the line $\kk v$.

The space $T_v$ is thus a sub-$L$-representation of $V$ and since $L$ is reductive there is a decomposition $$V=T_v\oplus E$$
with $E$ a representation of $L$. Define $S=\p(\kk v \oplus E)_\eta$. This is a closed subvariety of $\p(V)_\eta$ which is stable under $L$ and contains $y$. Note that $S$ is isomorphic to the affine space $y+E$ and that $S$ meets $Y$ tranversaly in $y$: indeed, the tangent spaces of $Y$ and $S$ at $y$ are $T_v/\kk v$ and $E$ which are supplementary in $V/\kk v$.
Also note that the variety $\p(V)_\eta$ has a unique closed $T$-orbit: the fixed point $y$, since 
$V$ is simple with lowest weight $\lambda_v$.

Consider $R_u(P)\times S$ as a $T$-variety via the action $t\cdot(p,z)=(tpt^{-1},t\cdot z)$. We also have that 
$R_u(P)\times S$ has a unique closed $T$-orbit: the fixed point $(e,y)$, since it has lowest weight for the $T$-action.

Consider the multiplication morphism $m:R_u(P)\times S\to \p(V)_\eta$. We want to prove that this morphism is an isomorphism. We first claim, that 
the differential $d_{(e,y)}m$ is injective.

Let $\gr_u(P)$ be the Lie algebra of $R_u(P)$. We know that $Py=R_u(P)Ly=R_u(P)y$ is open in $Y=G \cdot y$ therefore the morphism $R_u(P)\times \kk v\to Y$ is dominant and its tangent map $\gr_u(P)\times \kk v\to T_yY$ is surjective. 
We get the equality $T_yY=\gr_u(P)v/\kk v$. The same argument gives $T_v(G \cdot v)=\gr_u(P)v+ \kk v$. 

Furthermore, since $\eta$ is fixed by $R_u(P)$, we have $\scal{\eta,pv}=\scal{p^{-1}\eta,v}=\scal{\eta,v}=1$ for all $p\in R_u(P)$ and therefore $\eta$ is constant on $R_u(P)y$. This implies by derivation that $\eta$ vanishes on $\gr_u(P)v$. In particular $\gr_u(P)v$ and $\kk v$ are complement and  $T_v(G \cdot v)=\gr_u(P)v\oplus \kk v$. This also implies the equality $T_vV = \kk v \oplus \gr_u(P)v \oplus E$.

These equalities lead to the identifications of $S$ and $\p(V)_\eta$
with the affine spaces $v+E$ and $v+(\gr_u(P)v\oplus E)$. The morphism
$m$ is given by $m((p,(v+x))=p\cdot(v+x)$. We may now compute the
differential: $d_{(e,v)}m(\xi,x)=v+\xi\cdot v+x$ for $\xi\in\gr_u(P)$ and
$x\in E$. Indeed, the first two terms come from the differentiation of
the action of $R_u(P)$ on $v$ while the second term comes from the
differential of the action on $E$ which is linear.

We are left to prove that the map $\gr_u(P)\to\gr_u(P)v$ given by the
action on $v$ is injective. This is true since the intersection of
$R_u(P)$ with the stabiliser $G_y$ of $y$ is trivial thus $R_u(P)$
acts freely on $y$ and $v$ thus by differentiation the same is true on
the Lie algebra level.

Let $Z$ be the locus in $R_u(P)\times S$ where the differential of $m$
is not surjective. This is a closed subset of $R_u(P)\times S$. If $Z$
is non empty, then it contains a closed $T$-orbit which has to be
$(e,y)$, a contradiction. The morphism $m:R_u(P)\times S\to
\p(V)_\eta$ is therefore open. 

Let $Z$ be the complement of the image, then $Z$ is closed and
$T$-stable. If it is non empty, then it contains a closed $T$-orbit
which has to be $y$, a contradiction. Thus $m$ is surjective. 

Thus $m$ is a covering but since both varieties are affine spaces
which are simply connected, the map $m$ is an isomorphism.
\end{proof}

\begin{example}
  Consider $V = M_n(\kk)$ and $G = \GL_n(\kk) \times \GL_n(\kk)$ acting via $(P,Q)\cdot M = PMQ^{-1}$. Consider $Y$ the set of rank $1$ matrices. It is the only closed subvariety stable by $G$ and let $y = [M] \in Y$ with
  $$M = \left(\begin{array}{cccc}
    1 & 0 & \cdots & 0 \\
    0 & 0 & \ddots & \vdots \\
    \vdots & \ddots & \ddots & 0 \\
    0 & \cdots & 0 & 0 \\
  \end{array}
  \right)$$
  and $\eta$ the linear form defined by $\eta(a_{i,j}) = a_{1,1}$. If $P$ is the stabiliser of $\eta$, we have
  $$P = \left\{\left(\left(\begin{array}{cc}
    A & 0 \\
    C & D \\
  \end{array}\right),\left(\begin{array}{cc}
    A' & B' \\
    0 & D' \\
  \end{array}\right)\right) \ \big| \ A,A' \in M_1(\kk), D,D' \in M_{n-1}(\kk) \right\}$$
  $$G_y = \left\{\left(\left(\begin{array}{cc}
    A & B \\
    0 & D \\
  \end{array}\right),\left(\begin{array}{cc}
    A' & 0 \\
    C' & D' \\
  \end{array}\right)\right) \ \big| \ A,A' \in M_1(\kk), D,D' \in M_{n-1}(\kk) \right\}.$$
  Setting
  $$S = \left\{\left(\begin{array}{cc}
    1 & 0 \\
    0 & M \\
  \end{array}
  \right) \ \big| \ M \in M_{n-1}(\kk) \right\} \simeq M_{n-1}(\kk),$$
  we have the isomorphism $R_u(P) \times S \to \p(V)_\eta$ given by the action
  $$\left(\left(\begin{array}{cc}
    1 & 0 \\
    C & 1 \\
  \end{array}\right),\left(\begin{array}{cc}
    1 & B' \\
    0 & 1 \\
  \end{array}\right),\left(\begin{array}{cc}
    1 & 0 \\
    0 & M \\
  \end{array}\right)\right) \mapsto \left(\begin{array}{cc}
    1 & -B' \\
    C' & M - C'B' \\
  \end{array}\right).$$
\end{example}

\begin{remark}
  The above result enables to replace
  locally the study of quasi-projective $G$-varieties to quasi-affine
  $G$-varieties and of projective $G$-varieties to affine
  $G$-varieties.
\end{remark}

\subsection{Local structure for spherical varieties}

For a spherical variety $X$, we describe the local
structure not only along projective orbits but along any $G$-orbit $Y$. Recall that the set of $B$-stable prime divisors is finite. We define
$$\begin{array}{l}
  \D(X) = \{ D \subset X \textrm{ $B$-stable prime divisor} \}, \\
  \Delta(X) = \{ D \in \D(X) \ | \ D \textrm{ is not $G$-stable} \}, \\
  \D_Y(X) = \{ D \in \D(X) \ | \ Y \subset D \} \textrm{ and } \\
  \Delta_Y(X) = \{ D \in \Delta(X) \ | \ Y \subset D \} = \Delta(X) \cap \D_Y(X). \\
\end{array}$$
Finally denote the set of $B$-stable prime divisors containing no $G$-orbit by
$$\odelta(X) = \Delta(X) \setminus \bigcup_Y \Delta_Y(X) = \D(X) \setminus \bigcup_Y \D_Y(X)$$
where $Y$ runs in the set of $G$-orbits in $X$. Recall (for example from \cite[Section 5 and Theorem 5.2]{jacopo}) the definition of the $G$- and $B$-charts $X_{Y,G}$ and $X_{Y,B}$:
$$X_{Y,B}= 
X\setminus\bigcup_{D\in \D(X)\setminus \D_Y(X)}D 
\textrm{ and } X_{Y,G} = GX_{Y,B} = \{x\in X \ | \ 
\overline{Gx}\supset Y\}.$$
A $G$-spherical variety is called {\bf simple} if it has a unique closed $G$-orbit. Note that $X_{Y,G}$ is a simple spherical variety with unique closed orbit $Y$. Note also that $X_{Y,B}$ is the minimal B-stable affine open subset of X which intersects Y. Denote by $P$ the stabiliser of $X_{Y,B}$ \emph{i.e.} the set
of elements $g\in G$ with $g\cdot X_{Y,B}=X_{Y,B}$. This is a
parabolic subgroup containing $B$.

\begin{lemma}
\label{lemm-cartier}
Let $\mathcal{D}$ be the complement of $X_{Y,B}$ in $X_{Y,G}$. Then $\cO_{X_{Y,G}}(\mathcal{D})$ is Cartier and globally generated. The same is true for any of the sheaves associated to an irreducible component of $\mathcal{D}$ 
\end{lemma}

\begin{proof}
  Modulo replacing $G$ by a finite cover, we may assume that $\cO_{X_{G,Y}}(\mathcal{D})$ is $G$-linearised. In particular, the non Cartier locus and the locus where it is non globally generated are $G$-stable. If these loci are non empty, they must contain the only closed $G$-orbit: $Y$. But $\cO_{X_{G,Y}}(\mathcal{D})$ is locally free and globally generated outside $\mathcal{D}$ and therefore on an open subset of $Y$ proving the assertion. The same argument works for any irreducible component of $\mathcal{D}$.
\end{proof}

\begin{thm}
\label{thm-stru}
Keep notation as above.

1. There exists a Levi subgroup $L$ of $P$ and a closed subvariety $S$ of $X_{Y,B}$ with:
\begin{itemize}
\item[(a)] The variety $S$ is stable under $L$;
\item[(b)] The map $R_u(P)\times S\to X_{Y,B}$ defined by $(p,x)\mapsto
p\cdot x$ is a $P$-isomorphism.
\end{itemize}

2. The variety $S$ is affine $L$-spherical and
$S\cap Y$ is a $L$-orbit with isotropy subgroup $L_y$ containing $(L,L)$ the derived subgroup of $L$ for any $y \in S\cap Y$. The subgroup $L_y$ is independent of $y$. Denote it by $L_Y$.

3. There exists a closed $L_Y$-stable subvariety $S_Y \subset S$
containing an $L_Y$-fixed point such that the morphism 
$$L\times^{L_Y}S_Y\to S \textrm{ defined by } [l,y] \mapsto l\cdot y$$
is aa $L$-equivariant isomorphism. The variety $S_Y$ is affine $L_Y$-spherical
of rank $\rk(X)-\rk(Y)$. 
\end{thm}

\begin{proof}
  Since $X_{Y,B}$ is contained in $X_{Y,G}$ we may assume that $X = X_{Y,G}$ is
  simple.

1. We want to apply Proposition
  \ref{theo-proj}. Let $\mathcal{D}$ be the complement of $X_{Y,B}$ in
  $X$. We know that ${\mathcal{D}}$ is Cartier and globally generated
  therefore there exists  a canonical section $\eta$ of the line
  bundle $\co_X({\mathcal{D}})$ given by 
   $\co_X\stackrel{\eta}{\to}\co_X({\mathcal{D}})$. This section is an element of
   $H^0(X,\co_X({\mathcal{D}}))$. We may assume, replacing $G$ by a
   covering that ${\mathcal{D}}$ is $G$-linearised thus $G$ acts on
   $H^0(X,\co_X({\mathcal{D}}))$. Let $V^\vee$ be the $G$-submodule spanned
   by $\eta$. We have a $G$-equivariant morphism (this is indeed a
   morphism since ${\mathcal{D}}$ is globally generated):
$$\varphi:X\to \p(V)$$ 
defined by $x\mapsto[\sigma\mapsto\sigma(x)]$ for $\sigma\in
V^\vee\subset H^0(X,\co_X({\mathcal{D}}))$. By definition of $\eta$, we
have $X_{Y,B}\to\p(V)_\eta$. Note also that since $P$ stabilises
$X_{Y,B}$ it also stabilises ${\mathcal{D}}$ and thus $\eta$ is a
$P$-eigenfunction therefore $P$ also stabilises $[\eta]$ and
$\p(V)_\eta$. The map $X_{Y,B}\to \p(V)_\eta$ is therefore
$P$-equivariant. 

Since $V$ is simple, the intersection of all translates $gH_\eta$ of
the vanising divisor $H_\eta$ of $\eta$ is empty. Therefore if $Z$ is
a closed orbit in $\p(V)$, it will not be contained in $H_\eta$.
Choose $z$ in the dense $P$-orbit of $Z$, then $B^- \cdot z$ is dense in $Z$.
We may apply Proposition
\ref{theo-proj} to get a closed $L$-stable subvariety $S'$ of
$\p(V)_\eta$ such that the morphism $R_u(P)\times S'\to \p(V)_\eta$ is
a $P$-equivariant isomorphism. In particular $S'$ meets the image of
$X_{Y,B}$. Let $S=\varphi^{-1}(S')$. This is a closed $L$-stable
subvariety of $X_{Y,B}$ and we have a Cartesian diagram 
$$\xymatrix{R_u(P)\times S\ar[r]\ar[d]_{\id\times\varphi} & X_{Y,B}
  \ar[d]_{\varphi} \\
R_u(P)\times S'\ar[r] & \p(V)_\eta.}$$
This proves that the top map is an isomorphism.

2. We have finitely many $B$-orbits in $X_{Y,B}$ since $X$ is
spherical thus $B$ also has finitely many orbits in $R_u(P)\times
S$. Recall that $P=R_u(P)L$ thus $B=R_u(P)(L\cap B)$ and $L\cap
B$ is a Borel subgroup of $L$. Recall also that the action of
$P=R_u(P)L$ on $R_u(P)\times S$ is given by
$ul\cdot(u',x)=(ulu'l^{-1},l\cdot x)$ thus $B\cap L$ must have
finitely many orbits in $S$. Since $X_{Y,B}$ is normal as an open subset of $X$, the variety $S$ is also normal thus $S$
is $L$-spherical. It is an affine variety as closed subvariety of the affine variety $X_{Y,B}$ (see \cite[Theorem 5.2]{jacopo}). The isomorphism from part 1 of the theorem induces an isomorphism
$R_u(P)\times (S\cap Y)\to Y\cap X_{Y,B}$. The
right hand side is a $P$-orbit and a $B$-orbit, thus $S\cap
Y$ is an $L$-orbit and a $(B\cap L)$-orbit as well. 

Let $y\in S\cap Y$. We have $S\cap Y = L \cdot y = (B \cap L) \cdot y$ thus $(B\cap L)L_y = L$. 

\begin{lemma}
\label{lemm-derive}
  Let $H$ be a closed subgroup of a connected reductive group $G$ such
  that $G=HB$ for $B$ a Borel subgroup of $G$, then $H$ contains
  $(G,G)$.
\end{lemma}

\begin{proof}
  Since $(G,G)$ is connected, we may assume $H$ to be connected.
  First assume $G$ to be semisimple. Since $G = HB$, we have $G/B = H/(B\cap H)$. On the one hand this implies $\rk(G) = \rk(\pic(G/B)) = \rk(\pic(H/(H\cap B))) \leq \rk(H)$.
    On the other hand $\dim U_G = \dim G/B = \dim H/(H \cap B) \leq \dim U_H$ where $U_G$ and $U_H$ are maximal unipotent subgroups of $G$ and $H$. We deduce $\dim H \geq \dim G$ and $H = G$.

  For a general $G$, let $\pi:G\to G' = G/R(G)$ be the quotient of $G$ by its radical.
    We have $G' = H'B'$ with $H' = \pi(H)$ and $B' = \pi(B)$. Since $G'$ is semisimple, we deduce $G' = (G',G') \subset H'$. Thus $\pi\vert_H$ is surjective and $\pi((H,H)) = (G',G') = G'$. Since $(H,H) \cap R(G) \subset (G,G) \cap R(G)$ is finite, we get $\dim (H,H) = \dim G' = \dim (G,G)$ thus $(G,G) = (H,H) \subset H$.
\end{proof}

We deduce that $L_y$ contains $(L,L)$. This implies that $L_y$ does
not depend on $y$. Indeed, for $l \in L$, then $L_{l\cdot y} = lL_yl^{-1}$. For any $h\in L_y$ we have $lhl^{-1}h^{-1}\in (L,L)\subset L_y$ thus $lhl^{-1}\in L_y$ and $L_{l\cdot y}\subset L_y$ and by symmetry $L_{l\cdot y} = L_y$.
Denote by $L_Y$ this stabilisor. We have $(L,L)\subset L_Y$ thus
$L/L_Y$ is a quotient of $L/(L,L)$ which is a torus.

\vskip 0.1 cm

3. The orbit $S\cap Y = L \cdot y$ is thus isomorphic to the torus
$L/L_Y$. Let $(\chi_1,\cdots,\chi_n)$ a basis of the group of
characters of $L/L_Y$. Since $S\cap Y$ is closed in $S$ which is
affine, we can extend these functions to functions $(f_1,\cdots,f_n)$
on $S$. We may furthermore assume that these functions are $(L\cap
B)$-eigenfunctions of weights $(\chi_1,\cdots,\chi_n)$. These
functions do 
not vanish on the closed orbit $S\cap Y$ in $S$ thus they do not
vanish at all on $S$. These functions therefore define an
$L$-equivariant morphism $\psi:S\to (\G_m)^n\simeq L/L_Y$. Let $S_Y$ be the
fiber over the identity element of this morphism. The natural map
defined by the action: $L\times S_Y\to S$ factors through
$L\times^{L_Y}S_Y\to S$. This map is bijective. Indeed, if $s\in S$,
then there exists $l\in L$ such that $\bar l=\psi(s)$ and if $l'$
satisfies the same condition, then $l'=lh$ with $h\in L_Y$. Define
$s\mapsto [l,l^{-1}s] \in L \times^{L_Y} S_Y$. This is well
defined and an inverse map. But since $S$ is normal, this morphism
must be an isomorphism. 

Finally $\rk(X) - \rk(Y) = \rk(S) - \rk(S \cap Y) = \rk(S) -
\dim(L/L_Y) = \rk(S_Y)$. 
\end{proof}

\subsection{$G$-stable subvarieties}


We use the above result to prove that $G$-stable subvarieties in a spherical
variety are again spherical. We start with classical results on affine
$G$-varieties, see \cite[Theorem 9.4]{Gro2}

\begin{fact}
Let $X$ be an affine $G$-variety, then $\kk[X]^U$ is finitely generated.
\end{fact}

\begin{defn}
For an affine $G$-variety $X$, define the quotient $\pi:X\to X/\!/U$ as the morphism induced by the inclusion
$\kk[X]^U\to \kk[X]$. 
\end{defn}

Many of the properties of $X$ can be detected on $X/\! /U$. 

\begin{prop}
\label{prop-k(X)}
Let $X$ be an irreducible affine $G$-variety, then $X$ is normal if and only $X/\! /U$ is normal.
\end{prop}

\begin{proof}
One easily checks that the field of fractions of $\kk[X]^U$ is $\kk(X)^U$ (use that $U$-modules always have invariants). 
If $X$ is normal then so is $X/\! /U$. Indeed $\kk[X]$ is integrally
closed in $\kk(X)$ and thus $\kk[X]^U$ is integrally closed in $\kk(X)^U$.
Conversely, suppose that $X/\! / U$ is normal. Let $\pi:X'\to X$ be
the normalisation of $X$. We define a $G$-action on $X'$. The morphism $G\times X\to X$ induces a morphism $G\times X'\to
G\times X\to X$ and since $G\times X'$ is normal it factors through $X'$
so we have a commutative diagram:
$$\xymatrix{G\times X'\ar[r]\ar[d] & X' \ar[d] \\
G\times X \ar[r] & X.\\}$$
Because this is an action on an open subset (where $\pi$ is an
isomorphism) and the varieties are normal, this is an action.
Thus we also have
a quotient $X'/\! /U$ and a commutative diagram
$$\xymatrix{X'\ar[r]^{\pi}\ar[d] & X\ar[d]\\
X'/\! /U\ar[r]^{\bar \pi} & X/\! /U\\}$$
with $X'/\! /U$ and $X/\! /U$ normal varieties with $\kk(X')^U=\kk(X)^U$
so $\kk[X']^U$ and $\kk[X]^U$ have the same field of fractions. 
The algebra $\kk[X']$ is
the integral closure of $\kk[X]$ in $\kk(X)$ thus ideal 
$$I=\{f\in \kk[X]\ /\ f \kk[X']\subset \kk[X]\}.$$
is non trivial. It is stable under the action of $G$ and of $U$ so it
contains a $U$-invariant non trivial element $f\in \kk[X]^U$. We get 
$f \kk[X']^U\subset \kk[X]^U$. The
subspace $f \kk[X']^U$ is thus an ideal of $\kk[X]^U$ and thus a finite
$\kk[X]^U$-module. Therefore $\kk[X']^U$ is also a finite $\kk[X]^U$-module
but since $X/\! /U$ is normal we get $\kk[X']^U=\kk[X]^U$. Finally since
the $U$-invariants determine the module (note that we use here the assumption $\chara(\kk) = 0$ via representation theory) we get $\kk[X]=\kk[X']$.
\end{proof}

Recall that a toric variety is a spherical $T$-variety where $T$ is a torus.

\begin{cor}
  Let $X$ be an irreducible affine $G$-variety. Then $X$ is spherical if and only if $X/\! /U$ is a toric variety.
\end{cor}

\begin{proof}
If $X$ is normal so is $X/\! /U$ and the image of the dense $B$-orbit is a dense $T$-orbit.
Conversely if $X/\! /U$ is toric, then $X$ is normal and the inverse image of the dense $T$-orbit contains a dense $B$-orbit.
\end{proof}

Recall the following result on toric varieties.

\begin{lemma}
Let $Y$ be a toric $T$-variety. Then any irreducible $T$-stable subvariety is normal. 
\end{lemma}

\begin{cor}
  Let $X$ be a spherical $G$-variety, then any closed $G$-stable
  subvariety $X'$ is again a spherical $G$-variety.
\end{cor}

\begin{proof}
  We only have to prove that $X'$ is normal. By the local structure Theorem,
  we may assume $X$ affine. Then $X/\! /U$ is a normal affine toric
  variety and $X'/\! /U$ is a closed normal toric subvariety by the
  above lemma. The result follows.
\end{proof}

\begin{remark}
  One can also prove that a spherical variety has Cohen-Macaulay and even rational singularities (see \cite[Corollary 2.3.4]{survey}).
\end{remark}

\begin{example}
  In positive characteristic, the above result is not true anymore. Let $\chara(\kk) = p > 0$, let $G =\SL_2(\kk) \times \G_m^2$  and let $X = \kk^4$. Write $(x_1,x_2,y_1,y_2)$ for the coordinates in $X$ and write $x = (x_1,x_2)$ resp. $y = (y_1,y_2)$. Define a $\SL_2(\kk)$-action via the usual action $(g,x) \mapsto g\cdot x$ and $(g,y) \mapsto (F(g)^T)^{-1} \cdot y$ where $F$ denote the Frobenius map and $F(g)$ is the matrix obtained by applying $f$ to each coefficient. Define the $\G_m$ action as $(u,v) \cdot (x,y) =  (ux,vy)$ for $(u,v) \in \G_m^2$  and $(x,y) \in X$. Then $X$ is $G$-spherical. 

Now set $X' = V(x_1^py_1 + x_2^py_2)$. This is clearly $G$-stable and closed in $X$ but its singular locus contains $Z = V(x_1,x_2)$ which has codimension $1$ therefore $X'$ is not normal along $Z$ and is not $G$-spherical.
\end{example}

\subsection{Structure Theorem for toroidal varieties}

Recall that a spherical variety is {\bf toroidal} if $\odelta(X) = \Delta(X)$. For $X$ toroidal, we set
$$\Delta_X = \cup_{D \in \Delta(X)}D.$$
Note that if $\mathring{X}$ is the dense $B$-orbit, we have
$$\Delta_X = \overline{X \setminus \mathring{X}}.$$
Let $P_X$ be the stabiliser of $\mathring{X}$. It is also the stabiliser of $\Delta_X$ and $B \subset P_X$.

\begin{thm}
\label{str-toro}
  Let $X$ be spherical. The following conditions are
equivalent.

1. The variety $X$ is toroidal.

2. There exists a Levi subgroup $L$ of $P_X$ only depending on the open $G$-orbit of $X$ and a closed subvariety $Z$ of $X\setminus \Delta_X$
stable under $L$ such that the map
$$R_u(P_X)\times Z\to X\setminus\Delta_X$$
is a $P_X$-isomorphism. The group $(L,L)$ acts trivially on $Z$ which is a toric variety for a quotient of $L/(L,L)$. Any
$G$-orbit meets $Z$ along a unique $L$-orbit.
\end{thm}

\begin{proof}
  Assume that $X$ is toroidal. First remark that $\Delta_X$ is Cartier. Indeed this is a local condition and can be checked on the simple spherical varieties $X_{Y,G}$ for $Y$ any $G$-orbit. Then the restriction of $\Delta_X$ to $X_{Y,G}$ is Cartier by Lemma \ref{lemm-cartier}.

  Now proceed as in the proof of Theorem \ref{thm-stru} replacing $\D$ with $\Delta_X$, to obtain a variety $Z$ (called $S$ in the proof of Theorem \ref{thm-stru}).
  
The intersection of the dense $G$-orbit in $X$ with $X \setminus \Delta_X$ is the dense $B$-orbit $\mathring{X}$. The intersection of the dense $G$-orbit in $X$ with $Z$ is thus $\mathring{X} \cap Z$. It is a $B\cap L$-orbit and also a
$L$-orbit. By Lemma \ref{lemm-derive} we get that $(L,L)$ acts
trivially on $Z$ which has to be a toric variety under the action of a
quotient of $L/(L,L)$. 

Let $Y$ be a $G$-orbit in $X$. Then $Y$ is not contained in $\Delta_X$
therefore $Y\cap (X \setminus \Delta_X)$ is dense in $Y$ and 
$R_u(P_X)(Z\cap Y)$ is also dense in $Y$. But $Z\cap Y$ is the closure
of an $L$-orbit. The variety $Z$ is toric for some torus $T_Z$. Let
$Z'$ be the above orbit. The structure Theorem of spherical varieties applied
to toric varieties gives a $T_{Z'}$-variety $S_{Z'}$ and an isomorphism
$$T\times^{T_{Z'}}S_{Z'}\to Z.$$
The orbit $Z'$ therefore corresponds to a $T_{Z'}$-fixed point $s$ in
$S_{Z'}$. Consider the cone $\cC_{s}(S_{Z'})$ associated to $s$ (see for example \cite[Section 6]{jacopo}) 
and choose a basis $(\rho(\nu_D))$ (over $\Q$ of this cone) given by
$L_{Z'}$-stable divisors $D$. The affine chart gives that $s$ is the
intersection of these divisors therefore $Z'$ is the intersection of divisors
$D_1,\cdots D_r$ of $Z$ with $r={\rm codim}_{Z}(Z')$. Then
each divisor $X_i=\overline{R_u(P_X)D_i}$ is irreducible $B$-stable and
does not meet the dense $G$-orbit (this is true since $D_i$ does not meets the dense
$L$-orbit of $Z$). This implies that $X_i$ is $G$-stable. Now consider
$X'=\overline{R_u(P_X)Z'}$. It is a subvariety of codimension $r$ in
$X$ which is contained in the intersection of the $X_i$. Since
$\Delta_X$ contains no closed $G$-orbit the intersection of the $X_i$
has a dense open subset given by $(\bigcap_iX_i)\cap
(X\setminus\Delta_X)$. The variety $X'$ has to be an irreducible
components of the intersection of the $X_i$ and is thus $G$-stable. We
get $X'=\overline{Y}$ and $Y\cap (X\setminus\Delta_X)=R_u(P_X)Z'$ thus
$Y\cap Z=Z'$ proving the result.

Conversely, any $G$-orbit of $X$ meets $Z$ and thus is not contained
in $\Delta_X$ and therefore is not contained in any $B$-stable but not
$G$-stable divisor.
\end{proof}

\begin{remark}
  Note that $G(X\setminus\Delta_X)=X$.
\end{remark}

\begin{cor}
  \label{cor-sub-toro}
The irreducible $G$-stable subvarieties of a smooth toroidal variety are smooth, toroidal and transverse intersections of $G$-stable divisors.
\end{cor}

\begin{proof}
We need to consider the closure $Y$ of a $G$-orbit $G \cdot y$. The above Structure Theorem for toroidal varieties gives an isomorphism $X \setminus \Delta_X \simeq R_u(P_X) \times Z$ and any $G$-orbit meets $Z$ along a unique $L$-orbit. Furthermore, $Z$ is toric and smooth therefore $Y$ meets $Z$ along a smooth toric subvariety $Y'$ (see Lemma \ref{lem-toric} for a smoothness criterion for toric varieties). Then $R_u(P_X) \times Y'$ is an open subset of $Y$ with $G \cdot Y' = Y$ thus $Y$ is smooth and toroidal.

For the last assertion, $Y'$ being a toric subvariety of the smooth toric variety $Z$, it is a complete intersection of toric divisors and we get the result.
\end{proof}

\section{Divisors and Picard group}
\label{sec-pic}

In this section we describe, following results of Brion \cite{brion-nombre-car}, the Picard group and the group of Weil
divisors of a spherical variety. We identify globally generated and
ample divisors and give a description of the canonical divisor.
Recall that we assume $\chara(\kk)=0$.


\subsection{Simple spherical varieties}

Let us start with the following general result.

\begin{lemma}
Let $G$ be an affine normal $G$-variety
  containing a unique closed $G$-orbit $Y$. Then the restriction map
  $\pic(X)\to \pic(Y)$ is injective.
\end{lemma}

\begin{proof}
  Replacing $G$ by a finite cover we may assume that any line bundle
  $L$ on $X$ is $G$-linearised. If furthermore $L$ has a trivial
  restriction to $Y$, then there is a nowhere vanishing element
  $s\in H^0(Y,L\vert_Y)$. Since $H^0(Y,L\vert_Y)$ is a
  $G$-representation, this element has to be a $G$-eigenfunction (the
  composition $G\to Y\to \mathbb{A}^1$ defined by $g\mapsto s(g\cdot
  y)$ with $y\in Y$ is nowhere vanishing thus has to be a multiple of
  a character). But we have a surjective map $H^0(X,L)\to H^0(Y,L)$
  therefore we can lift $s$ to a section $s'\in H^0(X,L)^{(G)}$. The
  locus where $s'$ vanishes is then closed and $G$-stable therefore
  either empty or containing $Y$. The last case is impossible thus
  $s'$ is nowhere vanishing and $L$ is trivial
\end{proof}

\begin{remark}
  If $\chara(\kk)=p>0$, we cannot lift $s$ to a section $s'\in
  H^0(X,L)$ but only a power (in fact a $p^n$-th power) of $s$
  \emph{i.e.} there exists $s'\in H^0(X,L^{\otimes p^n})$ such that
  $s'\vert_Y=s^{p^n}$. We obtain that the kernel of the map
  $\pic(X)\to\pic(Y)$ is $p$-divisible.
\end{remark}

Let $X$ be a simple spherical $G$-variety with closed orbit $Y$.
Let $(\cC,\Delta)$ be the colored cone of $X$ (see \cite[Section 6]{jacopo}).

\begin{thm}
\label{thm-simple}
With the above notation, we have:

1. Any Cartier divisor on $X$ is linearly equivalent to 
$$\sum_{D\in \odelta(X)}n_DD \textrm{ with } n_D \in \Z.$$

2. The map $\cC^\perp \to \Z \odelta(X)$,
$\chi \mapsto \sum_{D \in \D_B \setminus 
  \D_X}\scal{\rho(\nu_D),\chi}D$
induces an exact sequence:
$$\cC^\perp\to \Z \odelta(X) \to \pic(X)\to 0.$$

3. A Cartier divisor is globally generated (resp. ample) if and
only if it is a linear combination $\sum_{D\in \odelta(X)}n_DD$ with $n_D\geq0$ (resp. $n_D>0$) for all $D$.
\end{thm}

\begin{proof}
  1. By the structure theorem for spherical varieties, there exists
  a parabolic subgroup $P$ and a closed subvariety $S$ of $X_{Y,B}$
  such that $X_{Y,B}$ is isomorphic to $R_u(P)\times S$ where $R_u(P)$
  is the unipotent radical of $P$. If $L$ is a Levi factor of
  $P$, then $S$ is $L$-spherical affine with a unique
  closed orbit $Y\cap S$ which is isomorphic to $L/L_Y$ such that
  $L_Y\supset (L,L)$. In particular $S\cap Y$ is a torus and
  $\pic(Y\cap S)$ is trivial. By the above lemma we get that
  $\pic(X_{Y,B})$ is trivial. 

So any divisor has support on the complement of $X_{Y,B}$. This is the
union of the divisors $D\in \odelta(X)$. By Lemma
\ref{lemm-cartier}, these divisors are Cartier.

2. Let $\delta=\sum_{D\in \odelta(X)}n_DD$ be linearly
equivalent to zero. There exists $f\in \kk(X)$ with 
$\divi(f)=\delta$. Since $\delta$ is a union of $B$-stable divisors,
the function $f$ has to be a $B$-eigenfunction. Furthermore, this function
does not vanish on $X_{Y,B}$ thus it does not vanish uniformly on $Y$
thus $\nu_D(f)=0$ (here $\nu_D$ is the valuation associated to $D$, see \cite[Section 4]{jacopo}) for divisors $D$ containing $Y$. In particular the
weight $\chi$ of $f$ lies in $\cC^\perp$ proving that $\divi (f)$ is
in the image of the left map of the above exact sequence.

Conversely, let $f\in \kk(X)^{(B)}$ be such that the weight $\chi$ of
$f$ is in $\cC^\perp$. The divisor $\divi(f)$ has to be a linear
combination of $B$-stable divisors not containing $Y$.

3. Let $\delta=\sum_{D\in \odelta(X)}n_DD$ be a Cartier
divisor and assume (by taking a covering of $G$ if necessary)
that $\delta$ is $G$-linearised. If all $n_D$ are non negative, then
$\delta$ is effective and we can consider the canonical section $\eta$
of this line bundle. We proved already that for $V$ the $G$-module spanned by
$\eta$ we have $V \otimes \cO_X \to \cO_X(\delta)$ is surjective (the
locus where the map is not surjective is closed and $G$-stable thus is
empty or contains $Y$. But this locus is contained in divisors not
containing $Y$).

If all $n_D$ are positive, we want to prove that $\delta$ is ample. It
is enough to prove that there exists a collection $(s_i)$ of sections
of $\delta$ which span $\delta$, such that $X_{s_i}$ is affine and
any element $f$ in $\kk[X_{s_i}]$ is of the form $f=f'/s_i^{n}$ for some
some $n$ and some $f'\in H^0(X,n\delta)$ \emph{i.e.} there are 
surjective morphism $(\oplus_nH^0(X,n\delta))_{(s_i)}\to \kk[X_{s_i}]$
giving an embedding on $X_{s_i}$. As above, let $\eta$ be the
canonical section of $\delta$. We have
$X_\eta=X_{Y,B}$. Since $X=X_{Y,G}=GX_{Y,B}$ is it
enough to prove our statement for $X_\eta=X_{Y,B}$, the statement will
follow for any $X_{g\eta}$ via the action of $g\in G$. We know that
$X_\eta$ is affine, let $f\in \kk[X_\eta] = \kk[X_{Y,B}]$. Since $f$ is
defined on this open set, the divisor of poles of $f$ is of the form
$$\divi_\infty(f)=\sum_{D \in \odelta(X),\ \nu_D(f)<0}\nu_D(f)D$$
thus for $n$ large enough $n\delta+\divi_\infty(f)$ is
effective. Therefore $\eta^nf$ is a global section of $n\delta$
proving the result.

Conversely, assume first that $\delta$ is globally generated, then the
exists $s\in H^0(X,\delta)$ with $s\vert_Y$ non trivial. We may choose
for $s$ a $B$-eigenfunction. We have $f\in \kk(X)^{(B)}$ such
that the effective divisor defined by $s$ is of the form
$\divi(f)+\delta$ and does not contain $Y$. Therefore it is a linear
combinaison of elements in $\odelta(X)$ with non negative
coefficients. 

If $\delta$ is ample, then the divisor 
$n\delta-\sum_{D\in \odelta(X)}D$
is globally generated for $n$ large enough and we may apply the
previous result.
\end{proof}

\begin{example}
  If $X$ is simple and the closed orbit $Y$ is complete, then $\cC$ is of maximal dimension and $\cC^\perp = 0$. In particular in this case
$\pic(X)=\Z\odelta(X).$
\end{example}

\begin{example}
For $X = G/P$ with $P$ a parabolic subgroup, we have $\pic(X)=\Z\odelta(X)$. Furthermore $\odelta(X)$ is the set of Schubert divisors and is indexed by the simple roots not in $P$.
\end{example}

\subsection{Weil divisors}

We now consider the group ${\rm Cl}(X)$ of Weil
divisors on a spherical variety $X$.

\begin{thm}
Let $X$ be a spherical variety.

1. Any Weil divisor on $X$ is linearly equivalent to 
$$\sum_{D \in \D(X)}n_DD \textrm{ with } n_D \in \Z.$$

2. The map $\X(X)\to \Z\D(X)$, 
$\chi\mapsto\sum_{D\in \D}\scal{\rho(\nu_D),\chi}D$ induces an
exact sequence: 
$$\X(X)\to \Z\D(X)\to {\rm Cl}(X) \to 0.$$
\end{thm}

\begin{proof}
1. Consider the dense $B$-orbit $\mathring{X}$. We claim that ${\rm Cl}(\mathring{X})=0$. 
This is a consequence of the local structure Theorem. Indeed,
let $Y$ be the dense $G$-orbit. The open affine subspace $X_{Y,B}$ is
$X\setminus\cup_{D\in \D(X)}D$ since no divisor can contain $Y$. The
intersection of $Y$ with $X_{Y,B}$ is the dense $B$-orbit $\mathring{X}$. Now there exists $P$ a parabolic subgroup of $G$, a
Levi factor $L$ of $P$ and $S$ a closed $L$-subvariety such that 
$X_{Y,B}\simeq R_u(P)\times S$. We get $\mathring{X}=Y\cap X_{Y,B}\simeq
R_u(P) \times (Y\cap S)$. But $Y\cap S$ is
isomorphic to $L/L_Y$ a torus, thus 
$$\mathring{X}\simeq R_u(P)\times(L/L_Y)$$
which is a product of $\G_a$ and $\G_m$ thus $\pic(\mathring{X})$ is
trivial. Since $\mathring{X}$ is an orbit thus smooth the same holds for ${\rm
  Cl}(\mathring{X})$.
In particular, any Weil divisor has support on the complement of $\mathring{X}$
which is the union of the divisors $D\in \D(X)$. 

2. Let $\delta = \sum_{D\in \D(X)}n_DD$ be a Weil divisor linearly
equivalent to zero. Then there exists $f\in \kk(X)$ such that
$\divi(f)=\delta$. Since $\delta$ is a union of $B$-stable divisors,
the function $f$ is a $B$-eigenfunction.
Conversely, let $f\in \kk(X)^{(B)}$ of weight $\chi\in\X(X)$. The
divisor of $f$ has to be a linear combination of $B$-stable divisors.
\end{proof}

\begin{cor}
  A simple spherical variety $X$ with closed orbit $Y$ is locally factorial (respectively locally $\Q$-factorial)
  if and only if for any $D\in \D_Y(X)$, there exists $\chi \in \X(X)$ (respectively $\chi \in \X(X)_\Q$) 
  such that 
$$\scal{\rho(\nu_{D'}),\chi}=\left\{
  \begin{array}{ll}
    1 & \textrm{if $D'=D$}\\
    0 & \textrm{if $D'\neq D$.}\\
  \end{array}\right.$$
In other words $X$ is locally factorial (respectively locally $\Q$-factorial) if and only if the elements $(\rho(\nu_D))_{D\in \D_Y(X)}$ are part of a basis of
$\X(X)^\vee$ (respectively of $\X(X)_\Q$).
\end{cor}

\begin{thm} \label{thm-Q-fact}
Let $X$ be a spherical variety.

(\i) The variety $X$ is locally factorial if and only if for any
$G$-orbit $Y$, the elements $(\rho({\nu_D}))_{D \in \D_Y(X)}$
form a part of a basis of
$\X(X)^\vee$.

(\i\i) The variety $X$ is locally $\Q$-factorial if and only if for
any $G$-orbit $Y$, the elements $(\rho({\nu_D}))_{D \in \D_Y(X)}$
form a linearly independent subset of
$\X(X)^\vee_\Q$.
\end{thm}

\begin{proof}
Since this is local we may assume that $X$ is simple with closed
orbit $Y$ and apply the previous result.
\end{proof}

\begin{example}
The above result gives a necessary but not sufficient condition for smoothness. Indeed, let $X$ be a smooth quadric of dimension at least 3.
Let $\hat X$
  be the cone over $X$. Then $\hat X$ is a 
locally factorial but not smooth spherical variety.
\end{example}

For toroidal varieties local factoriality gives a sufficent smoothness condition. Recall the following result on toric varieties.

\begin{lemma}
  \label{lem-toric}
  Let $X$ be a toric variety, then $X$ is smooth if and only if it is
  locally factorial \emph{i.e.} if and only if any of its cone is spanned by a basis of the monoid $\X(X)^\vee$.
\end{lemma}

\begin{cor}
  For $X$ is toroidal, $X$ is smooth if and only if it is locally
  factorial.
  \end{cor}

\begin{proof}
  By the structure Theorem of toroidal varieties (Theorem \ref{str-toro}), the open subset
  $X\setminus \Delta_X$ is isomorphic to $R_u(P_X)\times Z$ with $Z$ a
  toric variety. Therefore the singularities of $X\setminus \Delta_X$
  are those of $R_u(P_X)\times Z$. This open subset is smooth
  if and only if it is locally factorial. But $X = G(X \setminus \Delta_X)$ concluding the proof.
  \end{proof}

\subsection{Picard group, the general case}

Let $X$ be a spherical variety.
Recall that any Weil divisor $\delta$ can be
written in the form 
$$\delta=\sum_{D\in \D(X)}n_DD.$$

\begin{lemma}
  \label{lem-cartier}
  The divisor $\delta$ is Cartier if and only if for any $G$-orbit $Y$ of
  $X$, there exists $\chi_{\delta,Y}\in \X(X)$ such that
  $\scal{\rho(\nu_D),\chi_{\delta,Y}}=n_D$ for all $D\in \D_Y(X)$.
  \end{lemma}

\begin{proof}
  This follows directly from the case of simple spherical varieties
  and the fact that the $X_{Y,G}$ form a covering of $X$ when $Y$
  runs over all $G$-orbits.
\end{proof}

\begin{defn}
  Let $X$ be a spherical variety.
  
  1. Set $\cC_X = \cup_Y \cox$ where the union runs over all $G$-orbits $Y$.

  2. Define $\PL(X)$ the set of piecewise linear functions on $\cC_X$ as the
  subgroup of functions $l$ on $\cC_X$ such that, for any $G$-orbit $Y$, $l\vert_\cox$ is the restriction of an element of $\X(G/H)$,

  3. Define $\LL(X) \subset \PL(X)$ as the subgroup of linear functions \emph{i.e.} of the form $(\chi\vert_\cox)_Y$ for some $\chi \in \X(G/H)$.
\end{defn}

\begin{remark}
  \label{rem-res}
  We shall describe a function $l$ in $\PL(X)$ or $\LL(X)$ as a collection $(l_Y)_Y$ of functions indexed by the $G$-orbits defined by $l_Y = l\vert_\cox$. Note that we have $(l_Z)\vert_{\cox}=l_Y$, for any $G$-orbit $Z$ with $Z\subset\overline{Y}$.

  Note that if maximal cones do not have maximal dimension in $\X(X)^\vee_\Q$, then the above elements $l_Y$ do not determine functions on $\X(X)^\vee_\Q$
\end{remark}

\begin{thm}
  \label{thm-pic-s-t}
Let $X$ be a spherical variety.

1. There is an exact sequence
$$\cC_X^\perp \to \Z\odelta(X) \to \pic(X) \to
\PL(X)/\LL(X)\to 0.$$

2. The group $\pic(X)$ is of finite rank.

3. $\pic(X)$ is free if there exists a complete $G$-orbit.
\end{thm}

\begin{proof}
1. By Lemma \ref{lem-cartier}, for any Cartier divisor $\delta$, we have a family $l_\delta = (l_\delta)_Y = (\chi_{\delta,Y})_Y$ of characters indexed by the $G$-orbits. Note that the restriction condition in Remark \ref{rem-res} is satisfied. Since the image of a rationally trivial Cartier divisor lies in $\LL(X)$, this defines a morphism $\pic(X) \to \PL(X)/\LL(X)$ and Lemma \ref{lem-cartier} proves that for $l \in \PL(X)$, the divisor $\sum_{D \in \D(X)} \scal{\rho(\nu_D),l} D$ is Cartier so that this map is surjective.

  The kernel of this map is given by divisors of the form $\delta = \sum_{D \in \D(X)} n_D D$ such that there exists $l \in \LL(X)$ with $n_D = \scal{\rho(\nu_D),l}$ for $D \in \cup_Y \D_Y(X)$. Let $f \in \kk(X)^{(B)}$ with character $l$, we have $\delta - \divi(f) \in \Z\odelta(X)$.

  If $\delta = \sum_{D \in \odelta(X)} n_D D$ is rationally trivial, there exists $f \in \kk(X)^{(B)}$ with $\divi(f) = \delta$. But if $\chi$ is the character of $f$, we have $\scal{\rho(\nu_D),\chi} = n_D$. In particular  for $D \in \cup_Y \D_Y(X) = \D(X) \setminus \odelta(X)$, we have $\scal{\rho(\nu_D),\chi} = n_D = 0$ thus $\chi \in \cC_X^\perp$.

  2. The abelian group $\PL(X)$ is a subgroup of a product of free finitely generated abelian groups thus it is a free finitely generated abelian group. Since $\LL(X)$ is a subgroup of $\PL(X)$ it is also finitely generated. With the above exact sequence, we get the result.

  3. Note that $\PL(X)/\LL(X)$ is torsion free: if $l \in \PL(X)$ is such that $nl \in \LL(X)$, then $l$ itself is linear thus in $\LL(X)$. The quotient $\PL(X) / \LL(X)$ is therefore torsion free. If $X$ contains a proper $G$-orbit, then $\cC_X$ is of maximal dimension. This gives $\cC_X^\perp = 0$ and the torsion free assertion.
  \end{proof}

\begin{remark}
If $X$ has no complete $G$-orbit then the Picard group may have torsion as shows the example
$X=G$ with $G$ semi-simple not simply connected (this is a spherical variety for the action of $G\times G$).
\end{remark}

\begin{thm}
  \label{thm-pic-gg}
  Let $X$ be a spherical variety. For $\delta$ a Cartier $B$-stable divisor let
$l_\delta$ be the image of its class under the map $\pic(X)\to \PL(X)/\LL(X)$. We have 
  $$\delta=\sum_{D \in D(X) \setminus \odelta(X)} \scal{\rho(\nu_D),l_\delta} D + \sum_{D \in \odelta(X)} n_D D.$$
  Then $\delta$ is globally generated iff
    for any $G$-orbit $Y$, there exists $\chi_Y \in \X(X)$ with
  \begin{itemize}
  \item[$\bullet$] $\chi_Y\vert_\cox=l_\delta\vert_\cox$ ;
  \item[$\bullet$] $\chi_Y\vert_{\cC_X\setminus\cox}\leq l_\delta\vert_{\cC_X\setminus\cox}$ and
  \item[$\bullet$] $\scal{\rho(\nu_D),\chi_Y}\leq n_D$ for all
    $D \in \odelta(X)$.
  \end{itemize}
\end{thm}

\begin{proof}
The divisor $\delta$ is globally generated if and only if for any $G$-orbit $Y$, there exists $s\in H^0(X,\delta)$ such that $s\vert_Y$ does not vanish everywhere. We may furthermore assume that $s$ is a $B$-eigenvector. We then have
$\delta=A+\divi(f)$ with $A$ effective $B$-stable and $f$ a $B$-eigenfunction. Let $\chi_Y$ be the weight of $f$. Note that $A$ is effective and does not
contain $Y$ in its support. This is equivalent to the inequalities $\scal{\rho(\nu_D),l_\delta}\geq\scal{\rho(\nu_D),\chi_Y}$ for all $D\in \D(X) \setminus \odelta(X)$ with equality for $D\supset Y$ and $\scal{\rho(\nu_D),\chi_Y} \leq n_D$ for all $D \in \odelta(X)$.
\end{proof}

If $X$ is complete, the element $l=(l_Y)\in \PL(X)$ completely determines for each closed $G$-orbit $Y$ an element $\chi_Y\in\X(G/H)$ such that $l_Y=\chi_Y$ so that the value $l_Y(\rho({\nu_D}))=\scal{\chi_Y,\rho({\nu_D})}$ is well defined even if $\rho({\nu_D})$ does not lie in $\cox$. This simplifies the statement of the above result.

\begin{defn}
An element $l=(l_Y)\in \PL(X)$ is called convex if $l_Y \geq (l_Z)\vert_\cox$ for any orbit $Y$ and any closed orbit $Z$. If for $Z \not\subset \overline{Y}$ the inequality is strict on $\cox \setminus \cC_Z(X)$, then $l$ is called strictly convex.
\end{defn}

\begin{cor}
\label{theo-pic}
Let $X$ be a complete spherical variety.

For $\delta$ a Cartier $B$-stable divisor let
$l_\delta=(l_Y)$ be the image of its class under the map $\pic(X)\to \PL(X)/\LL(X)$. We have
$$\delta=\sum_{D \in \D(X) \setminus \odelta(X)} \scal{\rho(\nu_D),l_\delta} D + \sum_{D \in \odelta(X)} n_D D.$$

1. The divisor $\delta$ is globally generated if and only if $l_\delta = (l_Y)$ is convex and the inequality $l_Y(\rho({\nu_D}))\leq n_D$ holds
for all closed $G$-orbits $Y$ and all $D \in \odelta(X)$.

2. The divisor $\delta$ is ample if and only if $l_\delta = (l_Y)$ is strictly convex and the inequality $l_Y(\rho({\nu_D}))< n_D$ holds
for all closed $G$-orbits $Y$ and all $D \in \odelta(X)$.
\end{cor}

\begin{proof}
  1. Follows directly from the previous result.
  
2. By the previous result, a large multiple of $\delta$
separates closed orbits (\emph{i.e.} for two distinct closed orbits,
there is a divisor linearly equivalent to $\delta$ containing one of
the orbits and not the other) if and only if  $l_Y\neq l_Z$ for $Y$ and
$Z$ two distinct closed orbits. The proof therefore reduces to the
case of a simple spherical variety $X$ with closed orbit $Y$ in
which case we may choose $l_Y=0$. The result follows from Theorem \ref{thm-simple}.
\end{proof}

In the next three corollaries, $X$ is a complete spherical variety.

\begin{cor}
    Any ample divisor is globally generated.
\end{cor}

\begin{cor}
\label{nef=gg}
Nef and globally generated line bundles agree.
\end{cor}

\begin{proof}
By Corollary \ref{theo-pic}, the cone of globally generated line bundles is the closure of the ample cone. The latter is the nef cone
(see \cite[Theorem 1.4.23]{Lazarsfeld}).
\end{proof}

We write $N^1(X)$ for the group of divisors classes modulo numerical equivalence.

\begin{cor}
\label{pic=n1}
We have $\pic(X) = N^1(X)$.
\end{cor}

\begin{proof}
  This follows from the fact that $\pic(X)$ is torsion free for $X$ complete (see Theorem \ref{thm-pic-s-t}) and the fact that the group of homologically trivial divisors modulo rational equivalence is a torsion group \cite[Section 19.3.3]{fulton}.
\end{proof}

\begin{example}
It is not true that ample implies very ample. Consider $X$ the double cover of $\p^3$ ramified along the coordinate divisors. This is a normal toric variety for the torus $\G_m^3$ preserving the coordinates. However, the pull-back of $\cO_{\p^3}(1)$ is ample but not very ample: one checks that it precisely defines the above 2-to-1 cover. Note that $X$ is not smooth. For smooth toric varieties ample and very ample line bundles are the same (see \cite{demazure} or \cite{EW}).
\end{example}

\begin{cor}
A spherical variety $X$ is quasi-projective if and only if there
exists a strictly convex $\Q$-valued function on $\cC_X$ linear on each cone $\cox$.
\end{cor}

In \cite{knop}, Knop gives a characterisation of affine spherical
varieties. In the following statement $\mathcal{V}$ is the valuation cone of $X$ (see \cite[Section 4]{jacopo}).

\begin{thm}
A spherical variety $X$ is affine if and only if $X$ is simple and
there exists $\chi\in\X$ such that $\chi\vert_{\mathcal{V}}\leq0$,
$\chi\vert_{\cC_X}=0$ and
$\chi\vert_{\rho(\odelta(X))}>0$.
\end{thm}

\section{Curves and divisors}

In this section, we study $B$-stable curves and the duality between curves and divisors. We essentially review results from \cite{brion-mori} and \cite{brion-aus}. As a byproduct of this study we give an explicit canonical divisor on spherical varieties in Section \ref{sec-can}.

\subsection{Curves}

Let $X$ be a $G$-variety and $C$ be a complete irreducible curve in $X$ (proper integral one-dimensional subscheme) stable under the action of $B$. Denote by $\pi : \Ct \to C$ the normalisation map. Then $B$ also acts on $\Ct$ and $\pi$ is $B$-equivariant. Since $B$ is solvable, it has at least one fixed point in $C$.

\begin{prop}
  \label{prop-NUT-P}
  Assume that $B$ acts non trivially on $C$, then $C$ is rational and $B$ has at most two fixed points in $\Ct$.

  1. If $B$ has one fixed point in $\Ct$, then there exists a minimal parabolic subgroup $P$ such that the kernel of the action of $B$ on $C$ is $R(P)$, the radical of $P$.
  
  2. If $B$ has two fixed points in $\Ct$, then there exists a character $\chi$ such that $B$ acts via $\chi$.
\end{prop}

\begin{proof}
  Since $B$ acts non trivially, it has a dense orbit in $C$. But $B$ is rational so $C$ has to be rational. Now $\Ct$ is isomorphic to $\p^1$ and $B$ acts via automorphisms. Since the only automorphism of $\p^1$ fixing three points is trivial, $B$ has at most two fixed points.

  Consider $z \in C$ not fixed under the action of $B$. Then $B_z$, the stabiliser of $z$ is a closed subgroup of codimension $1$ in $B$. Let $U$ be the unipotent radical of $B$, then either $U \subset B_z$ or $U \cap B_z$ has codimension $1$ in $U$.

  Assume first that $U \subset B_z$. Then $U$ acts trivially on $z$ and since $U$ is normal in $B$, it acts trivially on any point of the $B$-orbit $B \cdot z$ and thefore $U$ acts trivially on $C$. In particular, the action of $B$ on $C$ factors through $B/U$ and thus through the action of a character $\chi : B \to \G_m$ (since $B/U$ is a torus and the kernel of its action on $C$ is a codimension one subgroup). Any non-trivial action of $\G_m$ on $\Ct \simeq \p^1$ fixes exactly two points, we are in case 2.

  Assume now that $H = B_z \cap U$ has codimension $1$ in $U$ thus $U$ acts non trivially on $\p^1$ and has therefore a unique fixed point. A maximal torus $T$ in $B$ acts on $\Ct \simeq \p^1$ with at least two fixed points so we can choose a $T$-fixed but not $B$-fixed point for $z$. In particular $T \subset B_z$ and $B_z = TH$. The subgroup $H$ is normalised by $T$ so it is the product of the root subgroups it contains. Let $\a$ be the unique root of $B$ not in $H$. If $\a$ is not simple, then $\a = \beta + \gamma$ with $U_\beta,U_\gamma \subset H$. This implies $U_\a \subset H$, a contradiction. In particular $\a$ is a simple root and $H = R_u(P)$, the unipotent radical of $P$, where $P$ is the minimal parabolic subgroup associated to $\a$. The kernel of the action of $B$ on $\Ct$, and then on $C$, is the intersection of the conjugates of $B_z = TH = TR_u(P)$ in $B$, this is the radical $R(P)$.
\end{proof}

Note that for $X$ spherical, the group $B$ never acts trivially on any curve $C$ so that the above proposition always applies.

\begin{defn}
A curve of type ($\chi$) is a $B$-stable curve with two $B$-fixed points.
\end{defn}

Assume that $\Ct$ has a unique $B$-fixed point $\xt$. Let $P(C)$ be the minimal parabolic subgroup defined by the above proposition. Let $x$ be the image of $\xt$ in $C$.

\begin{prop}
  \label{prop-unt}
  The curve $C$ is smooth and isomorphic to $\p^1$.
    
  Let $P = P(C)$. We have the following alternative.

  1. $P$ raises $x$ to $C$ so $C$ is stable under $P$. We have $C = P \cdot x$ and $C \subset G\cdot x$.

  2. $P$ raises $C$ to $P \cdot C$ with an edge of type {\rm (T)}. The map $P \times^B C \to P \cdot C$ is the normalisation map and is bijective.
    We have $P \times^B C \simeq \p^1 \times \p^1$ and $C \cap G \cdot x = \{x\}$.

  3. $P$ raises $C$ to $P \cdot C$ with an edge of type {\rm (N)}. The map $P \times^B C \to P \cdot C$ is two-to-one. We have $P \cdot C \simeq \p^2$ and $C \cap G \cdot x = \{x\}$.
\end{prop}

\begin{remark}
We will prove (Proposition \ref{prop-typeT-p1*p1}), using the Bia\l ynicki-Birula decomposition, that for $C$ raised by $P$ with an edge of type (T), the above morphism $P \times^B C \to P \cdot C$ is an isomorphism so $P \cdot C$ is isomorphic to $\p^1 \times \p^1$.
\end{remark}

\begin{proof}
  Let $z$ be a $T$-fixed but not $B$-fixed point in $C$. Such a point exists by the proof of the previous proposition. Let $R = R(P)$ be the radical of $P$. The group $R$ acts trivially on $z$ and on $C$. The group $S = P/R$ is simple of rank one and $S_z$ contains $B_z/R \supset TR/R$. So $S_z$ contains a maximal torus of $S$ and is a proper subgroup of $S$. We have three possiblitities: $S_z$ is a Borel subgroup of $S$ or $S_z$ is the maximal torus $T/(T \cap R)$ or $S_z$ is the normaliser of the maximal torus $T/(T \cap R)$.

  1. If $S_z$ is a Borel subgroup, then $P \cdot z = S \cdot z$ is isomorphic to $\p^1$ and  contains $B \cdot z$. In particular $C$ is isomorphic to $\p^1$, $C = P \cdot x$ and $C \subset G \cdot x$. This is case 1.

  2. If $S_z$ is the maximal torus $T/T \cap R$ of $S$, then $S_z = B_z/R$ so $P_z = B_z$ and the map $P \times^B C \to P \cdot C$ is birational over $z$ so we have an edge of type (T) or (U).

  3. If $S_z$ is the normaliser of maximal torus $T/T \cap R$ of $S$, then $B_z/R$ has index $2$ in $S_z$ so $B_z$ has index $2$ in $P_z$ and the map $P \times^B C \to P \cdot C$ is two-to-one over $z$ so we have an edge of type (N).

  We now prove that $C \cap G \cdot x = \{x\}$ for cases 2 and 3. We first prove that $C \cap P \cdot x = \{x\}$. Indeed $P \cdot x = S \cdot x$, and if $C$ contains some point $y \in S \cdot x$ with $y \neq x$, then $y$ is in the $B$-orbit of $z$ so $S_y$ and $S_z$ are conjugate. But $S_y$ and $S_x$ are also conjugate. This is imposssible since $S_x$ is a Borel subgroup of $S$ while $S_z$ is not. Now remark that $C \cap G \cdot x \subset (G \cdot x)^R$, so we only have to prove the equality $(G \cdot x)^R = P \cdot x$.

  The group $R$ being normal in $P$, we have $P \cdot x \subset (G \cdot x)^R$. Conversely, for $g \cdot x \in (G \cdot x)^R$, we have $g^{-1}Rg \subset G_x$. But $R$ is solvable and connected and $B \subset G_x$ is a Borel subgroup of $G_x$ so $g^{-1}Rg$ is conjugated in $G_x$ to a subgroup of $B$. So there exists $h \in G_x$ such that $(gh)^{-1}Rgh \subset B$. Write $gh = bnb'$ with $b,b' \in B$ and $n \in N_G(T)$. Since $B$ normalises $R$ we have $R \subset nBn^{-1}$ so $RT \subset nBn^{-1}$. This implies $n \in P$ thus $gh \in P$ and $g \cdot x = gh \cdot x \in P \cdot x$, thus $(G \cdot x)^R = P \cdot x$.

  We prove that, in cases 2 and 3, the curve $C$ is smooth. We only need to prove that the $B$-fixed point $x \in C$ is a smooth point. Note that $R$ acts trivially on $C$ so we consider $X$ as a $S$-variety with Borel subgroup $B/R$ and maximal torus $T/(T \cap R)$. By Theorem \ref{sumihiro}, we may assume that $X$ is isomorphic to $\p(M)$ where $M$ is a rational finite dimensional $S$-module. Write $M = \oplus_{r_i} M_{r_i}$ where $M_{r_i}$ is the irreducible $\SL_2(\kk)$-module of dimension $r_i+1$ and basis $u^k v^{r_i-k}$ such that the unipotent part of $B$ fixes $u$ and maps $v$ to $\{v + tu \ | \ t \in \kk \}$. For $m \in M \setminus \{0\}$, write $[m]$ for the corresponding point in $\p(M)$. Let $m,n \in M$ such that $z = [m]$ and $x = [n]$. Since $z$ is $T$-stable, we have $m = \sum_i \lambda_i u^{k_i}v^{r_i - k_i}$ such that $r_i - 2k_i$ is constant. Replacing $V$ by a submodule, we may assume that $\lambda_i \neq 0$ for all $i$. We then have
$B \cdot z =\{[\sum_i \lambda_i u^{k_i}(v+tu)^{r_i - k_i}] \ | \ t \in \kk \}$ and $x= [n]$ is the limit when $t$ goes to infinity so $n = \sum_i \lambda_i u^{r_i}$. Now locally at $x$, since $z$ is not $B$-stable, there exists an index $i$ with $k_i < r_i$ and the above expression in $t$ has a degree one term and is therefore smooth at $t = \infty$.

  Case 2: the normalisation of $\overline{S \cdot z}$ is the unique complete embedding of $\SL_2(\kk)/T_0$ (see Example \ref{exam-sl2}), it raises $C$ with an edge of type (T) and is isomorphic to $\p^1 \times \p^1$. The normalisation map is $S$-equivariant. It is an isomorphism onto $S \cdot z$ and bijective over the unique closed $S$-orbit $S \cdot x \simeq S/B \simeq \p^1$. Since $P \times^B C$ is a complete normal (since $C$ is smooth) embedding of $S \cdot z \simeq \SL_2(\kk)/T_0$ it is isomorphic to $\p^1 \times \p^1$ and is the normalisation of $P \cdot C$.

  Case 3: the normalisation of $\overline{S \cdot z}$ is the unique complete embedding of $\SL_2(\kk)/N_0$ (see Example \ref{exam-sl2}), it raises $C$ with an edge of type (N) and is isomorphic to $\p^2$. Since $P \times^B C$ is a complete normal (since $C$ is smooth) embedding of $S \cdot z \simeq \SL_2(\kk)/N_0$ it is isomorphic to $\p^2$. We therefore only need to prove that $\overline{S \cdot z}$ is smooth. Using notation as above we have $z = [v]$ with $v = \sum_i \lambda_i u^{k_i}v^{r_i - k_i}$. Furthermore $z$ is stable under the normaliser of $T/(T\cap R)$ therefore $k_i = r_i - k_i$ for all $i$. We may again assume that $\lambda_i \neq 0$ for all $i$. Now the projection $\p(V) \dasharrow \p(V_{r_1})$ is defined over $\overline{S \cdot z}$, so we only need that its image is smooth. But its image is the $k_1$-Veronese embedding of $\p^2$ and is therefore smooth.
\end{proof}

\begin{defn}
  Let $C$ be a $B$-stable curve with a unique $B$-fixed point $x$. 

  1. The curve $C$ is of type (U) if $P(C)$ raises $x$ to $C$ with an edge of type (U).

  2. The curve $C$ is of type (T) if $P(C)$ raises $C$ to $P(C) \cdot C$ with an edge of type (T).
  
  3. The curve $C$ is of type (N) if $P(C)$ raises $C$ to $P(C) \cdot C$ with an edge of type (N).
\end{defn}

For $C$ a curve of type (T) or (N) define
$$Q(C) = \{ g \in G \ | \ gP(C) \cdot C \subset P(C) \cdot C\}.$$
Then $Q(C)$ is a parabolic subgroup containing $P(C)$.

\begin{prop}
  \label{prop-GC-TN}
For $C$ of type {\rm (T)} or {\rm (N)}, the map $p : G \times^{Q(C)} P(C) \cdot C \to G \cdot C$ is birational. If furthermore $X$ is toroidal, then $p$ is an isomorphism.
\end{prop}

\begin{proof}
  Let $z$ be the unique $T$-fixed but not $B$-fixed point in $C$. Let $H = G_z$ and $\alpha$ be the simple root such that $P = P(C) = P_\a$ is the minimal parabolic subgroup associated to $\a$. We have $H \cap B = T R_u(P)$. The Lie algebra $\gh$ of $H$ is of the form
  $$\gh = \gt \oplus \left( \bigoplus_{\beta \in R^+ \setminus \{\a\}} \g_\beta \right) \oplus \left( \bigoplus_{\gamma \in E} \g_\gamma \right),$$
  for some subset $E \subset R^-$ (here $\gt$ is the Lie algebra of $T$, $\g_\beta$ are the root spaces in the Lie algebra $\g$ of $G$ and $R$ is the root system of $G$). We have $\gamma + \gamma' \in E$ for $\gamma,\gamma' \in E$ and $\gamma + \gamma' \in R$.

  We prove that indecomposable elements of $E$ are negatives of simple roots that are orthogonal to $\a$. Let $\gamma \in E$ be indecomposable such that $\gamma = \gamma' + \gamma''$ with $\gamma', \gamma'' \in R^-$. If $\gamma'$ and $\gamma''$ are different from $-\a$, then $\g_{-\gamma'},\g_{-\gamma''} \subset \gh$ thus $\g_{\gamma'} = [\g_{\gamma},\g_{-\gamma''}]$ and $\g_{\gamma''} = [\g_{\gamma},\g_{-\gamma'}]$ are contained in $\gh$, a contradiction to the indecomposability of $\gamma$. We thus have $\gamma' = -\a \neq \gamma''$ and $\g_{-\a} = [\g_\gamma,\g_{-\gamma''}] \subset \gh$. If $s_\a$ is the simple reflection associated to $\a$, we get that $s_\a(B)$ is contained in $H$ and $C = \overline{B \cdot z} \subset \overline{G \cdot z}$. But since the stabiliser $H$ of $z$ contains a Borel subgroup, the orbit $G \cdot z$ is projective and the curve is of type (U), a contradiction. So $\gamma$ is indecomposable in $R^-$ and $-\gamma$ is simple. If $\gamma$ is not orthogonal to $\a$, then $-\gamma + \a \in R^+ \setminus \{\a\}$ and $\g_\a = [\g_{-\gamma + \a},\g_\gamma]$ is contained in $\gh$ and $H$ contains $B$ a contradiction.

  Set $I = \{ \beta \in R^+ \ | \ \beta \textrm{ simple with } -\beta \in E\}$ and $J = I \cup \{\a\}$. Let $R_I$ and $R_J$ be the corresponding root systems. We have $R_J = R_I \cup \{-\a,\a\}$ and
  $$\gh = \gt \oplus \left( \bigoplus_{\beta \in R^+ \setminus R_J} \g_\beta \right) \oplus \left( \bigoplus_{\beta \in R_I} \g_\gamma \right).$$
  In particular $R_u(H) = R_u(P_J)$ the unipotent radical of the parabolic subgroup with root system $R_J$ and thus $H \subset P_J$. We have $P_\a H = P_J$ thus $P_J \cdot z \subset P_\a \cdot z$ so $P_J \cdot C \subset P_\a \cdot C$ and $P_J \subset Q(C)$. We also have $Q(C) \cdot z \subset (P_\a \cdot C) \cap G \cdot z = P_\a \cdot z$ thus $Q(C) \subset P_\a H$. Since $P_\a H = P_J$ we deduce that $H \subset Q(C) = P_J$. In particular the map $G \times^{Q(C)} P_\a \cdot z \to G \cdot z$ is an isomorphism and $\varphi$ is birational.

  If $X$ is toroidal, the same is true for any $G$-orbit so $G \cdot C$ is toroidal (this can be obtained from the Structure Theorem for toroidal varieties Theorem \ref{str-toro} as in the proof of Corollary \ref{cor-sub-toro}, see also \cite[Section 7]{jacopo}). Let $x$ be the $B$-fixed point in $C$. Since $C = B \cdot z \cup B \cdot x$, we have $G \cdot C = G \cdot z \cup G \cdot x$. The cone of $G \cdot z$ is therefore of codimension one in the cone of $G \cdot x$ and thus (see \cite[Lemma 6.4]{knop}) the orbit $G \cdot x$ is a divisor in $G \cdot C$. The morphism $G \times^{Q(C)} P \cdot x \to G \cdot x$ is therefore finite and since $G \cdot C$ is normal, the morphism $\varphi$ is an isomorphism.
\end{proof}

We finish with the behaviour of $B$-stable curves under morphisms.

\begin{prop}
  \label{prop-curv-map}
  Let $\varphi : X \to X'$ be a proper morphism between two $G$-varieties, let $C \subset X$ be a $B$-stable curve not contracted by $\varphi$.

  1. If $C$ is of type {\rm (}$\chi${\rm )}, {\rm (U)} or {\rm (N)}, then so is $\varphi(C)$.

  2. If $C$ is of type {\rm (T)}, then $\varphi(C)$ is of type {\rm (U)}, {\rm (T)} or {\rm (N)}.

  3. If $X'$ is toroidal and $\varphi$ birational, then $C$ and $\varphi(C)$ have the same type.
\end{prop}

\begin{proof}
  1. For type($\chi$), the Borel $B$ acts via a character and the same occurs for $\varphi(C)$. For type (U), the parabolic $P(C)$ acts transitively on $C$ and also on $\varphi(C)$. For type (N), the restriction of $\varphi$ to $P(C) \cdot C$ is a $\PGL_2(\kk)$-morphism with source $\p^2$. Such a morphism is either an isomorphism or trivial.

  2. If $C$ has a unique $B$-fixed point so does $\varphi(C)$ proving the claim.

  3. Let $C$ be of type (T) and $x$ be its $B$-fixed point. We prove that the restriction of $\varphi$ to $G \cdot C$ is an isomorphism. Let $\mathfrak{F}$ and $\mathfrak{F}'$ be the colored fans of $X$ and $X'$. Then $G \cdot x$ corresponds to a maximal cone in $\mathfrak{F}$ while $G \cdot C$ corresponds to a cone $\sigma \in \mathfrak{F}$ of codimension one. By Proposition \ref{prop-GC-TN}, we have an isomorphism $G \times^{Q(C)} P(C) \cdot C \to G \cdot C$ with $P(C) \subset Q(C)$. Since $P(C)$ has two orbits in $P(C) \cdot C$ (its normalisation is the $\SL_2(\kk)$-variety $\p^1 \times \p^1$), $G$ has two orbits in $G \cdot C$. In particular $\sigma$ is contained in a unique maximal cone of $\mathfrak{F}$ so is contained in a face of the valuation cone of $X$. Since $\varphi$ is proper and birational and $X'$ is toroidal, we get (see \cite[Theorem 4.2]{knop} or \cite[Theorem 8.15]{jacopo}) that $\sigma$ is also a cone of $\mathfrak{F}'$ proving our claim.
\end{proof}

\begin{example}
  Let $X = \p^1 \times\p^1$ be the unique $\SL_2(\kk)/T_0$ complete embedding, let $C = \p^1 \times \{0\}$ which is of type (T) and consider the following $\SL_2(\kk)$-morphisms:
    \begin{itemize}
  \item[1.] $\varphi : \p^1 \times \p^1 \to \p^1$, the first projection;
  \item[2.] $\varphi : \p^1 \times \p^1 \to \p^1 \times \p^1$, the identity and 
  \item[3.] $\varphi : \p^1 \times \p^1 \to \p^2$ the quotient by the involution $(x,y) \mapsto (y,x)$.
  \end{itemize}
  Then in case 1, the curve $\varphi(C)$ is of type (U), in case 2, the curve $\varphi(C)$ is of type (T) and in case 3, the curve $\varphi(C)$ is of type (N).   
\end{example}

\subsection{Duality between curves and divisors}

We describe the cone of effective curve and the duality between curves and divisors in connection with the exact sequence of Theorem \ref{thm-pic-s-t}. In this subsection we will assume that $X$ is complete.

Denote by $Z_1(X)$ the group of one-dimensional cycles on $X$ and by $A_1(X)$ its quotient by rational equivalence. Consider also the pairing $(C,D) \mapsto \deg(D\vert_C)$ between $Z_1(X)$ and the group of Cartier divisors on $X$. The orthogonal to the group of Cartier divisor is the subgroup of numerically trivial $1$-cycles and induces a stronger equivalence relation called numerical equivalence, the quotient by this relation is denoted by $N_1(X)$. We have a canonical surjective map $A_1(X) \to N_1(X)$. For $X$ complete, there groups have finite rank (see \cite[Section 19.1.4]{fulton}). In $N_1(X)_\Q = N_1(X) \otimes_\Z \Q$, denote by $NE(X)$ the cone spanned by the classes of effective curves.

\begin{defn}
Let $X$ be a spherical variety. A wall of the colored fan $\mathfrak{F}$ of $X$ is a codimension one cone in $\Lambda(X)^\vee$ not contained in the boundary of $\cV$. Note that a wall is a common face of two maximal colored cones.
\end{defn}

\begin{thm}
  Let $X$ be complete spherical variety.

  1. The space $N_1(X)_\Q$ is generated by classes $C_\mu$ indexed by the walls $\mu$ of $\mathfrak{F}(X)$ and classes $C_{D,Y}$ indexed by pairs where $Y$ is a closed orbit and $D \in \Delta(X) \setminus \Delta_Y(X)$.

  2. The cone $NE(X)$ is generated by the classes $C_\mu$ and $C_{D,Y}$.

\vskip 0.1 cm
  
  \noindent
  Write $\overline{N_1(X)}$ for the quotient of $N_1(X)$ by the subgroup spanned by the classes $C_\mu$.

\vskip 0.1 cm
  
  3. The image $\overline{C}_{D,Y}$ in $\overline{N_1(X)}$ of the class $C_{D,Y}$
  does not depend on $Y$ with $D \in \Delta(X) \setminus \Delta_Y(X)$. Furthermore $\overline{C}_{D,Y} = 0$ for $D \not \in \odelta(X)$.

  4. In $\overline{N_1(X)}_\Q$, the family $(\overline{C}_{D,Y})_{D \in \odelta(X)}$ is a basis and generates the image $\overline{NE(X)}$ of $NE(X)$.
\end{thm}

\begin{proof}
Recall from Corollary \ref{pic=n1}, that $\pic(X) = N^1(X)$. Recall also from Theorem \ref{thm-pic-gg} that any Cartier divisor $\delta$ is modulo linear equivalence of the form
$$\delta = \sum_{D \in \D(X) \setminus \odelta(X)} \scal{\rho(\nu_D),l_\delta} D + \sum_{D \in \odelta(X)} n_D D,$$
where $l_\delta = (l_Y)$ is the image of its class under the map $\pic(X)\to \PL(X)/\LL(X)$. Since $X$ is complete, we may consider $l_\delta$ as a function on $\Lambda(X)^\vee$. The divisor $\delta$ is globally generated if and only if $l_\delta$ is convex and $\scal{\rho(\nu_D),l_Y} \leq n_D$ for all $D \in \odelta(X)$.

Let $\mu$ be a wall of $\mathfrak{F}(X)$, we define the class $C_\mu$. The orthogonal $\mu^\perp$ of $\mu$ in $\Lambda(X)$ is one dimensional and we pick a primitive vector $v \in \mu^\perp$. Let $\mu^+$ and $\mu^-$ be the two maximal cones having $\mu$ as face. We pick $v$ so that $v\vert_{\mu^+} > 0$. The difference $l_\delta\vert_{\mu^+} - l_\delta\vert_{\mu^-}$ vanishes on $\mu$ since $l_\delta$ is continuous so $l_\delta\vert_{\mu^+} - l_\delta\vert_{\mu^-}$ is a multiple of $v$. It only depends on the class of $\delta$ in $\pic(X)$ (picking another representative changes $l_\delta$ by a linear function). This defines a linear form $C_\mu$ on $N^1(X) = \pic(X)$ via
$$C_\mu \cdot \delta = (l_\delta\vert_{\mu^+} - l_\delta\vert_{\mu^-})/v.$$

Let $Y$ be a closed orbit and $D \in \Delta(X) \setminus \Delta_Y(X)$. Writing $n_D = \scal{\rho(\nu_D),l_\delta}$ for $D \in \D(X) \setminus \odelta(X)$ so that
$$\delta = \sum_{D \in \D(X)} n_D D,$$
we define the linear form on $N^1(X) = \pic(X)$ via (with notation as above)
$$C_{D,Y} \cdot \delta = n_D - \scal{\rho(\nu_D),l_Y}.$$
Note that this only depends on the class of $\delta$ in $\pic(X)$. 

We prove that these classes span $N_1(X)$. By duality, it suffices to prove that a divisor $\delta$ vanishing on all these forms is linearly trivial. If $C_\mu \cdot \delta = 0$ for any wall $\mu$, then $l_\delta\vert_{\mu^+} = l_\delta\vert_{\mu^-}$ for all walls $\mu$ so that $l_\delta$ is linear and equal to the character of a $B$-semiinvariant function $f_\delta$.
For $D \in \odelta(X)$ and any closed orbit $Y$, we have $D \in \Delta(X) \setminus \Delta_Y(X)$ so that $C_{D,Y} \cdot \delta = 0$. This gives $n_D = \scal{\rho(\nu_D),l_Y} = \scal{\rho(\nu_D),l_\delta}$. We deduce 
$$\delta = \sum_{D \in \D(X)} \scal{\rho(\nu_D),l_\delta} D = \div(f_\delta).$$
In particular $\delta$ is linearly trivial and the classes $(C_\mu)_\mu$ and $(C_{D,Y})_{D,Y}$ span $N_1(X)$.
Furthermore, a divisor $\delta$ has non negative value on all classes $(C_\mu)_\mu$ and $(C_{D,Y})_{D,Y}$ if and only if $l_\delta$ is convex and $\scal{\rho(\nu_D),l_Y} \leq n_D$ for all $D \in \odelta(X)$ so if and only if $\delta$ is globally generated. These classes therefore span the cone $NE(X)$ of effective curves.

Let $D \in \Delta(X) \setminus \Delta_Y(X)$ with $D \not \in \odelta(X)$. Then there exists a closed orbit $Z$ with $D \in \Delta_Z(X)$ and we have $C_{Y,D} \cdot \delta = n_D - \scal{\rho(\nu_D),l_Y} = \scal{\rho(\nu_D),l_Z} - \scal{\rho(\nu_D),l_Y}$. Choosing a path from the cone of $Y$ to the cone of $Z$ and the corresponding walls $(\mu_i)$, the linear form $C_{D,Y}$ can be expressed as a linear combinaison of the linear forms $C_{\mu_i}$. In particular $\overline{C}_{D,Y} = 0$.

Now assume $D \in \odelta(X)$ and let $Y$ and $Z$ be closed orbits.
Then $\cdot (C_{D,Y} - C_{D,Z}) \cdot \delta = \scal{\rho(\nu_D),l_Z} - \scal{\rho(\nu_D),l_Y}$ and as above $C_{D,Y} - C_{D,Z}$ is a linear combinaison of the $(C_\mu)_\mu$. In particular $\overline{C}_{D,Y} = \overline{C}_{D,Z}$.

Finally the classes $(\overline{C}_{D,Y})_{D \in \odelta(X)}$ span $\overline{N_1(X)}$. Since for $D,D' \in \odelta(X)$ we have $C_{D,Y} \cdot D' = \delta_{D,D'}$ and $C_\mu \cdot D' = 0$. Since the divisors in $\odelta(X)$ form a free family in $\pic(X) = N^1(X)$ (see Theorem \ref{thm-pic-s-t}) and since $\overline{N_1(X)}$ is torsion free, assertion 4 follows.
\end{proof}

\begin{remark}
  According to Theorem \ref{thm-pic-s-t}, we have an exact sequence
  $$\Z\odelta(X) \to \pic(X) = N^1(X) \to \PL(X)/\LL(X) \to 0.$$
The classes $(C_\mu)$ with $\mu$ a wall of $\mathfrak{F}(X)$ are dual to $\PL(X)/L(X)$ while the classes $(\overline{C}_{D,Y})_{D \in \odelta(X)}$ are dual to $\Z\odelta(X)$.  
\end{remark}

\begin{prop}
  \label{prop-dual}
  Let $X$ be a complete spherical variety.

  1. Any effective cycle is rationally equivalent to an effective $B$-stable cycle.

  2. For any two distinct closed orbits, the following are equivalent
  \begin{itemize}
  \item[a.] There exists a 
    curve $C_{YZ}$ of type {\rm (}$\chi${\rm )} meeting $Y$ and $Z$.
  \item[b.] The cones $\cC_Y(X) \cap \cV$ and $\cC_Z(X) \cap \cV$ share the same wall $\mu$.
  \end{itemize}
  In that case $C_{YZ}$ is unique, isomorphic to $\p^1$ and its class in $N_1(X)$ is equal to $C_\mu$.

  3. For $D \in \odelta(X)$, any class $\overline{C}_{D,Y}$ is represented by an irreducible $B$-stable curve of type {\rm (U)}, {\rm (N)} or {\rm (T)}.
\end{prop}

\begin{proof}
  1. We follow a proof by Vust written in \cite[Lemma 6.1]{MJ}. We proceed by induction on the dimension of $B$ so we may assume that any effective curve is rationaly equivalent to an effective curve $C$ stable under a codimension one closed subgroup $\Gamma$ of $B$. We first prove that the image $B \times^\Gamma C$ of the map $B \times C \to (B/\Gamma) \times X$ defined by $(b,x) \mapsto (b\cdot \Gamma,bx)$ is closed. Let $Z$ be the closure and assume that there is a point $z \in Z$ not in $B \times^\Gamma C$. Note that $Z$ is of pure dimension 2. By letting $B$ act, we may assume that $z \in (\Gamma/\Gamma) \times X = F$. The intersection of $F$ with $Z$ is one dimensional and its intersection with $B \times^\Gamma C$ is $(\Gamma/\Gamma) \times C$ which is closed in $F$. In particular $Z \cap F$ must contain another one dimensional irreducible component, say $C'$, containing $z$. In particular $C' \setminus B \times^\Gamma C$ is one dimensional. But $B \cdot (C' \setminus B \times^\Gamma C)$ has to be two dimensional and contained in $Z \setminus B\times^\Gamma C$, a contradiction.

    Now let $j:B/\Gamma \to \p^1$ be a $B$-equivariant embedding and let $Y$ be the closure of the image of the composition $B \times^\Gamma C \to (B/\Gamma) \times X \to \p^1 \times X$. The variety $Y$ is of dimension 2. Since $B/\Gamma$ is not complete, there exists a point $t_0 \in \p^1$ fixed by $B$ (and therefore not in the image of $j$). Let $t_1 = j(\Gamma/\Gamma)$ and denote by $\pi : Y \to \p^1$ the restriction of the first projection. The fibers $\pi^{-1}(t_0)$ and $\pi^{-1}(t_1)$ are curves in $Y$ and so are their images $p_2(\pi^{-1}(t_0))$ and $p_2(\pi^{-1}(t_1))$ in $X$. These two curves are linearly equivalent. Since $t_0$ is $B$-fixed, the curve $p_2(\pi^{-1}(t_0))$ is $B$-stable. But since $B \times^\Gamma C$ is closed in $(B/\Gamma) \times X$, we have $p_2(\pi^{-1}(t_1)) = C$ proving the result.
    
    2. (a. $\Rightarrow$ b.) Consider $G \cdot C_{YZ}$ which is closed in $X$, irreducible and $G$-stable. It meets $Y$ and $Z$ and therefore contains both. Since furthermore $B$ acts on $C_{YZ}$ via a character, the rank of $G \cdot C_{YZ}$ is one. The combinatorial description of orbits in terms of colored cones (see \cite[Section 7]{jacopo}) implies that the cones $\cC_Y(X) \cap \cV$ and $\cC_Z(X) \cap \cV$ share the same wall $\mu$.

    (b. $\Rightarrow$ a.) Let $X_\mu$ be the closure of the orbit associated to $\mu$. It contains $Y$ and $Z$ and has rank one. We prove that $X_\mu^U$ is an irreducible $B$-stable curve, smooth, of type ($\chi$) and passing through $y$ and $z$.

    The locus $X_\mu^U$ is $B$-stable and any $B$-orbit in $X_\mu^U$
    is a $T$-orbit with $T$ a maximal torus. Since $X_\mu$ has rank
    one, this orbit is of dimension at most one. Furthermore, we have
    $y,z \in X_\mu^U$.

    We prove that $X_\mu^U$ is connected. We actually prove that for
    $W$ a connected $U$-variety $W^U$ is connected. By Theorem
    \ref{thm-chow-lemma}, we may assume that $W$ is projective and
    equivariantly embedded in $\p(V)$ where $V$ is a $G$-module. We
    proceed by induction on the dimension. For $\dim W = 0$ this is
    obvious. For $\dim W = 1$, if $U$ acts trivially then $W^U = W$
    and we are done. Otherwise $U$ has a dense orbit and the
    normalisation of $W$ is isomorphic to $\p^1$ with $U$ acting via
    a maximal unipotent subgroup of $\textrm{PGL}_2(\kk)$. Therefore
    $W^U$ is reduced to one point and the result follow. Assume now
    $\dim W >2$. Then intersecting with $\p(H)$ where $H$ is a
    $U$-stable linear subspace of $V$ containing $V^U$ and using
    Bertini's Theorem (see \cite[Corollary II.7.9]{hartshorne-book}), we are
    reduced to the case of $W \cap \p(H)$ which is connected and can
    be chosen of dimension one less except if $W = W^U$.

We now prove that $X_\mu^U$ is irreducible. The above argument implies that $X_\mu^U$ meets the three orbits of $X_\mu$ and that
$X_\mu = G \cdot X_\mu^U$. Let $x$ be a point in $X_\mu^U$ and in the
dense $G$-orbit. Following \cite{pauer}, we prove that $(G \cdot x)^U
= B \cdot x$. Embbed $G$-equivariantly $X_\mu$ in $\p(V)$ where $V$ is a
$G$-module. Then $x = [v]$ with $v \in V^U$. Decompose $V$ as a direct
sum of irreducible $G$-modules $V = \oplus_\lambda V_\lambda$, then
$v$ is the sum of highest weight vectors $v = \sum_\lambda v_\lambda$. Let $P_\lambda$ be the parabolic subgroup fixing
$v_\lambda$, we have $G_v = \cap_\lambda \ker(\lambda :
P_\lambda \to \G_m)$ and $N_G(G_x) = \cap_\lambda 
P_\lambda$. We deduce that $N_G(G_v)/G_v$ is a torus and
$N_G(G_v)  = BG_v$. Let $g \in G$ such that $g \cdot x
\in (G \cdot x)^U$. Then $g \cdot v \in \sum_\lambda \kk v_\lambda$
thus $g \in \cap_\lambda P_\lambda = N_G(G_v) = B
G_v$. We get $g \cdot v \in B \cdot v$ and $g \cdot x \in B
\cdot x$. In particular $X_\mu^U$ is irreducible as the closure of $(G
\cdot x)^U = B \cdot x$.

Therefore $X_\mu^U$ is an irreducible $B$-stable curve meeting $Y$ and
$Z$ in the two $B$-fixed points $y$ and $z$ with $y \in Y$ and $z \in
Z$. It is therefore of type ($\chi$). Following \cite[Folgerung
  1.6]{pauer}, we prove that $X_\mu^U$ is smooth and therefore
isomorphic to $\p^1$. We only need to check the smoothness in $y$ and
$z$. We check this in $y$, the proof for $z$ works the same way. Since
$X_\mu$ is a spherical variety, by the structure theorem
(Theorem \ref{thm-stru}) we may replace $X_\mu$ by an affine spherical variety
$W$ containing $y$.  Note that we have a dense $U$-orbit in $W$. Let
$U^-$ be the opposite unipotent subgroup, we prove that the
restriction $\kk[W]^{U^-} \to \kk[W^U]$ is an isomorphism. Indeed the
restriction morphism $f : \kk[W] \to \kk[W^U]$ is surjective and
$B$-equivariant. In particular $f(\kk[W]^{U^-}) = f(T \cdot
\kk[W]^{U^-}) = T \cdot f(\kk[W]^{U^-}) = B \cdot f(\kk[W]^{U^-}) = f(B \cdot
\kk[W]^{U^-})$. Writing $\kk[W]$ as a sum of finite dimensional
$G$-module, we have that $\kk[W]$ is generated by $B \cdot \kk[W]^{U^-}$
thus $f(\kk[W]^{U^-}) = f(\kk[W]) = \kk[W^U]$. Now since $W = G \cdot
(W^U) = \overline{U^-B} \cdot W^U = \overline{U^-B \cdot W^U} = \overline{U^- \cdot W^U}$, the restriction map $\kk[W]^{U^-} \to \kk[W^U]$ is injective.

This implies $W^U = \spec(\kk[W^U]) \simeq \spec(\kk[W]^{U^-}) =
W/\!/U^-$. The last quotient is normal so $W^U$ is normal of dimension $1$
thus smooth.

Now let $C$ be another curve of type ($\chi$) meeting $Y$ and $Z$. Then $G \cdot C$ is a complete spherical $G$-variety of rank one and contains $Y$ and $Z$. It has to be $X_\mu$ and since $C$ is fixed by $U$, we get $C \subset X_\mu^U$ and this is an equality since $X_\mu^U$ is an irreducible curve.

Finally, let $\chi$ be the character through which $B$ acts on $C_{YZ}$ with fixed points $y \in Y$ and $z \in Z$. For $\delta$ a Cartier divisor, let $\delta_y$ and $\delta_z$ be the weights of $T$ acting on the fibers of the corresponding line bundle at $y$ and $z$. Then for $l_\delta = (l_Y)$, we have $\delta_y - \delta_z = l_Y - l_Z$ and $\chi = v$ where $v$ is the primitive vector vanishing on the wall $\mu$. Since $C_{YZ} \cdot \delta = (\delta_y - \delta_z)/\chi$ (see Lemma \ref{lemm-num-equiv-courbes}) we get $C_{YZ} \cdot \delta = C_\mu \cdot \delta$ and $C_\mu$ is the class of $C_{YZ}$ in $N_1(X)$.

3. By 1 above, the cone $NE(X)$ is generated by $B$-stable irreducible curves and their images generate $\overline{NE}(X)$ in $\overline{N_1(X)}$. Since the latter is generated by the classes $\overline{C}_{D,Y}$, we get that the classes $\overline{C}_{D,Y}$ are represented by $B$-stable irreducible curves. These curves must be of type (U), (T) or (N) since they cannot be of type ($\chi$): a curve of type ($\chi$) meets two closed $B$-orbits and thus by 2. its class is equal to some $C_\mu$.
\end{proof}

\begin{cor}
  Let $X$ be a spherical variety, then any irreducible complete $B$-stable curve on $X$ is isomorphic to $\p^1$.
\end{cor}

\begin{proof} 
See Proposition \ref{prop-dual} for type ($\chi$) and Proposition \ref{prop-unt} for other types.
\end{proof}

\subsection{Curves modulo rational and numerical equivalences}

Recall that $Z_1(X)$ is the group of one-dimensional cycles and that $A_1(X)$ is the quotient of $Z_1(X)$ modulo rational equivalence. A general result of Fulton, MacPherson, Sottile and Sturmfels \cite{FMSS} ensures that $B$-stable curves generate the Chow group of curves and furthermore that any rational equivalence is given by $[\div(f)] = 0$ for $f \in \kk(W)^{(B)}$ a $B$-semiinvariant rational function on a $B$-stable rational surface $W$.

Note that the pairing $(C,D) \mapsto \deg(D\vert_C)$ between $Z_1(X)$ and the group of Cartier divisors on $X$ induces a map $\a_X : \pic(X) \to A_1(X)^\vee$. As in $N_1(X)$, we define $AE(X)$ to be the cone in $A_1(X)_\Q = A_1(X) \otimes_\Z \Q$ generated by classes of effective curves.

\begin{prop}
For $X$  complete and spherical $\a_X : \pic(X) \to A_1(X)^\vee$ is an isomorphism.
\end{prop}

\begin{proof}
    For $X$ smooth projective, since $T$ acts with finitely many fixed points, $X$ has a cellular decomposition (see \cite{BB}). Therefore $\pic(X) \to H^2(X,\Z)$ and $A_1(X) \to H_2(X,\Z)$ are isomorphisms and the result follows from Poincar\'e duality.

In general, we can find a smooth projective toroidal $G$-variety $\Xt$ and a birational $G$-morphism $\varphi : \Xt \to X$ (use \cite[Corollary 3.3.8]{survey} and the equivariant Chow-Lemma). The map $\varphi_*$ is surjective over $\Q$ and we consider the commutative diagram
$$\xymatrix{\pic(X) \ar[r] \ar[d] & A_1(X)^\vee \ar[d]\\
  \pic(\Xt) \ar[r] & A_1(\Xt)^\vee, \\}$$
with vertical maps $\varphi^*$ and the transpose $(\varphi_*)^T$ of $\varphi_*$. Then $(\varphi_*)^T$ is injective and since $\a_\Xt$ is an isomorphism, we get that $\a_X$ is injective.

Let $u \in A_1(X)^\vee$ and let $\deltat \in \pic(\Xt)$ such that $\a_\Xt(\deltat) = (\varphi_*)^T(u)$. Then $\deltat$ vanishes on the curves contracted by $\varphi$. We prove that this implies $\deltat \in \varphi^*\pic(X)$. This would imply the surjectivity of $\a_X$ since $(\varphi_*)^T$ is injective.

We are left to prove $\deltat \in \varphi^*\pic(X)$. If $X$ is toroidal, then the fan $\mathfrak{F}(\Xt)$ of $\Xt$ is a subdivision of the fan $\mathfrak{F}(X)$ of $X$ (see \cite[Section 7]{jacopo}). For any wall of $\mathfrak{F}(\Xt)$ which is not a wall $\mu$ of $\mathfrak{F}(X)$, there is a curve $C_\mu$ of type $(\chi)$ contracted by $\varphi$ (see Proposition \ref{prop-dual}). By Theorem \ref{thm-pic-s-t} the embedding $\varphi^*\pic(X) \subset \pic(\Xt)$ is defined by the equations $\deltat \cdot C_\mu = 0$ proving the claim.
For $X$ spherical there exists a unique minimal pair $(X',\pi)$ with $X'$ toroidal and $\pi : X' \to X$ birational. For $Y$ a closed orbit in $X$, there exists a unique closed orbit $Y'$ in $X'$ with $\pi(Y') = Y$ and $\pi$ restricts to a proper morphism $\pi_Y : X'_{Y',G} \to X_{Y,G}$. Choose a prime $B$-stable divisor $D'$ in $X'$ with $Y' \not \subset D'$ but $Y \subset \pi(D')=D$. Let $C' \subset X'$ be a $B$-stable curve spanning the ray of $\overline{C}_{D',Y'}$. Then since $D = \pi(D') \supset Y$ we get $\overline{C}_{D,Y} = 0$ so $C'$ is contracted by $\pi$. By Theorem \ref{thm-pic-s-t} the embedding $\pi^*\pic(X) \subset \pic(X')$ is defined by the equations $\delta' \cdot C_{D,Y} = 0$ for $D$ and $Y$ as above. The result follows from this and the fact that $\varphi$ factors through $\pi$.
\end{proof}

Denote by $\overline{A_1(X)}$ the quotient of $A_1(X)$ by the classes of curves of type ($\chi$).

\begin{cor}
  For $X$ spherical complete, we have an exact sequence
  $$0 \to A_1(X)_{tors} \to A_1(X) \to N_1(X) \to 0.$$
  This induces an isomorphism $\overline{A_1(X)} \to \overline{N_1(X)}$ and both are torsion free.
\end{cor}

\begin{proof}
The isomorphism $\a_X : \pic(X) \to A_1(X)^\vee$ factors through
$$\pic(X) \to N^1(X) \to N_1(X)^\vee \to A_1(X)^\vee.$$
In particular, we recover the fact that the map $\pic(X) \to N^1(X)$ is an isomorphism and the transpose of $A_1(X) \to N_1(X)$ is also an isomorphism. The kernel of this map is therefore the torsion subgroup. We will describe all rational equivalences in the next section. This will in particular prove that $\overline{A_1(X)}$ is torsion free.
\end{proof}

\subsection{Equivalences}

In this subsection, we describe linear and numerical equivalences between $B$-stable curves.

We start with a general result on action of tori. Let $C$ be a complete curve endowed with a non trivial action of a torus $T'$. Let $\pi : \Ct \to C$ be its normalisation. The curve $C$ is rational so $\Ct \simeq \p^1$. Let $0$ and $\infty$ be the $T'$-fixed points in $\Ct$ and $x$ and $y$ their images in $C$. We may identify $C \setminus \{x,y\} = \Ct \setminus\{0,\infty\}$ with $\kk^\times$ so that $T'$ acts via a character $\chi$. 

\begin{lemma}
  \label{lemm-num-equiv-courbes}
  Let $L$ be a $T'$-linearised line bundle on $C$, and let $L_x$ and $L_y$ be the weights of $T'$ on the stalks at $x$ and $y$.

  Then $L_x - L_y$ is a multiple of $\chi$ and the degree $(L \cdot C)$ of $\pi^*L$ is given by
  $$(L \cdot C) = (L_x - L_y)/\chi.$$
\end{lemma}

\begin{proof}
Since $L_x$ and $L_y$ are also the weights of $\pi^*L$ at $0$ and $\infty$, we may assume that $C = \p^1$, $x = 0$ and $y = \infty$. Furthermore, the action of $T'$ comes from an action on $\kk^2$ with weights $\chi$ and $0$. Since the difference $L_x - L_y$ does not depend on the choice of a linearisation, we may assume that with the above action $L = \cO_{\p^1}(n)$. In that case $(L \cdot C) = n$ and $L_0 = 0$ while $L_\infty = n \chi$ proving the result.
\end{proof}

\begin{prop}
  Let $X$ be a normal $G$-variety and $C$, $C'$ be $B$-stable curves.

  1. If $C$ and $C'$ are of type {\rm ($\chi$)} and meet the same connected components of $X^B$, then they are numerically equivalent.

  2. If $C$ and $C'$ are of type {\rm (U)}, meet the same connected component of $X^B$ and $P(C) = P(C')$, then they are numerically equivalent.
\end{prop}

\begin{proof}
  1. The group $B$ acts via a character so any maximal torus $T$ acts non trivially on $C$ and $C'$. For any line bundle $L$, there exists an integer $n$ so that $L^{\otimes n}$ is $T$-linearised. Note also that for a linearised line bundle $L$, the weights of $B$ on the stalks of $L$ at $B$-fixed points are constant on connected components of $X^B$. The assertion now follows from the previous lemma.

  2. Set $P = P(C) = P(C')$ be the minimal parabolic associated to $C$ and $C'$. Let $\a$ be the corresponding root (the only simple root whose opposite is also a root of $P$). Let $s_\a$ be the corresponding simple reflection and $n_\a$ be a representative in $N_G(T)$. Denote by $x$ and $x'$ the $B$-fixed points in $C$ and $C'$. Then $n_\a \cdot x$ and $n_\a \cdot x'$ are the other $T$-fixed points in $C$ and $C'$ and $T$ acts on $C$ and $C'$ via the character $\a$. Since $L_{n_\a \cdot x} = s_\a(L_x)$, the above lemma yields the equalities
  $$(L \cdot C) = (L_x - L_{n_\a \cdot x})/\a = (L_x - s_\a(L_x))/\a = \scal{\a^\vee,L_x}.$$
  Since $L_x = L_{x'}$ because $x$ and $x'$ are in the same connected component of $X^B$, the result follows.
\end{proof}

\begin{prop}
    \label{prop-rel1}
    Let $X$ be a projective $G$-variety, let $C$ be a curve of type {\rm ($\chi$)} with $B$-fixed points $x$ and $y$ and let $P$ be a minimal parabolic subgroup.

    Then in $A_1(X)$, the difference $[P \cdot x] - [P \cdot y]$ is equal to a multiple of $[C]$.

    \vskip 0.1 cm

 Note that if $P$ fixes $x$ (resp. $y$), then $P \cdot x$ (resp. $P \cdot y$) is a point and its class is trivial in $A_1(X)$ if not then $P \cdot x$ and $P \cdot y$ are curves of type {\rm (U)}.
\end{prop}

\begin{proof}
  Let $S = P/R(P)$ where $R(P)$ is the radical of $P$. Since $U$ acts trivially on $C$ so does $R(P)$ and we have an $S$-action on $C$ such that $P \cdot C = S \cdot C$. Replacing $X$ with $S \cdot C$, we may assume that $G$ has semisimple rank one.

  Let $\Ct$ be the normalisation of $C$ and consider $W = G \times^B \Ct$ which is a smooth projective ruled surface over $G/B \simeq \p^1$. Write $p : W \to G/B$ and $\pi : W \to P \cdot C$ for the canonical morphisms. We may identify $\Ct$ with the fiber of $p$ over $B/B$. The curves $W_x = G \times^B x$ and $W_y = G \times^B y$ are $B$-stable sections of $p$ both of type (U). The morphism $\pi$ maps $W_x$ (resp. $W_y$) onto $P \cdot x$ (resp $P \cdot y$). Now by \cite[Section V.2]{hartshorne-book}, there exists an integer $n$ with $[W_x] - [W_y] = n [C]$. Mapping to $P \cdot C$ via $\pi_*$, we get the result.
\end{proof}

\begin{remark}
  1. In \cite[Proposition 2.4]{brion-mori} a more precise relation is given: if $\chi$ is the character acting on $C$ and $\a$ is the simple root associated to $P$, then
  $$[P \cdot x] - [P \cdot y] = \scal{\chi,\a^\vee}[C].$$

  2. Any rational equivalence coming from a surface $W \subset X$ which is an $\SL_2(\kk)/H$-embedding with $H$ containing a maximal unipotent subgroup comes from one of the above relations. Indeed $\SL_2(\kk)/H$ contains the curve $B/H$ whose closure is of type ($\chi$).
\end{remark}

From the geometry of curves of type (T) and (N) we deduce the following.

\begin{prop}
  \label{prop-rel2}
  Let $X$ be a complete $G$-spherical variety and let $C$ be a $B$-stable irreducible curve that is not of type {\rm ($\chi$)}. Let $P$ be the associated minimal parabolic subgroup and let $x \in C$ be the unique $B$-fixed point.

  1. If $C$ is of type {\rm (N)}, then $2[C] = [P \cdot x]$ in $A_1(X)$ and $P \cdot x$ is of type {\rm (U)}.

  2. If $C$ is of type {\rm (T)}, then there exists a unique $B$-stable curve $C'$ of type {\rm (T)} contained in $P \cdot C$ such that $[C] + [C'] = [P \cdot x]$ in $A_1(X)$ and $P \cdot x$ is of type {\rm (U)}.
\end{prop} 

\begin{proof}
These relations are true in $P \cdot C$ which is isomorphic to $\p^2$ in case 1. (resp. $\p^1 \times \p^1$ in case 2.) with $P \cdot x$ isomorphic to a conic (resp. the diagonal $\p^1$) and $C$ isomorphic to a line tangent to the conic (resp. to $\p^1 \times \{0\}$). In case 2., the curve $C'$ is $\{0\} \times \p^1$.
\end{proof}

We now describe all possible rational equivalences between $B$-stable curves.

\begin{prop}
    If $X$ is toroidal, any equivalence relation between $B$-stable curve is obtained from Proposition \ref{prop-rel1} or Proposition \ref{prop-rel2} or involves only curves of type {\rm ($\chi$)} or is of the form
    $$C - C' = \gamma$$
    where $\gamma$ is a linear combinaison of curves of type {\rm ($\chi$)} and $C$ and $C'$ satisfy one of the following cases:
    \begin{itemize}
    \item[1.] $C$ and $C'$ are disjoint of type {\rm (T)} and $P(C) = P(C')$;
    \item[2.] $C$ and $C'$ are of type {\rm (T)} and $P(C) \neq P(C')$;
    \item[3.] $C$ and $C'$ are of type {\rm (U)} and $P(C)$ and $P(C')$ correspond to orthogonal simple root.
    \end{itemize}  
\end{prop}

\begin{proof}
  We first use a general result of Fulton, MacPherson, Sottile and Sturmfels \cite{FMSS} which ensures that $B$-stable curves generate the Chow group of curves (we already proved this) and furthermore that any rational equivalence is given by $[\div(f)] = 0$
  for $f \in \kk(W)^{(B)}$ a $B$-semiinvariant rational function on a $B$-stable rational surface $W$. We therefore consider a $B$-stable rational surface $W$ in $X$. Since $X$ is spherical, $W$ has a dense $B$-orbit $B \cdot w$. Since $\overline{G \cdot w}$ is spherical and toroidal, we may replace $X$ with $\overline{G \cdot w}$ and assume that $G \cdot w$ is dense in $X$. Let $H = B_w$ be its stabiliser.

  If a maximal torus of $B$ is contained in $H$, then $U$ has a dense orbit in $W$ so that $\kk(W)^{(B)} = \kk$ and the relation $[\div(f)] = 0$ is trivial.
  
  If $U$ is contained in $H$, then $U$ acts trivially on a dense subset of $W$ and therefore on $W$. In particular $U$ acts trivally on any $B$-stable curve in $W$ so that any $B$-stable curve in $W$ is of type ($\chi$). We obtain relations as above with no left hand term.

  We are left with the situation where $T_H = H \cap T$ and $U_H = H \cap U$ are both codimension one subgroups of $T$ and $U$ respectively. We have $H = T_H U_H$ and $T_H$ is the kernel of a character $\chi : T \to \G_m$ uniquely determined up to sign. Note that the group $\kk(W)^{(B)}/\kk$ is the character group of $T/T_H$ and therefore of rank one so that $W$ gives rise to a unique rational equivalence.

  Since $U_H$ contains the group $(U,U)$, we have at the Lie algebra level
  $${\rm Lie}(U_H) = V \oplus \bigoplus_{\beta \in R^+ \setminus S} \g_\beta,$$
  where $V$ is a $T_H$-stable codimension one subspace of $\oplus_{\a \in S}\g_a$ where $S$ is the set of simple roots.

{\bf Case 1}: $V$ is $T$-stable. Then 
$${\rm Lie}(U_H) = \bigoplus_{\beta \in R^+ \setminus \{\a\}} \g_\beta$$
for some simple root $\a$. If $P$ is the minimal parabolic subgroup corresponding to $\a$, then $R_u(P) \subset U_H \subset H$ therefore $R_u(P)$ acts trivially on $W$ so $P \cdot W = L \cdot W$ with $L$ a Levi subgroup of $P$ and any Borel subgroup of $L$ has a dense ofbit in $L \cdot W$. If $H$ contains the connected center of $L$, then $L$ acts via $L/R(L)$ which is isomorphic $\SL_2(\kk)$ of $\PGL_2(\kk)$ and the equivalence comes from Proposition \ref{prop-rel1} or Proposition \ref{prop-rel2}.

If $H$ does not contain the connected center of $L$, then $\chi$ is not a multiple of the simple root $\a$. We prove that there exists a set $I$ of simple roots orthogonal to $\a$ such that
$${\rm Lie}(G_w) = {\rm Lie}(T_H) \oplus \left(\bigoplus_{\beta \in R^+ \setminus R_J} \g_\beta \right)\oplus \left(\bigoplus_{\beta \in R_I}\g_\beta \right),$$
with $J = I \cup \{\a\}$. Indeed, choose a finite dimensional $G$-module $E$ and a point $e \in E$ such that $G_w$ is the stabiliser of the line $\kk e$. Decompose $E$ as a direct sum of simple $G$-modules $V_\lambda$ with highest weight $\lambda$. The set of $R_u(P)$-fixed points in each $V_\lambda$ is a simple $L$-module with highest weight $\lambda$. These weights are $\lambda,\lambda - \a, \cdots, \lambda - \scal{\a^\vee,\lambda}\a$ all with multiplicity one. Since $\chi$ is not proportional to $\a$, each eigenspace of $R_u(P)T_H$ in $V_\lambda$ is one-dimensional and $T$-stable. Therefore $T$ normalises $H$ and the claim follows as in the proof of Proposition \ref{prop-GC-TN}.

Note that $R_u(G_w) = R_u(P_J)$ thus $R_u(P_J) \subset G_w \subset P_J$. Since $X$ is toroidal, so is $W$ and the map $G \times^{P_J} \overline{P_J \cdot w} \to X$ is an isomorphism. In particular $\overline{P_J \cdot w} = \overline{L \cdot x} = L \cdot W$ is a toroidal $L$-variety. We may therefore assume that $G$ has semisimple rank one. In that case $H \subset T$ is the kernel of $\chi$.

Since $H \subset B$, we have an isomorphism $G \times^B \overline{B \cdot x} \to \overline{G \cdot x} = X$. Thus $W = \overline{B \cdot x}$ is normal and $X = G \times^B W = X$. Note that $T$ has a dense orbit in $B/H$ thus of dimension $2$ and the rank of $G/H$ is at least $2$. Since $U^-H/H$ has codimension at most $2$, the variety $G/H$ is spherical of rank $2$ thus $X$ is spherical of rank two. We consider codimension one $G$-orbits in $X$ with a non trivial $U$-action. Such a codimension one $G$-orbit corresponds to an invariant valuation $\nu \in \cV$ and by \cite[Theorem 3.6]{brion-pauer} its valuation cone is obtained from $\cV$ via projection from $\nu$. Note that a $G$-spherical variety with a trivial $U$-action is horospherical and therefore has a linear cone of valuations (cf. \cite[Corollary 11.7]{jacopo} or \cite{knop}). In particular, we see that there are exactly two codimension one $G$-orbits with a non trivial $U$-action obtained via projection from non-linear edges of the two dimensional valuation cone $\cV$ (since $X$ is toroidal complete of rank $2$, its cone is simplicial). In particular $X$ contains exactly two $L$-stable codimension one subvarieties with a non trivial $U$-action. These varieties have an open orbit isomorphic to $L/T$. Therefore in $W$, the complement of $B \cdot w$ contains exactly two curves $C$ and $C'$ with open orbit isomorphic to $B/T$. As in the proof of Proposition \ref{prop-unt} this implies that both curves are of type (T). The others curve are acted on trivially by $U$ so are of type ($\chi$).

We now prove that there is a relation of the form $C - C' = \gamma$ with $\gamma$ a linear combinaison of curves of type ($\chi$). The surface $W$ contains $B/H \simeq U \times T/\ker \chi$ as a $T$-stable open subset. Its generic isotropy group is $\ker\chi \cap \ker \a$. Moding out, we may assume that this intersection is trivial so that $\chi$ and $\a$ form a basis of the character group of $T$. Consider $\chi$ as a rational function on $W$ with zeros and poles outside $B/H$. Since $C$ contains $B/T \simeq U$ as $T$-stable open subset, the generic isotropic group of $T$ in $C$ is $\ker\a$. The same holds for $C'$. In particular, the fan of the toric surface $W$ contains two opposite half lines corresponding to $C$ and $C'$ and the values at $\chi$ of the associated cocharacters are $1$ and $-1$. We get a relation of the first form.

{\bf Case 2}: $V$ is not $T$-stable. Then there exists distinct simple roots $\a$ and $\a'$ and a positive integer $n$ such that $n \chi = \a - \a'$ and we have
$${\rm Lie}(U_H) = \ell \oplus \left(\bigoplus_{\beta \in R^+ \setminus \{\a,\a'\}} \g_\beta \right),$$
where $\ell$ is a line in $\g_\a \oplus \g_{\a'}$. Let $P$ be the parabolic subgroup of semisimple rank two associated to $\{\a,\a'\}$ and let $L$ be its Levi subgroup containing $T$. Then $H \supset R(P)$ thus $\overline{L \cdot w} = L \cdot W$ is acted on trivially by $R(P)$. Let $\lambda$ be a cocharacter with $\scal{\lambda,\a} = 0 = \scal{\lambda,\a'}$ and $\scal{\lambda,\beta} < 0$ for the other simple roots $\beta$. Then $C_G(\lambda(\G_m)) = L$ and $w \in X^{\lambda(\G_m)}$. Moreover, the weights of $\lambda$ in ${\rm Lie}(G)/{\rm Lie}(G_w)$ are non-negative therefore $\overline{L \cdot w} = X^\lambda$ the ``sink'' of the Bia\l ynicki-Birula decomposition (see Subsection \ref{bb-decomp}). The $L$-variety $\overline{L \cdot w}$ is spherical and toroidal by Proposition \ref{prop-luna}. Since the connected center of $L$ acts trivially we may assume that $G$ is semisimple of rank $2$. If the isotropy group $G_w$ is contained in $B$, then we may argue as in case 1 to get a relation of the second form. If the isotropy group $G_w$ is not contained in $B$, then one checks that $G/G_w$ is isomorphic to the quotient of $\SL_2(\kk) \times \SL_2(\kk)$ or $\PGL_2(\kk) \times \PGL_2(\kk)$ by its diagonal subgroup. We get a relation of the third form.
\end{proof}

  \begin{cor}
On a toroidal variety, any rational equivalence between $B$-stable curves is given in Propositions \ref{prop-rel1} and \ref{prop-rel2} above or is of the form
  $$n C - n' C' = \gamma$$
with $n,n' \in \{0,1\}$, where $C$ and $C'$ are of type {\rm (U)}, {\rm (T)} or {\rm (N)} and $\gamma$ is a linear combinaison of  curves of type {\rm ($\chi$)}.
  \end{cor}

\begin{cor}
  \label{cor-edges}
  Assume that $X$ is toroidal.

  In the quotient $\overline{N_1(X)}$ of $N_1(X)$ by curves of type {\rm ($\chi$)}, the class of any curve of type {\rm (T)} or {\rm (N)} spans a ray of the cone $\overline{NE(X)}$.

  A curve $C$ of type {\rm (U)} with associated minimal parabolic subgroup $P$ not spanning a ray of the cone $\overline{NE(X)}$ is contained in a $P$-surface with trivial $R(P)$-action which is isomorphic to the $P/TR(P)$-embedding $\p^1 \times \p^1$.
\end{cor}

\begin{remark}
  In \cite[Corollaire 3.8]{brion-mori} the curves of type (U) were claimed to span edges of the cone $\overline{NE(X)}$. This is clearly not the case for the $\SL_2(\kk)/T_0$ embedding $\p^1 \times \p^1$ but the above says that this is essentially the only possibility.
  
  More precisely, the $\SL_2(\kk)/T_0$ embedding $\p^1 \times \p^1$ has no curve of type ($\chi$), two curves $C$ and $C'$ of type (T) given by $\{0\} \times \p^1$ and $\p^1 \times \{0\}$ and a curve $C''$ of type (U) given by the diagonal embedding of $\p^1$. We obviouly have $[C''] = [C] + [C']$ in $\overline{N_1(X)} = N_1(X) = A_1(X)$ so that $C''$ is not spanning an edge of the cone $NE(X)$ (which is spanned by $[C]$ and $[C']$).
\end{remark}

\subsection{Big cells}

In this subsection we present general results of Luna \cite{luna-aus} and use them to compare the type of curves and the corresponding dual divisors. We prove the following result in the next two paragraphs.

\begin{thm}
  \label{thm-CD-NUT}
  Let $X$ be a complete $G$-spherical variety, let $P$ be a minimal parabolic subgroup of $G$ and let $D \in \odelta(X)$.

  If $P$ raises $D$ with an edge of type {\rm (U)}, {\rm (T)} resp. {\rm (N)} then there exists a $B$-stable curve $C$ of type {\rm (U)}, {\rm (T)} resp. {\rm (N)} whose class in $\overline{N_1(X)}$ is $\overline{C}_{D,Y}$.
\end{thm}

\begin{remark}
The above result is probably well known to the specialists but we could not find any reference for it. Some special cases especially for sober and wonderful embeddings are given in \cite{luna-aus}. It seems to appear for the first time in this form.
\end{remark}

\subsubsection{Bia\l ynicky-Birula decomposition}
\label{bb-decomp}
Recall an easy version of the Bia\l ynicki-Birula decomposition \cite{BB}, see also \cite[Section 3.1]{brion-equiv-CH}. Let $X$ be a variety with an action of a one-dimensional torus $T_0 \simeq \G_m$. For $Y \subset X^{T_0}$ a subvariety, define $X_Y = \{x \in X \ | \ \lim_{t \to 0} t \cdot x \textrm{ exists and lies in } Y \}$. Then $X_Y$ is a locally closed $T_0$-stable subvariety. Furthermore, let $\gamma_Y : X_Y \to Y$ be the map defined by $x \mapsto \lim_{t \to 0} t\cdot x$. Then $\pi_Y$ is a $T$-equivariant morphism. Note that for $X$ complete, for any $x \in X$, the limit $\lim_{t \to 0} t \cdot x$ exists and $X$ is the disjoint union of locally closed varieties $X_Y$ where $Y$ runs over all components of $X^{T_0}$. 

We present a result due to Luna (see \cite{luna-aus} and also \cite[Proof of Proposition 7.1]{brion-equiv-CH}) on the Bia\l ynicki-Birula decomposition for spherical varieties. Let $P$ be a minimal parabolic subgroup and let $\lambda : \G_m \to T$ be a cocharacter of the maximal torus such that $\scal{\lambda,\beta} > 0$ for the simple roots $\beta$ with $U_\beta \subset R(P)$ and $\scal{\lambda,\a} = 0$ for the last simple root $\a$. Such a cocharacter is called \textbf{adapted} to $P$. Let $T_0 = \lambda(\G_m)$ and consider the associated decomposition $X = \bigcup X_Y$ indexed by the connected components $Y$ of $X^{T_0}$. It is a finite decomposition and since the $X_Y$ are locally closed, there exists a $Y_0$ such that $X_{Y_0}$ is open in $X$, the big cell. Set $X_\lambda = X_{Y_0}$, $X^\lambda = Y_0$ and write $\gamma_\lambda : X_\lambda \to X^\lambda$ for the map $\gamma_Y$.

Denote by $P^-$ the minimal parabolic opposite to $P$ and by $R_u(P^-)$ its radical. Note that $P = \{g \in G \ | \ \lim_{t \to 0} \lambda(t)g\lambda(t)^{-1} \textrm{ exists} \}$ and that all the limits take value in $G^\lambda = C_G(T_0) = P \cap P^-$ which is the Levi subgroup $L$ of $P$ associated to $T$. Denote again by $\gamma_\lambda : P \to G^\lambda$ the map $g \mapsto \lim_{t \to 0} \lambda(t) g \lambda(t)^{-1}$. One easily checks that $X_\lambda$ is $P$-stable and $\gamma_\lambda(p \cdot x) = \gamma_\lambda(p) \cdot \gamma_\lambda(x)$ for $x \in X_\lambda$ and $p \in P$.

\begin{prop}
  \label{prop-luna}
With the above notation we have:

1. $R_u(P^-)$ acts trivially on $X^\lambda$.

2. $X^\lambda$ is $L$-spherical.

3. If $X$ is toroidal, then so is $X^\lambda$. 
\end{prop}

\begin{proof}
  1. We have $P^- = \{g \in G \ | \ \lim_{t \to 0} \lambda(t)^{-1}g\lambda(t) \textrm{ exists} \}$ and $R_u(P^-) = \{g \in G \ | \ \lim_{t \to 0} \lambda(t)^{-1}g\lambda(t) = 1 \}$. In particular for $g \in R_u(P^-)$ and $x \in X^\lambda$, we have $\lim_{t \to 0} \lambda(t)^{-1} g \cdot x = \lim_{t \to 0} (\lambda(t)^{-1} g \lambda(t)) \lambda(t)^{-1} \cdot x = \lim_{t \to 0} (\lambda(t)^{-1} g \lambda(t)) \cdot x = x$. So the closure $C$ of $\{ \lambda(t)^{-1} g \cdot x \ | \ t \in \G_m \}$ contains two points in $X^\lambda$ thus two $T_0$-fixed points thus the all of $C$ is fixed and equal to $\{x\}$. We get $g \cdot x = x$.

  2. We first prove that $X^\lambda$ is normal. Since $X_\lambda$ is covered by affine $T_0$-stable open subsets $V$, we only need to prove that $V \cap X^\lambda$ is normal. The action of $T_0$ induces a grading $\kk[V] = \oplus_{n \geq 0} \kk[V]_n$ (with $n \geq 0$ since $V \subset X_\lambda$). But $X$ being normal, so is the open subset $V$. The algebra $\kk[V]$ is therefore normal and so is $\kk[V]_0 = \kk[V \cap X^\lambda]$.

Now as remarked right before the proposition, the map $\gamma_\lambda$ is $L$-equivariant. The group $B_L = B \cap L$ is a Borel subgroup of $L$ and $\gamma_L(B) = B_L$. The dense $B$-orbit in $X$ is contained in $X_\lambda$ and its image is a dense $B_L$-orbit in $X^\lambda$.  

3. Let $D'$ be a $B^\lambda$-stable prime divisor in $X^\lambda$ and let $D = \overline{\gamma_\lambda^{-1}(D')}$. This is a $B$-stable divisor with $D \cap X^\lambda = D'$. If $D'$ contains a closed $L$-orbit $Y'$, then $D$ contains $Y'$. Let $y$ be the $(B^\lambda)^-$-fixed point in $Y'$, then $y$ is also $R_u(P^-)$-fixed thus $y$ is $B^-$-fixed. But $D$ being $B$-stable it contains $B \cdot y = BB^- \cdot y$ which is dense in $G \cdot y$ so $D$ contains a $G$-orbit and is $G$-stable. It follows that $D'$ is $L$-stable.
\end{proof}

Before proving Theorem \ref{thm-CD-NUT} we apply the above to improve Proposition \ref{prop-unt}.

\begin{prop}
\label{prop-typeT-p1*p1}
For $C$ of type {\rm (T)} and $P = P(C)$, the map $P \times^B C \to P \cdot C$ is an isomorphism and $P \cdot C$ is isomorphic to $\p^1 \times \p^1$.
\end{prop}

\begin{proof}
  We proved in Proposition \ref{prop-unt} that $P \times^B C$ is the normalisation of $P \cdot C$ and isomorphic to $\p^1 \times \p^1$ so we only need to prove that $P \cdot C$ is normal.

  The $G$-variety $G \cdot C$ is spherical thus we may replace $X$ with $G \cdot C$. Let $\lambda$ be a cocharacter adapted to $P^-$ the minimal parabolic subgroup opposite to $P$. Then $X^\lambda$ is a spherical $L$-variety. Let $z$ be the $T$-fixed non $B$-fixed point in $C$. Then $G \cdot z$ is dense in $X$ and $R_u(P)$ acts trivially on $z$. We get that $P^-L \cdot z$ is dense in $X$ and since $X_\lambda$ is $P^-$-stable, the orbit $L \cdot z$ meets $X_\lambda$. By \cite[Theorem A]{richardson}, we have that $L \cdot z$ is an irreducible component of $(G \cdot z)^{\lambda(\G_m)}$ therefore $X^\lambda = \overline{L \cdot z} = P \cdot C$. Now $X^\lambda$ is a spherical variety by the above proposition therefore normal and we are done.
\end{proof}

\subsubsection{Proof of Theorem \ref{thm-CD-NUT}} Let $X$ be a complete $G$-spherical variety and let $D \in \odelta(X)$. Note that there always exists a minimal parabolic subgroup $P$ raising $D$ to $X$. Let $\lambda : \G_m \to T$ be a cocharacter adapted to $P$.
Let $T_0 = \lambda(\G_m)$ and apply Proposition \ref{prop-luna} to the action of $T_0$ on $X$.

By Proposition \ref{prop-luna}, the variety $X^\lambda$ has a trivial action of $R_u(P^-)$ and is $S$-spherical for $S = P/R_u(P^-)$. Let $B_S$, $T_S$ and $U_S$ be $B/R_u(P^-)$, $T/(T \cap R_u(P^-))$ and $U/R_u(P^-)$. Denote by $B_S^-$ the Borel subgroup opposite to $B_S$ relative to $T_S$ and $L$ the Levi subgroup of $P$ relative to $T$.

Note that $P \cdot D = X$ meets $X_\lambda$ and since this last subset is $L$-stable, $D_\lambda = D \cap X_\lambda$ is not empty. We get that $D^\lambda = D \cap X^\lambda = \gamma_\lambda(D_\lambda)$ is not empty.

If $D^\lambda = X^\lambda$, then it contains a $B_S^-$-fixed point $y$ and since $R_u(P^-)$ acts trivially on $X^\lambda$ it fixes $y$ thus $y$ is $B^-$-fixed. In particular $BB^- \cdot y = B \cdot y \subset D$ and thus $G \cdot y \subset D$, a contradiction to $D \subset \odelta(X)$. Therefore $D^\lambda$ is a $B_S$-stable divisor of $X^\lambda$ containing no $S$-orbit (by the same argument).

Since $S$ has semisimple rank one and $X^\lambda$ is an embedding of $S/H$ for some spherical subgroup $H$ of $S$, we have three possible cases (see Example \ref{exam-sl2}): either $H = T_S$ or $H = N_S$ the normaliser of $T_S$ or $H$ contains $U_S$.

If $H = T_S$, then $D^\lambda$ is a $B_S$-stable curve of type (T). Since the square
$$\xymatrix{S \times^{B_S} D_\lambda \ar[r] \ar[d] & S \cdot D_\lambda \ar[d] \\
S \times^{B_S} D^\lambda \ar[r] & S \cdot D^\lambda }$$
is a cartesian product, the divisor $D$ is raised by $P$ with an edge of type (T). Furthermore, we have a $B_S^-$-stable curve $C$ of type (T) meeting $D^\lambda$ transversaly in one point. This curve is also $R_u(P^-)$-stable and therefore $B^-$-stable. By Corollary \ref{cor-edges}, the class of this curve generates an edge of $\overline{NE(X)}$. Since $C \cdot D = 1$ the class of $C$ must coincide with $C_{D,Y}$.
If $H = N_S$, then the same argument works as well.
Assume that $H$ contains $U_S$, then $X^\lambda$ contains a closed $S$-orbit $C$ isomorphic to $S/B$. This curve is $B_S^-$-stable and therefore also $B^-$-stable. It is of type (U) and meets $D^\lambda$ in one point. The above argument will work as soon as the class of $C$ spans an edge of $\overline{NE(X)}$. If this is not the case, then there exists a $P^-$-surface $W$ with a trivial $R_u(P^-)$-action and isomorphic to the $S/T_S$-embedding $\p^1 \times \p^1$. But then $W \subset X^{\lambda(\G_m)}$ and $W$ meets $X^\lambda$ along $C$. Therefore $W$ and $X^\lambda$ would be in the same connected component of $X^{\lambda(\G_m)}$ which has to be $X^\lambda$, a contradiction. 

  \section{The canonical divisor}
\label{sec-can}
  
Let $X$ be a spherical variety. Since $X$ is normal, we can define its canonical sheaf $\omega_X$ by $\omega_X = j_*\omega_{X^{sm}}$ where $j : X^{sm} \to X$ is the inclusion of the smooth locus. The sheaf $\omega_X$ is isomorphic to $\cO_X(K_X)$ for some Weil-divisor $K_X$. Any such Weil-divisor is called \textbf{ a canonical divisor of $X$}. In this section we describe an explicit $B$-stable canonical divisor.

\subsection{Computing $K_X$} We give an algorithm for computing a canonical divisor.

\begin{prop}
  There exists a canonical divisor $K_X$ of $X$ such that
$$-K_X = \sum_{D\in \D(X) \setminus \Delta(X)} D + \sum_{D\in \Delta(X)} a_D D$$
with $a_D$ a non negative integer.
\end{prop}

\begin{proof}
Replacing $X$ by the union of $G$-orbits of codimension at most one,
we may assume that $X$ is toroidal and smooth.
By the Structure Theorem for toroidal varieties (Theorem \ref{str-toro}), we have an isomorphism
$X\setminus\Delta_X\simeq(P_X)_u\times Z$ with
$Z$ a toric variety for a quotient of $L/[L,L]$. Let $T_0$ be the kernel of the action of $T$ on $Z$ and let $T_1$ be a subtorus of $T$ such that $T = T_0T_1$ and $T_0 \cap T_1$ is finite. The group $R_u(P)T_1 \subset B$ has an open orbit in $X \setminus \Delta_X$ and its generic isotropy subgroup is finite.

Let $\gp_X$ and $\gt_1$ be the Lie algebras of $P_X$ and $T_1$ and let $\sigma$ be the exterior product of the elements in a basis of $\gp_X \oplus \gt_1$. Then $\sigma$ is a global section of the anticanonical sheaf $\omega_X^\vee$ and the support of the vanishing locus of sigma is the complement of the open $B$-orbit in $X$. This implies that we have a canonical divisor of the form
$$K_X = - \sum_{D \in \D(X)} a_D D$$
 with $a_D \geq 0$. We are left to prove $a_D = 1$ for $D \not\in \Delta(X)$. We need to prove that $\sigma$ vanishes at order one for $D \not \in \Delta(X)$. This can be checked on $X \setminus \Delta_X \simeq R_u(P) \times Z$ and follows from the corresponding fact for toric varieties: the exterior power of a basis of $\gt_1$ vanishes at order one on the codimension one $T_1$-orbits in $Z$.
\end{proof}

\begin{thm}
  Let $X$ be a complete toroidal spherical variety. An anticanonical divisor of $X$ is 
  $$-K_X = \sum_{D \in \D(X)} a_D D$$
with 
  $$a_D = \left\{
  \begin{array}{cl}
  1 & \textrm{if } D \in \D(X) \setminus \Delta(X), \\
  1 & \textrm{if } D \in \Delta(X) \textrm{ is raised with an edge of type {\rm (T)} or {\rm (N)},} \\
  2 \scal{\a^\vee,\rho - \rho_I} & \textrm{if } D \in \Delta(X) \textrm{ is raised with an edge of type {\rm (U)},} \\
  \end{array}\right.$$
where $\alpha$ is the simple root associated to $P(C)$ with $C$ a $B$-stable curve of type {\rm (U)} representing $\overline{C}_{D,Y}$, where $G \cdot C = G/P_I$ for some set of simple roots $I$ and $\rho$ (resp. $\rho_I$) is the half-sum of positive roots in the root system $R$ (resp. $R_I$).
\end{thm}

\begin{proof}
  By the above Proposition, there exists such an anticanonical divisor and $a_D = 1$ for $D \in \D(X) \setminus \Delta(X)$. Note that since $X$ is toroidal, we have $\Delta(X) = \odelta(X)$. Let $D \in \odelta(X)$ and let $C$ be a $B$-stable curve with class equal to $\overline{C}_{D,Y}$ in $\overline{N_1(X)}$. Set $\omt_X = \omega_X(\sum_{D \in \D(X) \setminus \Delta(X)}D)$. We have $\omt_X = \cO_X(-\sum_{D \in \Delta(X)}a_D D)$ and $a_D = - \deg\omt_X\vert_C$.

  Let $Y$ be an irreducible $G$-stable subvariety of $X$. Then by Corollary \ref{cor-sub-toro}, $Y$ is smooth toroidal and a complete intersection of $G$-stable divisors. Let $\mathcal{A}(Y)$ be the set of these divisors. The $G$-stable divisors of $Y$ are given by $\D(Y) \setminus \Delta(Y) = \{ D\cap Y \ | \ D \in \D(X) \setminus (\Delta(X) \cup \mathcal{A}(Y))$. Since for $D \in \D(X) \setminus (\Delta(X) \cup \mathcal{A}(Y))$, we have $\cO_Y(D \cap Y) = \cO_X(D)\vert_Y$ we obtain by the adjunction formula the equality
  $$\begin{array}{lll}
    \omt_Y & = \omega_Y \otimes \cO_Y\left(\sum_{D' \in \D(Y) \setminus \Delta(Y)} D'\right) & \\
    & =
    \omega_X\left(\sum_{D \in \mathcal{A}(Y)}D\right) \otimes \cO_X
    \left(\sum_{D \in \D(X) \setminus (\Delta(X) \cup \mathcal{A}(Y)} D\right)\vert_X
    & = \omt_X\vert_Y.\\
  \end{array}
  $$
  Applying this to $Y = G \cdot C$ we get $a_D = - \deg\omt_{G \cdot C}\vert_C$. If $C$ has type (U) then there exists a parabolic subgroup $I$ such that $G \cdot C \simeq G/P_I$. The variety $G \cdot C$ has no $G$-stable divisor so $\omt_{G \cdot C} = \omega_{G \cdot C}$ and the result follows. If $C$ has type (T) or (N), then by Proposition \ref{prop-GC-TN}, we have $G \cdot C \simeq G \times^{Q(C)} P(C) \cdot C$. Since $C$ is identified with a fiber of the map $G \times^{Q(C)} P(C) \cdot C \to G/Q(C)$, we get that $\deg\omt_{G \cdot C}\vert_C = \deg\omt_{P(C) \cdot C}\vert_C$ where $P(C) \cdot C$ is seen as a toroidal $L$-variety with $L$ a Levi subgroup of $P(C)$. In type (T), we have $P(C) \cdot C \simeq \p^1 \times \p^1$ thus $\omt_{P(C) \cdot C}$ is isomorphic to $\cO(-2,-2) \otimes \cO(1,1)$ while $C$ is $\{0\} \times \p^1$. We get $a_D = 1$. In type (N), we have $P(C) \cdot C \simeq \p^2$ thus $\omt_{P(C) \cdot C}$ is isomorphic to $\cO(-3) \otimes \cO(2)$ while $C$ is a line. We get $a_D = 1$. 
\end{proof}

\begin{remark}
  \label{rem-can}
1. The above result gives an algorithm to compute a $B$-stable canonical divisor for any spherical variety. Indeed, let $X$ be a spherical variety, then embbed $X$ in $\overline{X}$ a complete spherical variety and let $\pi : \Xt \to \overline{X}$ be a toroidal resolution (such a resolution always exists, see for example \cite[Corollary 3.3.8]{survey}). Let $K_\Xt$ be a canonical divisor, since any spherical variety has rational singularities (see \cite[Corollary 2.3.4]{survey}), then $K_{\overline{X}} = \pi_*K_\Xt$ and $K_X =  K_{\overline{X}}\vert_X = \pi_*K_\Xt\vert_X$. 

2. The above result has an especially nice form for simple toroidal varieties since we have a unique closed orbit. We recover this way a result of Luna \cite{luna-aus}.
\end{remark}

\begin{example}
  In this example, we use the algorithm described in Remark \ref{rem-can} to compute the canonical divisor of the variety $X = \IG(2,2n+1)$ of isotropic line in $\CC^{2n+1}$ endowed with an antisymmetric bilinear form $\omega$ of maximal rank.

  Let $K$ be the kernel of the the form, it has dimension one and $\CC^{2n+1} = K \oplus V$ with $V$ of dimension $2n$ with a symplectic form $\omega\vert_V$ (thus non-degenerate). Fix $(e_i)_{i \in [1,2n]}$ a basis of $\CC^{2n}$ such that the form is given by $\omega(e_i,e_j) = \delta_{i,2n+1-i}$.

  The variety $X$ is spherical for the action of $G = \Sp(V,\omega\vert_V)$ and has three orbits $Y$, $Z$ and $\U = X \setminus (Y \cup Z)$. The first two are closed and defined as follows:
  $$Y = \{ \ell \in X \ | \ \ell \subset V \} \simeq \IG(2,2n) \textrm{ and } Z = \{ \ell \in X \ | \ K \subset \ell \} \simeq \p^{2n-1}.$$
  We have $\pic(X) = \Z$ and two $B$-stable divisors defined, with $\ell_2 = \scal{e_1,e_2}$, by
  $$D_Y = \overline{\{ \ell \in \U \ | \ \ell\cap V \subset e_1^\perp \}} \textrm{ and } D_Z = \overline{\{ \ell \in \U \ | \ V \cap (\ell+K) \cap \ell_2^\perp \neq 0 \}}.$$
A toroidal resolution $\varphi : \Xt \to X$, is obtained by blowing-up $Y$ and $Z$ in $X$:
  $$\Xt = \{ (p,\ell,\pi) \in \IG(1,2n+1) \times X \times \IG(3,2n+1) \ | \ p \subset \ell \subset \pi, p \subset V \textrm{ and } K \subset \pi\}.$$
  This variety has three $G$-orbits $\U$ and the exceptional divisors $E_Y$ and $E_Z$ over $Y$ and $Z$ which are both isomorphic to the incidence variety of isotropic points and lines in $V$. Write $\Dt_Y$ and $\Dt_Z$ for the strict transforms of $D_Y$ and $D_Z$ in $\Xt$. It is easy to check that the $B$-stable curves dual to $\Dt_Y$ and $\Dt_Z$ are
  $$\begin{array}{l}
  C_Y = \{(p,\ell,\pi) \in \Xt \ | \ p \subset e_1^\perp, l = \ell_2, \pi = \ell_2 + K \} \textrm{ and } \\
  C_Z = \{(p,\ell,\pi) \in \Xt \ | \ V \cap \pi \cap \ell_2^\perp \neq 0, p = e_1, \ell = K + e_1 \}.
  \end{array}$$
  We get $G \cdot C_Y = E_Y \simeq E_Z = G \cdot C_Z$ and by the above algorithm:
  $$-K_\Xt = E_Y + E_Z + (2n - 2) \Dt_Y + 2 \Dt_Z \textrm{ and } -K_X = (2n-2)D_Y + 2 D_Z.$$
Since $D_Y$ and $D_Z$ are equivalent to the same ample generator $H$ of $\pic(X)$, we get
$$-K_X = 2n H.$$
We refer to \cite{pasquier}, \cite{boris} or \cite{gpps} for more details on the geometry of these varieties and on the geometry of smooth Fano horospherical varieties of Picard rank one.
\end{example}

\addtocontents{toc}{\protect\setcounter{tocdepth}{2}}

\providecommand{\bysame}{\leavevmode\hbox to3em{\hrulefill}\thinspace}
\providecommand{\MR}{\relax\ifhmode\unskip\space\fi MR }
\providecommand{\MRhref}[2]{%
  \href{http://www.ams.org/mathscinet-getitem?mr=#1}{#2}
}
\providecommand{\href}[2]{#2}


\begin{thebibliography}{10}

\bibitem{piotr} Achinger, P., Perrin, N., \textit{Spherical multiple flags}. Advances studies in pure math. {\bf 71} (2016) 53--74.











\bibitem{bender} Bender, M., Perrin, N., \textit{Singularities of closures of spherical B-conjugacy classes of nilpotent orbits}. Preprint arXiv:1412.5654.
  
\bibitem{BB} Bia\l ynicki-Birula, A., \textit{Some theorems on actions of algebraic groups}. Ann. of Math. (2) {\bf 98} (1973), 480--497.




\bibitem{borelb} Borel, A., \textit{Linear algebraic groups}. Second edition. Graduate
Texts in Mathematics, 126. Springer-Verlag, New York, 1991.










\bibitem{brion-nombre-car} Brion, M., 
\textit{Groupe de Picard et nombres caract\'eristiques des vari\'et\'es sph\'eriques}. Duke Math. J. {\bf 58} (1989), no. 2, 397--424.




\bibitem{brion-mori} \bysame 
\textit{Vari\'et\'es sph\'eriques et th\'eorie de Mori}. Duke Math. J. {\bf 72} (1993), no. 2, 369--404.


\bibitem{brion-aus} \bysame, 
  \textit{Curves and divisors in spherical varieties}.  Algebraic groups and Lie groups,  21--34, Austral. Math. Soc. Lect. Ser., 9, Cambridge Univ. Press, Cambridge, 1997.

\bibitem{brion-equiv-CH} \bysame,
\textit{Equivariant Chow groups for torus actions}. Transform. Groups {\bf 2} (1997), no. 3, 225--267.  


\bibitem{brion-multi-free} \bysame, 
\textit{Multiplicity-free subvarieties of flag varieties}. Commutative algebra (Grenoble/Lyon, 2001), 13--23, Contemp. Math., {\bf 331}, Amer. Math. Soc., Providence, RI, 2003.






\bibitem{brion-sumihiro} \bysame,
\textit{Algebraic group actions on normal varieties}. Preprint arXiv:1703.09506.




\bibitem{brion-pauer} Brion, M., Pauer, F., \textit{Valuations des espaces homog\`enes sph\'eriques}. Comment. Math. Helv. {\bf 62} (1987), no. 2, 265--285.













\bibitem{demazure} Demazure, M., \textit{Sous-groupes alg\'ebriques de rang maximum du groupe de Cremona}. Ann. Sci. ENS {\bf 3} (1970), 507--588.
  


\bibitem{EW} Ewald, G., Wessels, U., \textit{On the ampleness of invertible sheaves in complete projective toric varieties}, Results Math. {\bf 19} (1991), no. 3-4, 275--278.
  

\bibitem{fulton-toric} Fulton, W., \textit{Introduction to toric varieties}. Annals of Mathematics Studies, 131. The William H. Roever Lectures in Geometry. Princeton University Press, Princeton, NJ, 1993.

\bibitem{fulton} \bysame, 
\textit{Intersection theory}. Second edition. Ergebnisse der Mathematik und ihrer Grenzgebiete. 3. Folge. A Series of Modern Surveys in Mathematics, 2. Springer-Verlag, Berlin, 1998.


\bibitem{FMSS} Fulton, W., MacPherson, R., Sottile, F., Sturmfels, B., \textit{Intersection theory on spherical varieties}. J. Algebraic Geom. {\bf 4} (1995),  no. 1, 181--193.


\bibitem{jacopo} Gandini, J. \textit{Sanya lectures: Embeddings of spherical homogeneous spaces}, this volume.

  \bibitem{GP} Gandini, J., Pezzini, G., \textit{Orbits of strongly solvable spherical subgroups on the flag variety}. Preprint  arXiv:1411.5818.

    \bibitem{gpps} Gonzales R., Pech C., Perrin N., Samokhin, A. 
  \textit{Geometry of rational curves on horospherical varieties of Picard rank one}. In preparation.


\bibitem{Gro1} Grosshans, F.D., \textit{The invariants of unipotent radicals of parabolic subgroups}. Invent. Math. {\bf 73} (1983), 1--9.

\bibitem{Gro2} Grosshans, F.D \textit{Algebraic Homogeneous Spaces and Invariant Theory}, Lecture Notes in Mathematics 1673, Springer-Verlag, Berlin, 1997.
  





\bibitem{hartshorne-book} Hartshorne, R., \textit{Algebraic geometry}. Graduate Texts in Mathematics, No. 52. Springer-Verlag, New York-Heidelberg, 1977.





















\bibitem{knop} Knop, F., 
\textit{The Luna-Vust theory of spherical embeddings}. Proceedings of the Hyderabad Conference on Algebraic Groups (Hyderabad, 1989), 225--249, Manoj Prakashan, Madras, 1991.





\bibitem{K2} \bysame, 
\textit{On the set of orbits for a Borel subgroup}. Comment. Math. Helv. {\bf 70} (1995), no. 2, 285--309.








\bibitem{Lazarsfeld} Lazarsfeld, R., \textit{Positivity in algebraic geometry. I. Classical setting: line bundles and linear series}. Ergebnisse der Mathematik und ihrer Grenzgebiete. 3. Folge. A Series of Modern Surveys in Mathematics, 48. Springer-Verlag, Berlin, 2004.











\bibitem{luna-aus} Luna, D., 
  \textit{Grosses cellules pour les varie\'ete\'es sphe\'eriques}. Algebraic groups and Lie groups,  267--280, Austral. Math. Soc. Lect. Ser., 9, Cambridge Univ. Press, Cambridge, 1997.


\bibitem{LV} Luna, D., Vust, T., \textit{Plongements d'espaces homog{\`e}nes}. Comment. Math. Helv. {\bf 58} (1983), no. 2, 186--245.







\bibitem{MJ} Moser-Jauslin, L.,
  \textit{The Chow rings of smooth complete $\SL(2)$-embeddings}. Compositio Math. {\bf 82} (1992), no. 1, 67--106.
  





\bibitem{pauer} Pauer, F., \textit{\"Uber gewisse $G$-stabile Teilmengen in projektiven R\"aumen}. Manusc. Math. {\bf 66} (1989), 1--16.
  


\bibitem{pasquier} Pasquier, B., 
\textit{On some smooth projective two-orbit varieties with Picard number 1}.  Math. Ann. {\bf 344} (2009),  no. 4, 963--987. 



\bibitem{boris} Pasquier, B., Perrin, N., \textit{Local rigidity of
  quasi-regular varieties}. Math. Z. {\bf 265} (2010),  no. 3,
  589--600. 






\bibitem{survey} Perrin N.,
  \textit{On the geometry of spherical varieties}. 
  Trans. Groups. {\bf 19} (2014), no.1, 171--223.







\bibitem{ressayre-minimal}
  Ressayre, N., 
\textit{Spherical homogeneous spaces of minimal rank.} Adv. Math. {\bf 224} (2010), no. 5, 1784--1800.

\bibitem{richardson} Richardson, R.W. \textit{On orbits of algebraic groups and Lie groups}. Bull. Austral. Math. Soc. {\bf 25} (1982), no. 1, 1--28.

\bibitem{RS1} Richardson, R.W., Springer, T.A.,
  \textit{The Bruhat order on symmetric varieties}, Geom. Dedicata {\bf 35} (1990), 389--436.

\bibitem{RS2} \bysame, 
\textit{Complements to:
"The Bruhat order on symmetric varieties''}. Geom. Dedicata {\bf 49}
(1994), no. 2, 231--238.




\bibitem{rosentlicht} Rosenlicht, M., \textit{A remark on quotient spaces}.  An. Acad. Brasil. Ci. {\bf 35}  1963 487--489.





\bibitem{serre} Serre, J.-P., \textit{Espaces fibr\'es alg\'ebriques}. S\'eminaire Claude Chevalley, 3 (1958), Expos\'e No. 1, 37 p.





\bibitem{sumihiro} Sumihiro, H., \textit{Equivariant completion}.  J. Math. Kyoto Univ. {\bf 14}  (1974), 1--28.

\bibitem{sumihiro2} \bysame, 
  \textit{Equivariant completion. II}. J. Math. Kyoto Univ. {\bf 15} (1975), no. 3, 573--605.





\bibitem{timashev} Timash{\"e}v, D.A., 
\textit{Homogeneous spaces and equivariant embeddings}. Encyclopaedia
of Mathematical Sciences, 138. Invariant Theory and Algebraic
Transformation Groups, 8. Springer, Heidelberg, 2011.







\end{thebibliography}
\end{document}